\documentclass[11pt]{amsart}%
\usepackage{palatino, mathpazo}
\usepackage{amssymb}
\usepackage{amsfonts}
\usepackage{graphicx}
\usepackage{amsmath}
\usepackage[shortlabels,inline]{enumitem}
\usepackage{hyperref}
\usepackage{cite}
\usepackage{caption}
\usepackage{comment}%
\providecommand{\U}[1]{\protect\rule{.1in}{.1in}}
\hypersetup{
	colorlinks=true,
	linkcolor=blue,
	filecolor=black,
	urlcolor=black,
	citecolor=red,
}
\newtheorem{theorem}{Theorem}[section]

\newtheorem{conjecture}{Conjecture}
\newtheorem{corollary}[theorem]{Corollary}

\newtheorem{example}{Example}
\newtheorem{lemma}[theorem]{Lemma}

\newtheorem{proposition}[theorem]{Proposition}
\newtheorem{remark}[theorem]{Remark}

\numberwithin{equation}{section}
\makeatletter
\def\fps@figure{htbp}
\def\fnum@figure{\textbf{Fig. \thefigure}}
\makeatother
\begin{document}
\title[Componentwise Geometry and Monodromy]{Componentwise Geometry and Monodromy of Generalized Lam\'e Equations}
\author{Ting-Jung Kuo}
\address[Ting-Jung Kuo]{Department of mathematics, National Taiwan Normal University,
Taipei, 11677, Taiwan \& National Center for Theoretical Sciences, No.1 Sec.4
Roosevelt Rd., National Taiwan University, Taipei 10617, Taiwan.}
\email{tjkuo1215@ntnu.edu.tw, tjkuo1215@gmail.com}
\author{Xuanpu Liang}
\address[Xuanpu Liang]{Laboratory of Mathematics and Complex Systems (Ministry of
Education), School of Mathematical Sciences, Beijing Normal University,
Beijing 100875, China.}
\email{xuanpuliang@mail.bnu.edu.cn}

\begin{abstract}
We develop a componentwise geometric and monodromy theory for the
one-support generalized Lam\'e equation on an elliptic curve, with
singularities at \(0\) and \(\pm p\). Its log-free curve decomposes
canonically into irreducible even and non-even components, the former
being governed by elliptic Painlev\'e~VI.

For the non-even component, we construct the hyperelliptic spectral
curve, Baker--Akhiezer functions, and addition map, and prove that
\[
\deg \sigma_{n,p}^{(1)}=n(n+1).
\]
After quotienting by the involution \(T\mapsto -T\), we identify the
non-even spectral curve with the classical Lam\'e spectral curve of
weight \(n\), compatibly with the addition map and the rational function
\(\kappa\). This identification is
realized by
\[
\widetilde B=T^2-n(n+1)\wp(p),
\]
and associates every non-even generalized Lam\'e equation with a unique
classical Lam\'e equation on the same elliptic curve having equivalent
period monodromy.

Fixing \(\widetilde B\) yields an isomonodromic deformation with
\(\tau\) fixed. Together with the Painlev\'e-VI deformation on the even
component, it gives a componentwise interpretation of the collision
\(p\to0\), and yields a finite descent on the admissible
completely reducible locus. The classical spectral, finite-gap,
finite-monodromy, and curvature theories consequently transfer to the
non-even component.

Finally, we place the one-support family within the broader symmetric
setting $\left(n_0,n_1,n_2,n_3,\frac12,\frac12\right).$
The one-support family forms an affine genus-zero hierarchy. By
contrast, for $\left(1,1,0,0,\frac12,\frac12\right),$
the non-even normalization is generically elliptic and degenerates to
genus zero on the discriminant locus, while the arithmetic genus of the
full compactified log-free curve remains equal to two. This first genus
jump delineates the boundary of the present affine theory and motivates
a genus-dependent componentwise geometry beyond it.
\end{abstract}

\maketitle
\tableofcontents

\vspace{1em}

\noindent\textbf{Notation and conventions.}

Let
\[
\Lambda_{\tau}=\mathbb{Z}\oplus\mathbb{Z}\tau, \qquad E_{\tau}=\mathbb{C}%
/\Lambda_{\tau}, \qquad\tau\in\mathbb{H}.
\]
Write
\[
\omega_{0}=0,\qquad\omega_{1}=1,\qquad\omega_{2}=\tau,\qquad\omega_{3}%
=1+\tau,
\]
so that
\[
E_{\tau}[2] = \left\{  \frac{\omega_{k}}{2}:0\leqslant k\leqslant3 \right\}
\pmod{\Lambda_{\tau}}.
\]
We write
\[
\wp(z)=\wp(z;\tau),\qquad\zeta(z)=\zeta(z;\tau),\qquad\sigma(z)=\sigma
(z;\tau)
\]
when $\tau$ is fixed, and set
\[
e_{j}=\wp\left( \frac{\omega_{j}}{2}\right) , \qquad\eta_{j}=2\zeta\left(
\frac{\omega_{j}}{2}\right) , \qquad j=1,2,3.
\]

\section{Introduction}\label{Introduction}

\subsection{The symmetric deformation and the component problem}
The Darboux--Treibich--Verdier potential is
\begin{equation}
q_{\mathrm{DTV}}(z)=\sum_{k=0}^{3}n_k(n_k+1)
  \wp\left(z+\frac{\omega_k}{2}\right),
\qquad n_k\in\mathbb{Z}_{\geqslant 0}.
\end{equation}
Its one-support specialization is the classical Lam\'e potential;
its spectral geometry is part of the theory of elliptic finite-gap
equations and singular curvature equations
\cite{BBEIM,CLW,GW1,LW,Takemura1,TV}.

Within the global theory of generalized Lam\'e curves of
Chou--Wang--Wu (C--W--W) \cite{Chou-Wang-Wu-II}, we consider the symmetric data
\begin{equation}
\mathbf n=(n_0,n_1,n_2,n_3,\tfrac12,\tfrac12),
\qquad
\mathbf p=(0,\tfrac{\omega_1}{2},\tfrac{\omega_2}{2},
  \tfrac{\omega_3}{2},p,-p),
\end{equation}
where $p\in E_\tau\setminus E_\tau[2]$.
Adding one apparent pair $\pm p$ preserves the symmetry of the DTV
configuration. This family has a distinguished irreducible even
log-free component: its parameter curve is rational, and its
isomonodromic deformation is governed by elliptic Painlev\'e VI
\cite{Chen-Kuo-Lin-Hamiltonian,Chen-Kuo-Lin-Lame I,Chen-Kuo-Lin-Lame II,Chen-Kuo-Lin-Lame III}. The remaining problem is to determine the
non-even components and their monodromy.

The non-even parameter geometry varies with the integral weights.
One positive weight gives a rational component; the first
two-support example is generically elliptic. This genus variation
organizes the componentwise problem. Throughout, the genus of a
log-free component means the geometric genus of its normalization,
not the genus of its spectral cover. The global arithmetic genus
does not determine these genera, since it also contains singularity
and intersection defects.

We develop the spectral and monodromy theory of the one-support
genus-zero layer
\begin{equation}
\mathbf n=(n,0,0,0,\tfrac12,\tfrac12),\qquad n\in\mathbb Z_{>0},
\end{equation}
up to translation by a half-period. Consider
\begin{equation}
w''(z)=q_n(z;p,T_1,T_2,B)w(z),\label{GLEn}
\end{equation}
with
\begin{equation}
\begin{aligned}
q_n(z;p,T_1,T_2,B)
={}&n(n+1)\wp(z)+\frac34\bigl(\wp(z+p)+\wp(z-p)\bigr)\\
&+T_1\bigl(\zeta(z+p)-\zeta(z)\bigr)
 +T_2\bigl(\zeta(z-p)-\zeta(z)\bigr)+B.\label{potential T_1, T_2}
\end{aligned}
\end{equation}
Here $T_1,T_2,B\in\mathbb C$, and $B$ is determined by the log-free
conditions.

\begin{theorem}\label{Main Thm log free decomposition}
The log-free variety of \eqref{GLEn} decomposes as
\begin{equation}
V_{n,p}(\tau)=V^{(0)}_{n,p}(\tau)\cup V^{(1)}_{n,p}(\tau).\label{decomposition}
\end{equation}
where the two irreducible components are
\begin{equation}
V^{(0)}_{n,p}(\tau)
=\{(T_1,T_2)\in V_{n,p}(\tau)\mid T_1+T_2=0\},\label{apparent space, even, A}
\end{equation}
\begin{equation}
V^{(1)}_{n,p}(\tau)
=\left\{(T_1,T_2)\in V_{n,p}(\tau)\mid
 T_1-T_2+\frac{\wp''(p)}{2\wp'(p)}=0\right\}.\label{apparent space, noneven, T}
\end{equation}
They are respectively even and non-even, and meet at
\begin{equation}
\mathbf T_*=\left(-\frac{\wp''(p)}{4\wp'(p)},
                         \frac{\wp''(p)}{4\wp'(p)}\right).
\end{equation}
\end{theorem}

The explicit equations give affine coordinates
\[
(T_1,T_2)=(A,-A),\qquad
(T_1,T_2)=\left(T-\frac{\wp''(p)}{4\wp'(p)},
                    T+\frac{\wp''(p)}{4\wp'(p)}\right).
\]
Thus $V^{(0)}_{n,p}(\tau)\simeq\mathbb A^1_A$ and
$V^{(1)}_{n,p}(\tau)\simeq\mathbb A^1_T$.
We denote the corresponding equations by
$\mathrm{GLE}^{(0)}_n(p,A,\tau)$ and $\mathrm{GLE}^{(1)}_n(p,T,\tau)$.
The coordinate $T$ is the basis of the intrinsic non-even theory
constructed in Sections~\ref{Log-free variety}--\ref{Isomonodromic deformation}.

\subsection{Intrinsic spectral geometry and monodromy}
The decomposition lifts to the generalized Lam\'e curve
\[
Y_{n,p}(\tau)=Y^{(0)}_{n,p}(\tau)\cup Y^{(1)}_{n,p}(\tau).
\]
The addition map $\sigma_{n,p}:Y_{n,p}(\tau)\to E_\tau$ has
total degree $2n(n+1)+1$ \cite{Chou-Wang-Wu-II}.

\begin{theorem}\label{Main Thm 0}
The componentwise degrees are
\[
\deg\sigma^{(0)}_{n,p}=n(n+1)+1,
\qquad \deg\sigma^{(1)}_{n,p}=n(n+1).
\]
\end{theorem}

For the non-even component we construct the spectral curve,
Baker--Akhiezer functions, addition map and rational function
$\kappa$. Section~\ref{Section, Degree of the Non-Even Addition Map} computes its degree independently of the
global formula. A rigidity theorem rules out nontrivial
isomonodromy with varying $\tau$ and forces the minimal relation
between the addition map and $\kappa$ to be independent of $p$.
Degeneration then identifies this relation with the classical
Lam\'e relation.

Let $\Gamma^{(i)}_{n,p}(\tau)$ be the spectral curve identified
with $Y^{(i)}_{n,p}(\tau)$ by its zero-divisor embedding, and let
$\hat\Gamma^{(1)}_{n,p}(\tau)$ be the quotient by
$(T,C)\mapsto(-T,C)$. Write $\widetilde\Gamma_n(\tau)$ for the
spectral curve of the classical equation
\begin{equation}
w''(z)=\bigl(n(n+1)\wp(z)+\widetilde B\bigr)w(z),
\qquad \widetilde B\in\mathbb C,
\end{equation}
denoted by $\mathrm L_n(\widetilde B,\tau)$.

\begin{theorem}\label{thm, isomorphism}
For every $\tau\in\mathbb H$ and $p\in E_\tau\setminus E_\tau[2]$,
there is a natural isomorphism
\[
\iota_p:\hat\Gamma^{(1)}_{n,p}(\tau)
\xrightarrow{\sim}\widetilde\Gamma_n(\tau)
\]
compatible with the addition maps and primitive rational functions:
\[
\tilde\sigma_n\circ\iota_p=\hat\sigma^{(1)}_{n,p},
\quad \tilde\kappa\circ\iota_p=\kappa.
\]
\end{theorem}

On the nonsingular locus, both curves normalize the same model
$W_n(\sigma;\kappa)$ $=0$. The spectral-polynomial identity extends
holomorphically across the discriminant, so the explicit coordinate
translation and its compatibilities also hold on singular fibers
(Proposition~\ref{prop:extension-singular-fibers}).

\begin{theorem}\label{Main Thm 1}
Under this isomorphism, every $\mathrm{GLE}^{(1)}_n(p,T,\tau)$ corresponds
to a unique $\mathrm L_n(\widetilde B,\tau)$ on the same elliptic curve,
where
\begin{equation}
\widetilde B=T^2-n(n+1)\wp(p).\label{correspondence}
\end{equation}
The two equations have equivalent period monodromy.
\end{theorem}

The Lam\'e correspondence is obtained from the componentwise
monodromy theory. Its accessory-coordinate formula then yields
an invertible first-order Darboux-type operator.
The constructions of Sections ~\ref{Log-free variety}--\ref{Isomonodromic deformation}---the spectral curve, zero-divisor embedding, addition map, component degree, and monodromy rigidity---constitute the intrinsic theory. The Darboux operator of Section~\ref{Section, Mono Equiv} realizes, but does not generate, this structure in the one-support genus-zero setting. Beyond this setting, the direct Darboux construction generally leaves the prescribed one-pair family, while Proposition~\ref{prop:exceptional-spectral-obstruction} shows that even a rational parameter component need not admit a birational reduction to a pure DTV potential.

\subsection{Componentwise isomonodromy}
\begin{theorem}\label{Main Thm 2}
Fix $\tau\in\mathbb{H}$ and $\widetilde B\in\mathbb{C}$. On the two-sheeted
cover of $E_\tau\setminus E_\tau[2]$ defined by
\[
T(p)^2=n(n+1)\wp(p)+\widetilde B,\label{T(p)'s square}
\]
the family $\mathrm{GLE}^{(1)}_n(p,T(p),\tau)$ is isomonodromic. Its
two local branches have the same period monodromy.
\end{theorem}

The even Painlev\'e-VI deformation varies $\tau$ and reaches
Lam\'e weights $n\pm1$ as $p\to0$; the non-even deformation
fixes $\tau$ and preserves weight $n$. They account for the
three terms in the global degeneration formula of \cite{Chou-Wang-Wu-II}.

\begin{theorem}\label{thm, componentwise degeneration}
Under $p\to0$, the multiplicity-one limits of weights $n-1$ and
$n+1$ are realized by the even Painlev\'e-VI branches. The
multiplicity-two limit of weight $n$ is realized by the non-even
sheets $T(p)$ and $-T(p)$.
\end{theorem}

\begin{theorem}\label{thm, connectivity intro}
Fix $n\in\mathbb{Z}_{>0}$ and let $(r,s)\in\mathcal A_n$, with
$\mathcal A_n$ defined in \eqref{def, An}. Every completely reducible
$\mathrm{GLE}_n(p,\mathbf T,\tau)$ realizing $(r,s)$ is isomonodromically
connected to a classical Lam\'e equation of weight one on a suitable
elliptic curve through a finite chain of Painlev\'e-VI and non-even
isomonodromic deformations.
\end{theorem}

\subsection{The first positive-genus component}
For the next weight distribution, the non-even parameter geometry
changes, while the even component remains rational. 

\begin{theorem}
\label{thm, (1,1,0,0) non-even component}
Fix $\tau\in\mathbb{H}$. For generic $p\in E_\tau\setminus E_\tau[2]$,
the non-even log-free component for
$(1,1,0,0,\tfrac12,\tfrac12)$ has smooth projective model
\begin{equation}
W^2=t(t+1)P(t),
\end{equation}
where
\begin{equation}
\begin{aligned}
P(t)={}&\left(\wp(p)+\wp\left(p+\frac{\omega_1}{2}\right)+e_1\right)t^2
       +(2\wp(p)+e_1)t\\
&+e_1-\wp\left(p+\frac{\omega_1}{2}\right).
\end{aligned}
\end{equation}
Its genus is one. On the finite locus where the discriminant of $P$
vanishes, this model is nodal and its normalization is rational.
The full compactified log-free curve has arithmetic genus two
throughout.
\end{theorem}

For a reducible non-even locus, we take the sum of the genera of its normalizations. We conjecture that, for generic $(\tau,p)$, this sum vanishes exactly when at most one integral weight is positive. The one-support family has genus zero, whereas Theorem~\ref{thm, (1,1,0,0) non-even component} shows that the first two-support case is generically elliptic.

\subsection{Consequences}
\begin{corollary}
If $\widetilde Q_n(\widetilde B;\tau)$ is the monic Lam\'e
spectral polynomial, then
\begin{equation}
Q^{(1)}_{n,p}(T;\tau)
=\widetilde Q_n\bigl(T^2-n(n+1)\wp(p);\tau\bigr).
\end{equation}
\end{corollary}

Define the period discriminant and spectral set by
\[
\widetilde\Delta_n(\widetilde B)=\tfrac12\operatorname{tr}
 \widetilde M_1(\widetilde B),\qquad
\Sigma_n^L=\{\widetilde B\in\mathbb{C}\mid
 \widetilde\Delta_n(\widetilde B)\in[-1,1]\},
\]
and define $\Delta^{(1)}_{n,p}$ and $\Sigma^{(1)}_{n,p}$ analogously.
Theorem~1.4 gives
\[
\Delta^{(1)}_{n,p}(T)
=\widetilde\Delta_n\bigl(T^2-n(n+1)\wp(p)\bigr).
\]

\begin{corollary}
The spectral set $\Sigma^{(1)}_{n,p}$ has the finite-gap property
and is invariant under $T\mapsto-T$.
\end{corollary}

For integral $n$, finite monodromy of the algebraic Lam\'e
equation is dihedral. Let $L_n(N)$ be the number of its
scalar-equivalence classes with monodromy group $D_N$, as in the
Dahmen--Beukers conjecture \cite{S. Dahmen}, proved in \cite{Chen-Kuo-Lin-Dahmen,Chou-Wang-Wu-I}.

\begin{corollary}
A non-even isomonodromic family has finite monodromy if and only if
its associated classical Lam\'e equation has finite monodromy.
\end{corollary}

The algebraic group $D_N$ corresponds to the period group $C_N$
of the elliptic equation. Under
\[
(x,y,x_0,y_0,T)\mapsto
(\lambda^2x,\lambda^3y,\lambda^2x_0,\lambda^3y_0,\lambda T),
\]
where $(x,y)=(\wp(z),\wp'(z))$ and
$(x_0,y_0)=(\wp(p),\wp'(p))$, the quantity
$T^2-n(n+1)x_0$ has the Lam\'e accessory weight.

\begin{corollary}
For $n\in\mathbb{Z}_{>0}$ and $N\geqslant 3$, there are exactly $L_n(N)$
scalar-equivalence classes of non-even isomonodromic families of
one-integer weight $n$ with period monodromy group $C_N$.
\end{corollary}

The curvature equation
\begin{equation}
\Delta v+e^v=8\pi n\delta_0\quad\text{on }E_\tau
\label{PDE 8pin at 0}
\end{equation}
has a solution, equivalently an even family of solutions, precisely
when some $\mathrm L_n(\widetilde B,\tau)$ has unitary monodromy
\cite{CLW,LW-AnnMath,LW}. For
\begin{equation}
\Delta u+e^u=8\pi n\delta_0+4\pi(\delta_p+\delta_{-p})
\quad\text{on }E_\tau,
\label{PDE, 8pin+4pi+-p}
\end{equation}
the even and non-even families lie over $V^{(0)}_{n,p}(\tau)$
and $V^{(1)}_{n,p}(\tau)$. The even problem depends on $p$
through Painlev\'e VI \cite{Chen-Kuo-Lin-Painleve VI}; Theorem~\ref{Main Thm 1} gives
\[
\text{a non-even family of \eqref{PDE, 8pin+4pi+-p}}
\quad\Longleftrightarrow\quad
\text{an even family of \eqref{PDE 8pin at 0}}.
\]
Thus non-even families exist either for every admissible $p$ or for
none on a fixed $E_\tau$. The rectangular nonexistence result of
Chen--Lin \cite{Chen-Lin-sharp nonexistence} yields:

\begin{corollary}
For $\tau\in i\mathbb R_{>0}$ and $n\in\mathbb{N}$, equation
\eqref{PDE, 8pin+4pi+-p} has no non-even family of solutions for any
$p\in E_\tau\setminus E_\tau[2]$.
\end{corollary}

\medskip
\noindent\textit{Acknowledgements.}
Ting-Jung Kuo was supported by NSTC 113-2628-M-003-001-MY4 and
thanks the National Center for Theoretical Sciences for its
continued support.

\section{Geometry of the Log-Free Locus}

\label{Log-free variety}

Fix $n\in\mathbb{Z}_{>0}$ and $p\in E_{\tau}\setminus E_{\tau}[2]$, and
consider  \eqref{GLEn}$_{n,p}$:
\[
y^{\prime\prime}(z)=q_{n}(z;p,T_{1},T_{2},B)y(z),\quad z\in\mathbb{C},
\]
where the potential $q_{n}(z;p,T_{1},T_{2},B)$ is defined by \eqref{potential T_1, T_2}.

\begin{theorem}
\label{apparent,main thm} The generalized Lam\'{e} equation
\eqref{GLEn}$_{n,p}$ is log-free at each singularity $z=0,\pm p$ if and only
if the parameters $T_{1},T_{2}$ satisfy
\begin{equation}
(T_{1}+T_{2})\left(  T_{1}-T_{2}+\frac{\wp^{\prime\prime}(p)}{2\wp^{\prime
}(p)}\right)  =0, \label{AP1}%
\end{equation}
and $B$ is determined by%
\begin{equation}
B=\frac{T_{1}^{2}+T_{2}^{2}}{2}-\frac{T_{1}-T_{2}}{2}\zeta(2p)-\frac{3}{4}%
\wp(2p)-n(n+1)\wp(p). \label{AP2}%
\end{equation}

\end{theorem}

Applying the Frobenius method at $z=\pm p$ gives the following necessary conditions.

\begin{lemma}
\label{apparent,necessary condition}The generalized Lam\'{e} equation
\eqref{GLEn}$_{n,p}$ is log-free at $z=\pm p$ whenever the parameters
$T_{1},T_{2}$ and $B$ satisfy (\ref{AP1}) and (\ref{AP2}).
\end{lemma}

\begin{proof}
The local exponents of \eqref{GLEn}$_{n,p}$ at $z=\pm p$ are $-1/2$ and $3/2$.
Suppose \eqref{GLEn}$_{n,p}$ is log-free at $\pm p$, there exist local
solutions of the form
\begin{equation}
\label{1024equ1}%
\begin{split}
y_{p}(z)  &  =\sum_{j=0}^{\infty}c_{j}(z-p)^{j-\frac{1}{2}},\quad
c_{0}=1,\quad c_{j}\in\mathbb{C},\quad j\in\mathbb{Z}_{>0},\\
y_{-p}(z)  &  =\sum_{j=0}^{\infty}d_{j}(z+p)^{j-\frac{1}{2}},\quad
d_{0}=1,\quad d_{j}\in\mathbb{C},\quad j\in\mathbb{Z}_{>0}.
\end{split}
\end{equation}
Near $z=\pm p$ the potential admits Laurent expansions
\[%
\begin{split}
q_{n}(z;p,T_{1},T_{2},B)  &  =\sum_{i=-2}^{\infty}q_{i,p}(z-p)^{i}%
\quad\text{near}\quad z=p,\\
q_{n}(z;p,T_{1},T_{2},B)  &  =\sum_{i=-2}^{\infty}q_{i,-p}(z+p)^{i}%
\quad\text{near}\quad z=-p,
\end{split}
\]
with coefficients
\begin{equation}
\label{1024equ6}%
\begin{split}
q_{-2,p}  &  =q_{-2,-p}=\frac{\,3\,}{4},\quad q_{-1,p}=T_{2},\quad
q_{-1,-p}=T_{1},\\
q_{0,p}  &  =B-\left(  T_{1}+T_{2}\right)  \zeta(p)+T_{1}\,\zeta
(2p)+n(n+1)\wp(p)+\frac{\,3\,}{4}\wp(2p),\\
q_{0,-p}  &  =B+\left(  T_{1}+T_{2}\right)  \zeta(p)-T_{2}\,\zeta
(2p)+n(n+1)\wp(p)+\frac{\,3\,}{4}\wp(2p).
\end{split}
\end{equation}
Substituting \eqref{1024equ1} into \eqref{GLEn}$_{n,p}$ and equating
coefficients of $(z\pm p)^{j-\frac{5}{2}}$ gives, for each $j\in
\mathbb{Z}_{\geq0}$,
\begin{equation}
\label{1024equ3}%
\begin{split}
c_{j}\left(  j-\frac{1}{2}\right)  \left(  j-\frac{3}{2}\right)   &
=\sum_{i=-2}^{j-2}q_{i,p}\,c_{j-2-i}\,,\\
d_{j}\left(  j-\frac{1}{2}\right)  \left(  j-\frac{3}{2}\right)   &
=\sum_{i=-2}^{j-2}q_{i,-p}\,d_{j-2-i}\,.
\end{split}
\end{equation}
Using $q_{-2,\pm p}=\tfrac34$, the relations \eqref{1024equ3} reduce to
\begin{equation}
\label{1024equ4}%
\begin{split}
j\left(  j-2\right)  c_{j}  &  =\sum_{i=-1}^{j-2}q_{i,p}\,c_{j-2-i}\,,\quad
j\in\mathbb{Z}_{>0},\\
j\left(  j-2\right)  d_{j}  &  =\sum_{i=-1}^{j-2}q_{i,-p}\,d_{j-2-i}\,,\quad
j\in\mathbb{Z}_{>0}.
\end{split}
\end{equation}
Taking $j=1,2$ in \eqref{1024equ4} yields the algebraic constraints
\begin{equation}
\label{1024equ5}%
\begin{cases}
~q_{-1,p}\,T_{2}-q_{0,p}=0,\\
~q_{-1,-p}\,T_{1}-q_{0,-p}=0.
\end{cases}
\end{equation}
Substituting \eqref{1024equ6} into \eqref{1024equ5} yields
\[%
\begin{cases}
~\left(  T_{1}+T_{2}\right)  \left(  T_{1}-T_{2}+\dfrac{\wp^{\prime\prime}
(p)}{\,2\wp^{\prime}(p)\,}\right)  =0,\\[4pt]%
~B=\dfrac{1}{2}\left(  T_{1}^{2}+T_{2}^{2}\right)  -\dfrac{1}{2}\left(
T_{1}-T_{2}\right)  \,\zeta(2p)-\dfrac{3}{\,4\,}\, \wp(2p)-n(n+1)\,\wp(p).
\end{cases}
\]

\end{proof}

Define
\begin{equation}
V_{n,p}(\tau):=\left\{  \mathbf{T}=(T_{1},T_{2})\in\mathbb{C}^{2}\text{ such
that (\ref{AP1}) holds true}\right\}  .\label{apparent space, whole}%
\end{equation}
Then
\begin{equation}
V_{n,p}(\tau)=V_{n,p}^{(0)}(\tau)\cup V_{n,p}^{(1)}(\tau
)\label{Apparent space}%
\end{equation}
where $V_{n,p}^{(0)}(\tau)$ and $V_{n,p}^{(1)}(\tau)$ are defined by
\eqref{apparent space, even, A} and \eqref{apparent space, noneven, T}, respectively.

\textbf{Case I (Even component): }Let $\mathbf{T}$ $\in$ $V_{n,p}^{(0)}(\tau
)$, and parameterize the accessory parameters as
\[
\mathbf{T}=\left(  A,-A\right)  ,\text{ }A\in\mathbb{C}.
\]
The three-singularity GLE (\ref{GLEn}) reduces to GLE$_{n}^{(0)}(p,A,\tau)$:
\begin{equation}
y^{\prime\prime}(z)=q_{n}^{(0)}(z;p,A)y(z),\text{ }z\in E_{\tau},
\label{GLE 2, even}%
\end{equation}
where%
\begin{equation}
q_{n}^{(0)}(z;p,A)=\left[
\begin{array}
[c]{l}%
n(n+1)\wp(z)+\frac{3}{4}(\wp(z+p)+\wp(z-p))\\
+A(\zeta(z+p)-\zeta(z-p))+B
\end{array}
\right]  \label{potential, even}%
\end{equation}
with the accessory parameter $B$ determined by
\begin{equation}
B=A^{2}-\zeta(2p)A-\frac{3}{4}\wp(2p)-n(n+1)\wp(p). \label{B(A),even}%
\end{equation}

\textbf{Case II (Non-even component): }Let $\mathbf{T}\in$ $V_{n,p}^{(1)}%
(\tau)$, and parameterize the accessory parameters as
\[
\mathbf{T}=\left(  T-\frac{\wp^{\prime\prime}(p)}{4\wp^{\prime}(p)}%
,T+\frac{\wp^{\prime\prime}(p)}{4\wp^{\prime}(p)}\right)  ,\text{ }%
T\in\mathbb{C}.
\]
In this setting, the three-singularity GLE (\ref{GLEn}) becomes GLE$_{n}%
^{(1)}(p,T,\tau)$:%
\begin{equation}
y^{\prime\prime}(z)=q_{n}^{(1)}(z;p,T)y(z),\text{ }z\in E_{\tau},
\label{GLE 3, noneven}%
\end{equation}
where the non-even potential takes the form%
\begin{equation}
q_{n}^{(1)}(z;p,T)=\left[
\begin{array}
[c]{l}%
n(n+1)\wp(z)+\frac{3}{4}(\wp(z+p)+\wp(z-p))\\
+T(\zeta(z+p)+\zeta(z-p)-2\zeta(z))\\
-\frac{\wp^{\prime\prime}(p)}{4\wp^{\prime}(p)}(\zeta(z+p)-\zeta(z-p))+B
\end{array}
\right]  , \label{potential, noneven}%
\end{equation}
where the corresponding value of $B$ is given by%
\begin{equation}
B=T^{2}+\frac{\wp^{\prime\prime}(p)}{2\wp^{\prime}(p)}\zeta(p)-\frac{1}%
{2}(2n^{2}+2n-3)\wp(p). \label{B(T),noneven}%
\end{equation}
The potential $q_{n}^{(1)}$ is even precisely when $T=0$, which is the unique
intersection point of $V_{n,p}^{(0)}(\tau)$ and $V_{n,p}^{(1)}(\tau)$.

The Frobenius argument at $z=0$ gives the following result
\cite{Chen-Kuo-Lin-Lame I}.

\begin{lemma}
[\cite{Chen-Kuo-Lin-Lame I}]\label{lemma apparent,T1+T2=0} Let $\mathbf{T}%
=(A,-A)\in V_{n,p}^{(0)}(\tau),$ and define $B$ by \eqref{B(A),even}. Then the
even symmetry equation $\mathrm{GLE}_{n}^{(0)}$($p,A,\tau$) is log-free at
$z=0$.
\end{lemma}

By Lemma \ref{apparent,necessary condition}, it remains to prove log-freeness
at $z=0$ for $\mathbf{T}\in V_{n,p}^{(1)}(\tau)$. At $z=0$, the direct
Frobenius recurrence is intractable; we instead prove the following result by
the Hermite--Halphen method.

\begin{theorem}
\label{theorem, noneven's apparent} Let $\mathbf{T}\in$ $V_{n,p}^{(1)}(\tau)$
and $B$ be given by \eqref{B(T),noneven}. Then the non-even equation
GLE$_{n}^{(1)}$($p,T,\tau$) is log-free at $z=0$.
\end{theorem}

The non-even symmetric equation GLE$_{n}^{(1)}$($p,T,\tau$) is log-free if and
only if it is log-free at $z=0$. Define, for any%
\[
\mathbf{a}=\left(  a_{1},\cdot\cdot\cdot,a_{n+1}\right)  \in\mathrm{Sym}%
^{n+1}E_{\tau}^{\times},\text{ where }E_{\tau}^{\times}=E_{\tau}%
\setminus\{0\},
\]
and for any constant $c\in\mathbb{C}$, the \textbf{Hermite-Halphen type}
function $y_{\mathbf{a},c,p}(z)$ (denoted by $y_{\mathbf{a},c}(z)$ when $p\in
E_{\tau}\setminus E_{\tau}[2]$ is fixed) by%
\begin{equation}
y_{\mathbf{a},c}(z):=\dfrac{e^{cz}\,\prod\limits_{i=1}^{n+1}\sigma(z-a_{i}%
)}{~\sigma^{n}(z)\sqrt{\sigma(z+p)\,}\sqrt{\sigma(z-p)~}\,}.
\label{Hermite noneven}%
\end{equation}
Equation (\ref{Hermite noneven}) gives
\begin{equation}
y_{\mathbf{a},c}(z)=C\cdot z^{-n}\cdot h(z) \label{asymptotic}%
\end{equation}
for some constant $C\not =0$, where $h(z)$ is holomorphic at $z=0$.

Suppose that GLE$_{n}^{(1)}$($p,T,\tau$) admits a Hermite-Halphen type
solution $y_{\mathbf{a},c}(z)$. Then, by (\ref{asymptotic}), the singularity
at $z=0$ is log-free. Conversely, if the equation is log-free at $z=0$, then
it is log-free at all its singularities. This implies that the monodromy
matrices $M_{1}(T)$ and $M_{2}(T)$ satisfy the abelian relation%
\begin{equation}
M_{1}(T)M_{2}(T)=M_{2}(T)M_{1}(T),\label{M1M2=M2M1}%
\end{equation}
which in turn ensures the existence of a solution of the form $y_{\mathbf{a}%
,c}(z)$.

\begin{proposition}
\label{lemma, existence of Hermite Ansatz} The equation GLE$_{n}^{(1)}%
$($p,T,\tau$) is log-free if and only if there exists $\mathbf{a}$
$\mathbf{\in}\mathrm{Sym}^{n+1}E_{\tau}^{\times}$ and $c\in\mathbb{C}$ such
that $y_{\mathbf{a},c}(z)$ defined by \eqref{Hermite noneven} is a solution to
$\mathrm{GLE}_{n}^{(1)}$($p,T,\tau$).
\end{proposition}


Suppose $y_{\mathbf{a},c}(z)$ is a solution of GLE$_{n}^{(1)}$($p,T,\tau$) for
some $\mathbf{a}$ $\mathbf{\in}$ $\operatorname{Sym}^{n+1}E_{\tau}^{\times}$
and $c$ $\in$ $\mathbb{C}$. Since the local exponent difference of
GLE$_{n}^{(1)}$($p,T,\tau$) at $z$ $=$ $0$ is $2n+1$, it follows that
$a_{i}\not =0$ $\forall$ $1\leqslant$ $i$ $\leqslant$ $n+1$. In contrast, the
local exponent difference at $z=\pm p$ is $2$, which allows for the
possibility that two of the $a_{i}$'s may equal $p$ or $-p$.

After reordering, exactly one of the following occurs:

\begin{itemize}
\item[(a-i)] $a_{i}\not =\pm p$ and $a_{i}\not =a_{j}$ for any $1\leqslant
i\not =j\leqslant n+1$.\medskip

\item[(a-ii)] $a_{n}=a_{n+1}=p$, $a_{i}\not =\pm p$ and $a_{i}\not =a_{j}$ for
any $1\leqslant i\not =j\leqslant n-1$.\medskip

\item[(a-iii)] $a_{n}=a_{n+1}=-p$, $a_{i}\not =\pm p$ and $a_{i}\not =a_{j}$
for any $1\leqslant i\not =j\leqslant n-1$.\medskip
\end{itemize}

We define the admissible set $Y_{n,p}^{(1)}(\tau)$ $\subset$
$\operatorname{Sym}^{n+1}E_{\tau}^{\times}$ as follows:
\begin{equation}
Y_{n,p}^{(1)}(\tau):=\left\{  \mathbf{a}\in\operatorname{Sym}^{n+1}E_{\tau
}^{\times}\left\vert
\begin{array}
[c]{l}%
\text{There exists }c\in\mathbb{C}\text{ such that }y_{\mathbf{a},c}(z)\\
\text{defined in (\ref{Hermite noneven}) is a solution }\\
\text{to GLE}_{n}^{(1)}(p,T(\mathbf{a}),\tau)\text{ for some }T(\mathbf{a}%
)\in\mathbb{C}%
\end{array}
\right.  \right\} , \label{Yn's def}%
\end{equation}
where $T(\mathbf{a})$ is given in Theorem~\ref{thm, ya solves noneven ode} or
Theorem~\ref{thm, ya solves noneven ode 2} below.

Let
\[
\mathcal{V}^{(1)}_{n,p}(\tau):=\{T(\mathbf{a})\mid\,\mathbf{a}\in
Y_{n,p}^{(1)}(\tau)\}\subset{V}^{(1)}_{n,p}(\tau).
\]
By Proposition~\ref{lemma, existence of Hermite Ansatz}, there is a natural
morphism
\[
T:Y_{n,p}^{(1)}(\tau)\longrightarrow\mathcal{V}^{(1)}_{n,p}(\tau),
\quad\mathbf{a}\longmapsto T(\mathbf{a}).
\]

\begin{lemma}
\label{lem, Mj not +-Id} For each $T\in\mathcal{V}^{(1)}_{n,p}(\tau)$, the two
period monodromy matrices $M_{1}(T)$ and $M_{2}(T)$ cannot both belong to
$\{\pm I\}$.
\end{lemma}

\begin{proof}
Suppose, to the contrary, that $M_{j}(T)\in\{\pm I\},j=1,2.$ Let $y_{1},y_{2}$
be two linearly independent solutions of $\mathrm{GLE}_{n}^{(1)}(p,T,\tau)$
and set
\[
f=\frac{y_{1}}{y_{2}}.
\]
Since the period monodromies are both scalar, $f$ is single-valued on
$E_{\tau}\setminus\{0,\pm p\}$. Moreover, since all singularities are
log-free, $f$ extends to a nonconstant meromorphic map
\[
f:E_{\tau}\longrightarrow\mathbb{P}^{1}.
\]

The local exponent differences at $0$ and $\pm p$ are $2n+1$ and $2$,
respectively. Hence the ramification indices of $f$ at these points are $2n+1$
and $2$. Since
\[
f^{\prime}(z)=\frac{W(y_{1},y_{2})}{y_{2}(z)^{2}},
\]
where the Wronskian $W(y_{1},y_{2})\neq0$ is constant, $f$ has no other
ramification points on $E_{\tau}\setminus\{0,\pm p\}$. Let $d=\deg f$. The
Riemann--Hurwitz formula gives
\[
0=-2d+2n+2,
\]
and hence
\[
d=n+1.
\]

But the local degree at $0$ satisfies
\[
e_{0} (f)=2n+1\leqslant\deg f=n+1,
\]
which is a contradiction for $n\in\mathbb{N}$.
\end{proof}

Define
\begin{equation}
\overset{o}{X}_{n,p}(\tau):=\{\mathbf{a}\in Y_{n,p}^{(1)}(\tau)\mid
a_{i}\not =\pm p\text{ and }a_{i}\neq\pm a_{j},\forall1\leqslant i,j\leqslant
n+1\}. \label{Xn,p's def}%
\end{equation}
Lemma~\ref{lem, Mj not +-Id} gives $\#T^{-1}(T)$ $\leqslant2.$  Since
\[
\dim_{\mathbb{C}}\mathcal{V}_{n,p}^{(1)}(\tau)\leqslant\dim_{\mathbb{C}%
}V_{n,p}^{(1)}(\tau)=1\ \text{and }\overset{o}{X}_{n,p}(\tau)\subset
Y_{n,p}^{(1)}(\tau),
\]
we obtain
\begin{equation}
\dim_{\mathbb{C}}\overset{o}{X}_{n,p}(\tau) \leqslant\dim_{\mathbb{C}}%
Y_{n,p}^{(1)}(\tau) \leqslant1.\label{dim<=1}%
\end{equation}

\begin{theorem}
\label{thm, ya solves noneven ode} Assume that $\mathbf{a}\in
\operatorname{Sym}^{n+1}E_{\tau}^{\times}$ satisfies (a-i). Then
$y_{\mathbf{a},c}(z)$ defined in \eqref{Hermite noneven} is a solution to
$\mathrm{GLE}_{n}^{(1)}$($p,T,\tau$) if and only if $\mathbf{a}$ satisfies the
following equations
\begin{equation}%
\begin{split}
&  \sum_{i=1}^{n+1}\zeta(a_{i})+\dfrac{1}{\,2(n+1)\,}\sum_{i=1}^{n+1}%
\dfrac{\wp^{\prime}(a_{i})}{\,\wp(a_{i})-\wp(p)\,}\\[4pt]
=  &  \sum_{\substack{i=1\\i\neq j}}^{n+1}\zeta(a_{i}-a_{j})+n\,\zeta
(a_{j})+\dfrac{1}{\,2\,}\Big(\zeta(a_{j}+p)+\zeta(a_{j}-p)\Big),
\end{split}
\label{ya is sol. condition1}%
\end{equation}
for all $1\leqslant j\leqslant n$, and
\begin{equation}
\sum_{i=1}^{n+1}\dfrac{1}{\,\wp(a_{i})-\wp(p)\,}=0.
\label{ya is sol. condition2}%
\end{equation}
Moreover, the constants $c$ and $T$ are determined by $\mathbf{a}$ and $p$ as
follows:
\begin{align}
c  &  =\sum_{i=1}^{n+1}\zeta(a_{i})+\dfrac{1}{\,2(n+1)\,}\sum_{i=1}%
^{n+1}\dfrac{\wp^{\prime}(a_{i})}{\,\wp(a_{i})-\wp(p)\,}%
,\label{c in terms of aj}\\[4pt]
T  &  =\dfrac{n}{\,2(n+1)\,}\sum_{i=1}^{n+1}\dfrac{\wp^{\prime}(a_{i})}%
{\,\wp(a_{i})-\wp(p)\,}, \label{T in terms of aj}%
\end{align}
and the constant $B$ is given by \eqref{B(T),noneven}.
\end{theorem}

\begin{theorem}
\label{thm, ya solves noneven ode 2} Assume that $\mathbf{a}$ $\in
\operatorname{Sym}^{n+1}E_{\tau}^{\times}$ satisfies type (a-ii). Then
$y_{\mathbf{a},c}(z)$ defined in \eqref{Hermite noneven} is a solution to
GLE$_{n}^{(1)}$($p,T,\tau$) if and only if $\mathbf{a}$ satisfies the
following equations
\begin{equation}%
\begin{split}
&  \sum_{i=1}^{n-1}\zeta(a_{i})+2\zeta(p)+\dfrac{1}{4}\sum_{i=1}^{n-1}%
\dfrac{2\wp^{\prime}(a_{i})+\wp^{\prime}(p)}{\,\wp(a_{i})-\wp(p)\,}+\dfrac
{1}{2}\dfrac{\wp^{\prime\prime}(p)}{\wp^{\prime}(p)}\\[4pt]
=  &  \sum_{\substack{i=1\\i\neq j}}^{n+1}\zeta(a_{i}-a_{j})+n\,\zeta
(a_{j})+\dfrac{1}{\,2\,}\Big(\zeta(a_{j}+p)-3\zeta(a_{j}-p)\Big),
\end{split}
\label{ya is sol. condition2-1}%
\end{equation}
for all $1\leqslant j\leqslant n-1$, and
\begin{equation}
\sum_{i=1}^{n-1}\dfrac{\wp^{\prime}(p)}{\,\wp(a_{i})-\wp(p)\,}+\dfrac{2n}%
{n+3}\sum_{i=1}^{n-1}\dfrac{\wp^{\prime}(a_{i})}{\,\wp(a_{i})-\wp(p)\,}%
+\dfrac{2(n-1)}{n+3}\dfrac{\wp^{\prime\prime}(p)}{\wp^{\prime}(p)}=0.
\label{ya is sol. condition2-2}%
\end{equation}
Moreover, the constants $c$ and $T$ are determined by $\mathbf{a}$ and $p$ as
follows:%
\begin{align}
c  &  =\sum_{i=1}^{n-1}\zeta(a_{i})+2\zeta(p)+\dfrac{1}{4}\sum_{i=1}%
^{n-1}\dfrac{2\wp^{\prime}(a_{i})+\wp^{\prime}(p)}{\,\wp(a_{i})-\wp
(p)\,}+\frac{1}{2}\frac{\wp^{\prime\prime}(p)}{\wp^{\prime}(p)}%
,\label{c in terms of aj,2}\\[4pt]
T  &  =-\dfrac{3}{4}\sum_{i=1}^{n-1}\dfrac{\wp^{\prime}(p)}{\,\wp(a_{i}%
)-\wp(p)\,}+\frac{1}{2}\frac{\wp^{\prime\prime}(p)}{\wp^{\prime}(p)},
\label{T in terms of aj,2}%
\end{align}
and the constant $B$ is given by \eqref{B(T),noneven}.
\end{theorem}

The type (a-iii) formulas are obtained from those of type (a-ii) by replacing
$p$ with $-p$. We prove only Theorem~\ref{thm, ya solves noneven ode}.

\begin{proof}
[Proof of Theorem~\ref{thm, ya solves noneven ode}]From
\eqref{Hermite noneven},
\begin{align*}
\dfrac{\,y_{\mathbf{a},c}^{\prime}(z)\,}{y_{\mathbf{a},c}(z)}  &  =
c+\sum_{i=1}^{n+1}\zeta(z-a_{i})-n\zeta(z)-\dfrac{1}{\,2\,}\big(\zeta
(z+p)+\zeta(z-p)\big),\\[4pt]
\left(  \dfrac{\,y_{\mathbf{a},c}^{\prime}(z)\,}{y_{\mathbf{a},c}(z)}\right)
^{\prime}  &  =n\wp(z)+\dfrac{1}{\,2\,}\big(\wp(z+p)+\wp(z-p)\big)-\sum
_{i=1}^{n+1}\wp(z-a_{i}).
\end{align*}

Define the elliptic function
\[
g_{\mathbf{a},c}(z;p,T):=\left(  \dfrac{\,y_{\mathbf{a},c}^{\prime}%
(z)\,}{y_{\mathbf{a},c}(z)}\right)  ^{\prime}+\left(  \dfrac{\,y_{\mathbf{a}%
,c}^{\prime}(z)\,}{y_{\mathbf{a},c}(z)}\right)  ^{2}-q_{n}^{(1)}(z;p,T).
\]

Then $y_{\mathbf{a},c}(z)$ solves GLE$_{n}^{(1)}$($p,T,\tau$) if and only if
$g_{\mathbf{a},c}(z)\equiv0$, equivalently, $g_{\mathbf{a},c}(z)$ has no poles
and vanishes at one point.

Observe that $g_{\mathbf{a},c}(z;p,T)$ has at most simple poles at $z=0,\pm
p,a_{1},\cdots,a_{n+1}$. Requiring the residues at $z=\pm p$ to vanish gives
\begin{equation}
c+T-\sum\limits_{i=1}^{n+1}\zeta(a_{i}-p) -\left(  n+1\right)  \zeta(p)=0,
\label{z=p}%
\end{equation}
and
\begin{equation}
c+T-\sum\limits_{i=1}^{n+1}\zeta(a_{i}+p) +\left(  n+1\right)  \zeta(p) =0.
\label{z=-p}%
\end{equation}

Subtracting \eqref{z=p} and \eqref{z=-p}, and using
\begin{equation}
\zeta(a+b)=\zeta(a)+\zeta(b)+\frac{1}{2}\frac{\wp^{\prime}(a)-\wp^{\prime}%
(b)}{\wp(a)-\wp(b)}, \label{add formula, zeta}%
\end{equation}
gives \eqref{ya is sol. condition2}.

At $z=0$, the vanishing of the residue gives
\begin{equation}
T-n\left(  c-\sum\limits_{i=1}^{n+1}\zeta(a_{i})\right)  =0. \label{z=0}%
\end{equation}
Combining \eqref{z=0} with \eqref{z=p}, we determine $c$ and $T$ uniquely as
in \eqref{c in terms of aj} and \eqref{T in terms of aj}.

The corresponding residue conditions at $z=a_{j}$, $1\le j\le n+1$, read
\begin{equation}
c-\sum\limits_{\substack{i=1\\i\neq j}}^{n+1}\zeta(a_{i}-a_{j})-n\zeta
(a_{j})\\[4pt]
+\dfrac{1}{2}\left(  \zeta(p-a_{j})-\zeta(p+a_{j})\right)  =0. \label{z=aj}%
\end{equation}

Substituting \eqref{c in terms of aj} into \eqref{z=aj} gives the system of
$n+1$ equations in \eqref{ya is sol. condition1}. Since the sum of residues of
an elliptic function must be zero, one of these $n+1$ conditions can be
reduced. Hence \eqref{ya is sol. condition1} holds for $1\le j\le n$.

Relation \eqref{B(T),noneven} follows from the constant term of the equation
$g_{\mathbf{a},c}(p;p,$ $T)=0$.
\end{proof}

By Theorem \ref{thm, ya solves noneven ode} and Theorem
\ref{thm, ya solves noneven ode 2}, the set $Y_{n,p}^{(1)}(\tau)$ can be
equivalently described as
\begin{equation}
Y_{n,p}^{(1)}(\tau)=\left\{  \mathbf{a}\left\vert
\begin{array}
[c]{l}%
\text{either (a-i), \eqref{ya is sol. condition2} and
\eqref{ya is sol. condition1}}_{j}\text{ hold }\forall1\leqslant j\leqslant
n\text{;}\\
\text{or (a-ii), \eqref{ya is sol. condition2-2}}_{p}\text{ and
\eqref{ya is sol. condition2-1}}_{j,p}\text{ hold }\forall1\leqslant
j\leqslant n-1\text{;}\\
\text{or (a-iii), \eqref{ya is sol. condition2-2}}_{-p}\text{ and
\eqref{ya is sol. condition2-1}}_{j,-p}\text{ hold }\forall1\leqslant
j\leqslant n-1
\end{array}
\right.  \right\}  . \label{Yn,p, apparent}%
\end{equation}
Set
\begin{gather}
\left(  x_{p},y_{p}\right)  : =\left(  \wp(p),\wp^{\prime}(p)\right)
,\nonumber\\
\left(  x_{i},y_{i}\right)  : =\left(  \wp(a_{i}),\wp^{\prime}(a_{i})\right)
,\text{ }1\leqslant i\leqslant n+1. \label{x,y}%
\end{gather}
By the classical differential equation satisfied by the Weierstrass $\wp
$-function, we have
\[
y_{i}^{2}=4x_{i}^{3}-g_{2}x_{i}-g_{3}\quad\forall1\leqslant i\leqslant
n+1\text{ and }i=p.
\]
Then (\ref{ya is sol. condition2}) becomes
\begin{equation}
\sum_{i=1}^{n+1}\frac{1}{\,x_{i}-x_{p}\,}=0. \label{ya is sol. condition2*}%
\end{equation}
Assume further that $\mathbf{a}$ $\in\overset{o}{X}_{n,p}(\tau)\subset
Y_{n,p}^{(1)}(\tau)$. Then $\mathbf{a}$ satisfies (a-i), and using the
addition formula:
\[
\zeta(a+b)=\zeta(a)+\zeta(b)+\frac{1}{2}\frac{\wp^{\prime}(a)-\wp^{\prime}%
(b)}{\wp(a)-\wp(b)},
\]
the system \eqref{ya is sol. condition1} becomes equivalent to:%

\begin{align}
&  \sum_{\substack{i=1\\i\neq j}}^{n+1}\left[  \dfrac{1}{\,x_{i}-x_{j}%
\,}-\dfrac{1}{\,n+1\,}\cdot\dfrac{1}{\,x_{i}-x_{p}\,}\right]  y_{i}%
\label{ya is sol. condition1*}\\
&  +\left[  \,\sum_{\substack{i=1\\i\neq j}}^{n}\left(  \dfrac{1}%
{\,x_{i}-x_{j}\,}\right)  +\dfrac{n}{\,n+1\,}\cdot\dfrac{1}{\,x_{j}-x_{p}%
\,}\right]  y_{j}=0,\text{ }\forall1\leqslant j\leqslant n.\nonumber
\end{align}

Identifying a point $\mathbf{a}=(a_{1},\cdots,a_{n+1})\in\mathrm{Sym}%
^{n+1}E_{\tau}$ with
\[
(\mathbf{x},\mathbf{y})=(x_{1},\cdots,x_{n+1},y_{1},\cdots,y_{n+1}%
)\in\mathbb{C}^{2n+2}%
\]
via (\ref{x,y}), the set $\overset{o}{X}_{n,p}(\tau)$ can be characterized as:%
\begin{equation}
\overset{o}{X}_{n,p}(\tau)=\left\{  \mathbf{a}\left\vert
\begin{array}
[c]{l}%
x_{i}\not =x_{p}\text{, }y_{i}\not =\infty,0,\text{ }\forall1\leqslant
i\leqslant n+1\\
x_{i}\neq x_{\ell}\text{ }\forall i\not =\ell,\text{ }1\leqslant
i,\ell\leqslant n+1\\
\text{such that \eqref{ya is sol. condition2*},
\eqref{ya is sol. condition1*}}_{j}\text{ hold }\forall1\leqslant j\leqslant
n\text{.}%
\end{array}
\right.  \right\}  , \label{Xnp}%
\end{equation}
which defines an affine algebraic subvariety of $\operatorname{Sym}%
^{n+1}E_{\tau}$.

Let $\mathbf{a}_{0}(p)\in Y_{n,p}^{(1)}(\tau)$ be the zero configuration
associated with
\begin{equation}
\text{GLE}_{n}^{(1)}(p,0,\tau)=\text{GLE}_{n}^{(0)}(p,A_{0},\tau)\text{,
}A_{0}=-\frac{\wp^{\prime\prime}(p)}{4\wp^{\prime}(p)}\text{.} \label{id1}%
\end{equation}
Since this equation is even, it is log-free at $z=0$. We first treat the case
\[
\mathbf{a}_{0}(p)\in\mathring{X}_{n,p}(\tau),
\]
for which a dimension argument gives log-freeness for every $T$. The
complementary set
\begin{equation}
\Theta(\tau):=\left\{  p\in E_{\tau}\setminus E_{\tau}[2]|\,\mathbf{a}%
_{0}(p)\in Y_{n,p}^{(1)}(\tau)\setminus\overset{o}{X}_{n,p}(\tau)\right\}  .
\label{theta tau's def}%
\end{equation}
is then shown to be finite and removed by approximation.

\begin{theorem}
\label{thm, Xn,p is smooth} Fix $\tau\in\mathbb{H}$ and let $p\in E_{\tau
}\setminus E_{\tau}[2]$. Suppose $\mathbf{a}_{0}(p)\in\overset{o}{X}%
_{n,p}(\tau)$, then the irreducible component of $\overset{o}{X}_{n,p}(\tau)$
containing $\mathbf{a}_{0}(p)$ is an algebraic curve.
\end{theorem}

The system \eqref{ya is sol. condition1*} is a homogeneous linear system
consisting of $n$ equations in the variables $y_{1},\cdot\cdot\cdot,y_{n+1}$.
This system is represented by an $n\times(n+1)$ matrix $M:=(M_{ij}%
)_{n\times(n+1)}$, whose entries are given by
\begin{equation}
M_{ij}:=%
\begin{cases}
~~\sum\limits_{\substack{k=1\\k\neq j}}^{n+1}\left(  \dfrac{1}{\,x_{k}%
-x_{j}\,}\right)  +\dfrac{n}{\,n+1\,}\cdot\dfrac{1}{\,x_{j}-x_{p}\,}, & \quad
i=j,\\[18pt]%
~~\dfrac{1}{\,x_{j}-x_{i}\,}-\dfrac{1}{\,n+1\,}\cdot\dfrac{1}{\,x_{j}-x_{p}%
\,}, & \quad i\neq j.
\end{cases}
\label{the matrix M from x,y}%
\end{equation}
Then%

\begin{equation}
\label{eq:M}\resizebox{1.1\textwidth}{!}{$ M= \begin{pmatrix} \displaystyle \sum_{k=2}^{n+1}\frac{1}{x_k-x_1} +\frac{n}{n+1}\frac{1}{x_1-x_p} & \displaystyle \frac{1}{x_2-x_1} -\frac{1}{n+1}\frac{1}{x_2-x_p} & \cdots & \displaystyle \frac{1}{x_{n+1}-x_1} -\frac{1}{n+1}\frac{1}{x_{n+1}-x_p} \\[2mm] \displaystyle \frac{1}{x_1-x_2} -\frac{1}{n+1}\frac{1}{x_1-x_p} & \displaystyle \sum_{\substack{k=1\\ k\ne 2}}^{n+1} \frac{1}{x_k-x_2} +\frac{n}{n+1}\frac{1}{x_2-x_p} & \cdots & \displaystyle \frac{1}{x_{n+1}-x_2} -\frac{1}{n+1}\frac{1}{x_{n+1}-x_p} \\ \vdots & \vdots & \ddots & \vdots \\ \displaystyle \frac{1}{x_1-x_n} -\frac{1}{n+1}\frac{1}{x_1-x_p} & \displaystyle \frac{1}{x_2-x_n} -\frac{1}{n+1}\frac{1}{x_2-x_p} & \cdots & \displaystyle \sum_{\substack{k=1\\ k\ne n}}^{n+1} \frac{1}{x_k-x_n} +\frac{n}{n+1}\frac{1}{x_n-x_p} \end{pmatrix}_{n\times(n+1)} $}.
\end{equation}

Under \eqref{ya is sol. condition2*}, we have
\[
\sum_{i=1}^{n+1}\frac{1}{\,x_{i}-x_{p}\,}=0.
\]

\begin{lemma}
\label{lem, rank M<n} Let $M=(M_{ij})_{n\times(n+1)}$ be the matrix defined by
(\ref{eq:M}). Assume that \eqref{ya is sol. condition2*} holds. Then
\[
\operatorname{rank}M\leqslant n-1.
\]

\end{lemma}

\begin{proof}
Set
\[
t_{i}:=\dfrac{1}{\,x_{i}-x_{p}\,}\neq0,\quad1\leqslant i\leqslant n+1.
\]
Then condition \eqref{ya is sol. condition2*} is equivalent to
\begin{equation}
\sum_{i=1}^{n+1}t_{i}=0. \label{1027equ1}%
\end{equation}
Write $M=M^{\prime}D$ with
\[
D:=\operatorname{diag}\left(  t_{1},\ldots,t_{n+1}\right)  \in
\operatorname{GL}(n+1,\mathbb{C}),
\]
and
\[
M^{\prime}=
\begin{cases}
\sum\limits_{\substack{k=1\\k\neq j}}^{n+1}\left(  \dfrac{t_{k}}{\,t_{j}%
-t_{k}\,}\right)  +\dfrac{n}{\,n+1\,}, & \quad i=j,\\
~~\dfrac{t_{i}}{\,t_{i}-t_{j}\,}-\dfrac{1}{\,n+1\,}, & \quad i\neq j.
\end{cases}
\]

Since $D$ is invertible, we have
\[
\operatorname{rank}M =\operatorname{rank}(M^{\prime}D) =\operatorname{rank}%
M^{\prime}.
\]

Set the matrix $G=(G_{ij})$ by
\[
G:=\operatorname{diag}(t_{1}-t_{n+1},\ldots,t_{n}-t_{n+1})M^{\prime},
\]
so that
\begin{equation}
G_{ij}=
\begin{cases}
~~\sum\limits_{\substack{k=1\\k\neq j}}^{n+1}\left(  \dfrac{t_{k}\left(
t_{j}-t_{n+1}\right)  }{\,t_{j}-t_{k}\,}\right)  +\dfrac{n}{\,n+1\,}\left(
t_{j}-t_{n+1}\right)  , & \quad i=j,\\
~~\dfrac{t_{i}\left(  t_{i}-t_{n+1}\right)  }{\,t_{i}-t_{j}\,}-\dfrac
{1}{\,n+1\,}\left(  t_{i}-t_{n+1}\right)  , & \quad i\neq j.
\end{cases}
\label{the matrix G}%
\end{equation}
Thus
\[
\operatorname{rank}M^{\prime}=\operatorname{rank}G.
\]

Fix $j\in\{1,\ldots,n+1\}$. We claim that
\[
\sum_{i=1}^{n}G_{ij} =\frac{n}{n+1}\sum_{i=1}^{n+1}t_{i} .
\]
\noindent\textbf{Case 1.} ~\thinspace\ If $j\neq n+1$, then
\begin{align*}
\sum\limits_{i=1}^{n}G_{ij}  &  =\sum\limits_{\substack{k=1\\k\neq j}%
}^{n+1}\left(  \dfrac{t_{k}\left(  t_{j}-t_{n+1}\right)  }{\,t_{j}-t_{k}%
\,}\right)  +\dfrac{n}{\,n+1\,}\left(  t_{j}-t_{n+1}\right) \\
&  \qquad+\sum\limits_{\substack{i=1\\i\neq j}}^{n}\left(  \dfrac{t_{i}\left(
t_{i}-t_{n+1}\right)  }{\,t_{i}-t_{j}\,}-\dfrac{1}{\,n+1\,}\left(
t_{i}-t_{n+1}\right)  \right) \\
&  =\dfrac{t_{n+1}}{\,n+1\,}+\dfrac{n\,t_{j}}{\,n+1\,}+\dfrac{\,\left(
n-1\right)  t_{n+1}\,}{n+1}+\dfrac{n}{\,n+1\,}\sum
\limits_{\substack{i=1\\i\neq j}}^{n}t_{i}=\dfrac{n}{\,n+1\,}\sum_{i=1}%
^{n+1}t_{i}.
\end{align*}

\textbf{Case 2.} ~\thinspace\ If $j=n+1$, then
\[
\sum\limits_{i=1}^{n}G_{i,n+1}=\sum\limits_{i=1}^{n}\left(  \dfrac
{t_{i}\left(  t_{i}-t_{n+1}\right)  }{\,t_{i}-t_{n+1}\,}-\dfrac{1}%
{\,n+1\,}\left(  t_{i}-t_{n+1}\right)  \right)  =\dfrac{n}{n+1}\sum
_{i=1}^{\,n+1\,}t_{i}.
\]
Thus
\[
\left(
\begin{array}
[c]{c|c}%
Id_{n-1} & \mathbf{0}\\[2pt]\hline
1\cdots1 & 1
\end{array}
\right)  _{\!n\times n}G=\left(
\begin{array}
[c]{c}%
G^{\prime}\\[2pt]\hline
\frac{n}{n+1}\sum\limits_{i=1}^{n+1}t_{i}\cdots\frac{n}{n+1}\sum
\limits_{i=1}^{n+1}t_{i}%
\end{array}
\right)  _{\!n\times(n+1)}%
\]
for some $(n-1)\times(n+1)$ matrix $G^{\prime}$.

By \eqref{1027equ1}, the last row vanishes; hence
\[
\operatorname{rank}G=\operatorname{rank}G^{\prime}\leqslant n-1.
\]
\medskip
\end{proof}

By Lemma~\ref{lem, rank M<n}, one equation in system
\eqref{ya is sol. condition1*} is redundant under
\eqref{ya is sol. condition2*}. Omitting the n-th equation, we may write
$\overset{o}{X}_{n,p}(\tau)$ as:
\begin{equation}
\overset{o}{X}_{n,p}(\tau)=\left\{  \mathbf{a}\left\vert
\begin{array}
[c]{l}%
x_{i}\not =x_{p}\text{, }y_{i}\not =\infty,0,\text{ }\forall1\leqslant
i\leqslant n+1,\\
x_{i}\neq x_{\ell}\text{ }\forall i\not =\ell,\text{ }1\leqslant
i,\ell\leqslant n+1,\\
\text{such that \eqref{ya is sol. condition2*},
\eqref{ya is sol. condition1*}}_{j}\text{ hold }\forall1\leqslant j\leqslant
n-1\text{.}%
\end{array}
\right.  \right\}  . \label{Xnp,new}%
\end{equation}

Define
\[
F:\mathbb{C}^{2n+2}\longrightarrow\mathbb{C}^{2n+1}%
\]
\[
(\mathbf{x},\mathbf{y})=(x_{1},\cdots,x_{n+1},y_{1},\cdots,y_{n+1}%
)\longmapsto(f_{1},\cdots,f_{n+1},h_{1},\cdots,h_{n-1},k_{1}),
\]
where%
\[
f_{j}(\mathbf{x},\mathbf{y})=y_{j}^{2}-(4x_{j}^{3}-g_{2}x_{j}-g_{3}),
\]%
\begin{align*}
h_{j}(\mathbf{x},\mathbf{y})  &  =\sum_{\substack{i=1\\i\neq j}}^{n+1}\left[
\dfrac{1}{\,x_{i}-x_{j}\,}-\dfrac{1}{\,n+1\,}\cdot\dfrac{1}{\,x_{i}-x_{p}%
\,}\right]  y_{i}\\
&  +\left[  \,\sum_{\substack{i=1\\i\neq j}}^{n}\left(  \dfrac{1}%
{\,x_{i}-x_{j}\,}\right)  +\dfrac{n}{\,n+1\,}\cdot\dfrac{1}{\,x_{j}-x_{p}%
\,}\right]  y_{j},
\end{align*}%
\[
k_{1}(\mathbf{x},\mathbf{y})=\sum_{i=1}^{n+1}\frac{1}{\,x_{i}-x_{p}\,}.
\]
Via (\ref{x,y}), we obtain the isomorphism
\begin{equation}
\overset{o}{X}_{n,p}(\tau)\cong V(F):=\{(\mathbf{x},\mathbf{y})\in
\mathbb{C}^{2n+2}\mid F(\mathbf{x},\mathbf{y})=0\}. \label{Xo,np and F(x,y)=0}%
\end{equation}

\begin{proof}
[Proof of Theorem \ref{thm, Xn,p is smooth}]Let $(\mathbf{x}_{0}%
,\mathbf{y}_{0})$ denote the corresponding point of $\mathbf{a}_{0}(p)$. Since
$\mathbf{a}_{0}(p)\in\overset{o}{X}_{n,p}(\tau)$, we have $F(\mathbf{x}%
_{0},\mathbf{y}_{0})=0.$

By Krull's height theorem,  every irreducible component of
\(V(F)\) has codimension at most \(2n+1\). Hence the irreducible
component containing \((\mathbf{x}_0,\mathbf{y}_0)\) has dimension at
least
\[
2n+2-(2n+1)=1.
\]
Via the isomorphism \eqref{Xo,np and F(x,y)=0}, it follows that the irreducible
component of \(\mathring{X}_{n,p}(\tau)\) containing
\(\mathbf{a}_0(p)\), denoted by
\(\mathring{X}_{n,p}(\tau)_{\mathbf{a}_0(p)}\), satisfies
\begin{equation}
\dim_{\mathbb{C}}\overset{o}{X}_{n,p}(\tau)_{\mathbf{a}_{0}(p)}\geqslant1.
\label{dim, X0,np,>1}%
\end{equation}
Together with (\ref{dim<=1}), \eqref{dim, X0,np,>1} gives
\[
\dim_{\mathbb{C}}\overset{o}{X}_{n,p}(\tau)_{\mathbf{a}_{0}(p)}=1.
\]

Hence the irreducible component of $\mathring{X}_{n,p}(\tau)$ containing
$\mathbf{a}_{0}(p)$ is an algebraic curve.
\end{proof}

\begin{theorem}
\label{thm, noneven apparent for generic p} Assume the hypotheses of Theorem
\ref{thm, Xn,p is smooth} hold. Then the equation GLE$_{n}^{(1)}(p,T,\tau)$ is
log-free at $z=0$ for all $T$ $\in\mathbb{C}$.
\end{theorem}

\begin{proof}

Consider the Laurent expansion of the potential $q_{n}^{(1)}(z;p,T)$ near
$z=0$:
\[
q_{n}^{(1)}(z;p,T)=\sum_{i=-2}^{\infty}q_{i,1}\,z^{i},\quad\text{near}\quad
z=0,
\]
where each $q_{i,1}\in\mathbb{Q}[\wp(p),\wp^{\prime}(p),e_{j}(\tau)][T]$.

Substituting the local solution
\[
y_{T}(z)=\sum_{m=0}^{\infty}c_{m}(T,p,\tau)z^{-n+m},\text{ with }%
c_{0}:=1\text{ and }c_{m}\in\mathbb{C},
\]
into the equation and expanding around $z=0$, we obtain the recurrence
relation:%
\begin{equation}
j\left(  j-\left(  2n+1\right)  \right)  c_{j}=\sum_{i=-1}^{j-2}%
q_{i,1}\,c_{j-2-i}. \label{re}%
\end{equation}
The coefficients $c_{1},\cdot\cdot\cdot,c_{2n}$ are thus determined
recursively from $c_{0}$ using this formula. Moreover, it follows from the
form of $q_{i,1}$ that
\[
c_{m}\in\mathbb{Q}[\wp(p),\wp^{\prime}(p),e_{j}(\tau)][T]\text{, }1\leqslant
m\leqslant2n.
\]

Define the polynomial $f_{p}(T)\in\mathbb{Q}[\wp(p),\wp^{\prime}(p),e_{j}%
(\tau)][T]$ by evaluating the right-hand side of the recurrence at $j=2n+1$:
\begin{equation}
f_{p}(T):=\sum_{i=-1}^{2n-1}q_{i,1}\,c_{2n-1-i}\text{.}\label{def, fp(T)}%
\end{equation}

Then for each $T\in\mathbb{C}$, the equation GLE$_{n}^{(1)}(p,T,\tau)$ is
log-free at $z=0$ if and only if $f_{p}(T)=0$.

We claim that
\[
f_{p}(T)\equiv0.
\]
Suppose, for contradiction, that $f_{p}(T)\not \equiv 0$. Then $f_{p}(T)$ has
at most $\deg_{T}f_{p}$ distinct roots, so the equation would be log-free only
for finitely many values of $T$.

However, since $\mathbf{a}_{0}(p)$ $\in$ $\overset{o}{X}_{n,p}(\tau)$
$\neq\emptyset$, Theorem \ref{thm, Xn,p is smooth} implies that the irreducible
component $\overset{o}{X}_{n,p}(\tau)_{\mathbf{a}_{0}(p)}$ is an algebraic
curve. For every $\mathbf{a}\in\overset{o}{X}_{n,p}(\tau)_{\mathbf{a}_{0}(p)}%
$, there exists a Hermite-Halphen type solution $y_{\mathbf{a},c}(z)$ solving
GLE$_{n}^{(1)}(p,T,\tau)$, with $T$ given by \eqref{T in terms of aj}.

The image of $T$ on $\mathring{X}_{n,p}(\tau)_{a_{0}(p)}$ is infinite.
Otherwise it would be finite, say $\{T_{1},\ldots,T_{d}\}.$ Since each fiber
contains at most two points, the component itself would then be finite,
contradicting the fact that it has dimension one.

Therefore, the equation is log-free for infinitely many distinct values of
$T$, implying that $f_{p}(T)$ has infinitely many roots---a contradiction.
Hence $f_{p}(T)\equiv0$, and GLE$_{n}^{(1)}(p,T,\tau)$ is automatically
log-free at $z=0$ for all $T\in\mathbb{C}$.
\end{proof}

\begin{theorem}
\label{thm, theta is finite}For any fixed $\tau\in\mathbb{H}$, the set
$\Theta(\tau)$ is finite.
\end{theorem}

We use the following facts from the even spectral theory of $\mathrm{GLE}%
_{n}^{(0)}$~\cite{Chen-Kuo-Lin-Lame I}. Let $\Phi(z):=y_{1}(z)y_{2}(z)$ be the
product of any two solutions of GLE$_{n}^{(0)}$($p,A$,$\tau$). Then $\Phi(z)$
satisfies its second symmetric product equation:
\begin{equation}
\Phi^{\prime\prime\prime}(z)-4q_{n}^{(0)}(z;p,A)\,\Phi^{\prime}(z)-2q_{n}%
^{(0)\prime}(z;p,A)\,\Phi(z)=0\text{ on }E_{\tau}\text{.}
\label{3rd ode, even}%
\end{equation}


\begin{lemma}
\label{Lemma of Takemura} \cite{Chen-Kuo-Lin-Lame I} Up to a nonzero constant
multiple, there exists a unique nontrivial even elliptic solution $\Phi
_{e,p}^{(0)}(z;A)$ of the third-order equation \eqref{3rd ode, even} of the
form
\begin{equation}
\Phi_{e,p}^{(0)}(z;A) = \sum_{j=0}^{n}b_{j,p}(A)\wp(z)^{n-j} + \frac
{d_{n,p}(A)}{\wp(z)-\wp(p)},\label{5}%
\end{equation}
where $b_{j,p}(A)$ and $d_{n,p}(A)$ are polynomials in $A$ whose coefficients
are rational functions of $\wp(p)$, $\wp^{\prime}(p)$, and $e_{k}(\tau)$.
These polynomials have no nonconstant common divisor, and the leading
coefficient of $b_{n,p}(A)$ is $\frac12$. Moreover,
\[
2n = \deg_{A} b_{n,p}(A) = \max\left\{  \deg_{A} b_{j,p}(A),\deg_{A}
d_{n,p}(A) \right\} .
\]
Consequently,
\[
\deg_{A}\Phi_{e,p}^{(0)}(z;A)=2n.
\]

\end{lemma}

\begin{remark}
\label{remark, weight}  Assign the weights
\[
\operatorname{wt}A=1,\quad\operatorname{wt}\wp=\operatorname{wt}e_{k} =2,
\quad\operatorname{wt}\wp^{\prime}=3.
\]
Then $b_{j,p}(A)$ and $d_{n,p}(A)$ are weighted homogeneous of weights
\[
\operatorname{wt} b_{j,p}(A)=2j, \qquad\operatorname{wt} d_{n,p}(A)=2n+2,
\]
respectively. Hence every term in \eqref{5} has weight $2n$, and therefore
$\Phi_{e,p}^{(0)}(z;A)$ is weighted homogeneous of weight $2n$.

In particular,
\begin{equation}
p^{2j} b_{j,p}(A_{0}(p))\to\beta_{j},\ p^{2n+2}d_{n,p}(A_{0}(p))\to\delta
_{n}\quad\text{as }p\to0.\label{bj,p's behavior at 0}%
\end{equation}

\end{remark}

Define
\begin{equation}
\begin{aligned} Q^{(0)}_{n,p}(A):=&\dfrac{1}{2}\Phi_{e,p}^{(0)}(z;A)\,\Phi_{e,p}^{(0)\prime\prime }(z;A)-\dfrac{1}{4}\Phi_{e,p}^{(0)\prime}(z;A)^{2}\\ &-q_{n}^{(0)}(z;p,A)\Phi_{e,p}^{(0)}(z;A)^2 ,\end{aligned} \label{spectral polynomial, even's def}%
\end{equation}
which is independent of $z$ by \eqref{3rd ode, even}. Indeed, $Q_{n,p}%
^{(0)}(A)$ is a polynomial in $A$, with coefficients rational in $\wp(p)$,
$\wp^{\prime}(p)$, $e_{k}(\tau)$, and $\deg_{A}Q_{n,p}^{(0)}(A)=4n+2$. The
polynomial $Q_{n,p}^{(0)}(A)$ is referred to as the spectral polynomial
associated with the even-symmetry generalized Lam\'{e} equation GLE$_{n}%
^{(0)}$($p,A,\tau$). The identities
\[
q_{n}^{(0)}(z;p,A)=q_{n}^{(0)}(z;-p,-A),\ Q_{n,p}^{(0)}(A)=Q_{n,-p}^{(0)}(-A)
\]
do not eliminate the dependence on $\wp^{\prime}(p)$. For $n=1$, for
instance,
\begin{equation}
Q_{1,p}^{(0)}(A)=-Y_{1}(A)Y_{2}(A), \label{spectral poly for even n=1}%
\end{equation}
where
\begin{align*}
Y_{1}(A)  &  =A^{3}-\frac{5\wp^{\prime\prime}(p)}{4\wp^{\prime}(p)}%
A^{2}+3\left(  \wp(p)+\frac{1}{16}\frac{\wp^{\prime\prime}(p)^{2}}{\wp
^{\prime}(p)^{2}}\right)  A\\
&  \quad\quad+\frac{\wp^{\prime}(p)}{2}-\frac{9\wp^{\prime\prime}(p)}%
{4\wp^{\prime}(p)}\wp(p)+\frac{9}{64}\frac{\wp^{\prime\prime}(p)^{3}}%
{\wp^{\prime}(p)^{3}},\\
Y_{2}(A)  &  =A^{3}-\frac{\wp^{\prime\prime}(p)}{4\wp^{\prime}(p)}A^{2}%
-\frac{5}{16}\frac{\wp^{\prime\prime}(p)^{2}}{\wp^{\prime}(p)^{2}}%
A+2\wp^{\prime}(p)-\frac{3}{64}\frac{\wp^{\prime\prime}(p)^{3}}{\wp^{\prime
}(p)^{3}}.
\end{align*}

For $A_{0}(p)$ in (\ref{id1}), since GLE$_{n}^{(0)}$($p,$ $A_{0},\tau$) is an
even elliptic equation, $y_{\mathbf{a}_{0}(p)\mathbf{,}c}(-z)$ is also a
solution. Furthermore, $y_{-\mathbf{a}_{0}(p),-c}(z)=d\,y_{\mathbf{a}%
_{0}(p),c}(-z)$ for some $d\in\mathbb{C}^{\ast}$ and hence $y_{-\mathbf{a}%
_{0}(p),-c}(z)$ is also a solution of $\mathrm{GLE}_{n}^{(0)}(p,A_{0},\tau)$.

Define%
\begin{equation}
\Theta_{1}(\tau):=\{p\in E_{\tau}\setminus E_{\tau}[2]|\mathbf{a}_{0}(p)\text{
is either of type (a-ii) or (a-iii)}\}, \label{theta1, def}%
\end{equation}
and%
\begin{equation}
\Theta_{2}(\tau):=\left\{  p\in E_{\tau}\setminus E_{\tau}[2]\left\vert
\begin{array}
[c]{l}%
\mathbf{a}_{0}(p)\text{ is of type (a-i) and }\mathbf{a}_{0}(p)\cap
-\mathbf{a}_{0}(p)\neq\emptyset
\end{array}
\right.  \right\}  . \label{theta2,def}%
\end{equation}
By definition,
\begin{equation}
\Theta(\tau)=\Theta_{1}(\tau)\cup\Theta_{2}(\tau). \label{theta=theta1+theta2}%
\end{equation}

\begin{lemma}
\label{lem, theta2} Let $\tau\in\mathbb{H}$. Then
\begin{equation}
\Theta_{1}(\tau) = \left\{  p\in E_{\tau}\setminus E_{\tau}[2] \;\middle|\;
d_{n,p}(A_{0}(p))=0 \right\} \label{Theta1 given by dnp}%
\end{equation}
and
\begin{equation}
\Theta_{2}(\tau)=\left\{  p\in E_{\tau}\setminus E_{\tau}[2]|\,Q_{n,p}%
^{(0)}(A_{0}(p))=0\right\}  . \label{Theta2 given by Q}%
\end{equation}

\end{lemma}

\begin{proof}
By the classification \emph{(a-i)}--\emph{(a-iii)}, the product
\[
y_{\mathbf{a}_{0}(p),c(p)}(z) y_{\mathbf{a}_{0}(p),c(p)}(-z)
\]
has a simple pole at $z=p$ in case \emph{(a-i)}, whereas it is holomorphic
there in cases \emph{(a-ii)} and \emph{(a-iii)}. Since
\[
\Phi_{e,p}^{(0)}(z;A_{0}(p)) = \sum_{j=0}^{n} b_{j,p}(A_{0}(p))\wp(z)^{n-j} +
\frac{d_{n,p}(A_{0}(p))} {\wp(z)-\wp(p)}
\]
and $\wp^{\prime}(p)\neq0$ for $p\notin E_{\tau}[2]$, identity
\eqref{Theta1 given by dnp} follows.

For \eqref{Theta2 given by Q}, it is known from \cite[Proposition~2.4]%
{Chen-Kuo-Lin-Lame I} that
\[
\mathbf{a}_{0}(p)\cap(-\mathbf{a}_{0}(p))\neq\emptyset
\]
if and only if
\[
Q_{n,p}^{(0)}(A_{0}(p))=0.
\]

Suppose that $\mathbf{a}_{0}(p)$ is of type \emph{(a-ii)} or \emph{(a-iii)}
and, contrary to the claim,
\[
\mathbf{a}_{0}(p)\cap(-\mathbf{a}_{0}(p))\neq\emptyset.
\]
Then the Wronskian of $y_{\mathbf{a}_{0}(p),c}(z)$ and $y_{\mathbf{a}%
_{0}(p),c}(-z)$ vanishes. Hence these two solutions are linearly dependent,
and there exists $d\in\mathbb{C}^{\times}$ such that
\[
y_{\mathbf{a}_{0}(p),c}(z) = d\,y_{\mathbf{a}_{0}(p),c}(-z) =
d\,y_{-\mathbf{a}_{0}(p),-c}(z).
\]
Thus
\[
\mathbf{a}_{0}(p)=-\mathbf{a}_{0}(p).
\]
Since $\mathbf{a}_{0}(p)$ is of type \emph{(a-ii)} or \emph{(a-iii)}, this
forces
\[
p=-p,
\]
and hence $p\in E_{\tau}[2]$, contradicting $p\in E_{\tau}\setminus E_{\tau
}[2]$.

Hence every zero of $Q_{n,p}^{(0)}(A_{0}(p))$ lies in case \emph{(a-i)},
proving \eqref{Theta2 given by Q}.
\end{proof}

\begin{proposition}
\label{prop, theta1 theta2 finite} For any fixed $\tau\in\mathbb{H}$, both
$\Theta_{1}(\tau)$ and $\Theta_{2}(\tau)$ are finite.
\end{proposition}

\begin{proof}
Recall \eqref{id1} and let
\[
\Phi_{p}(z):=\Phi_{e,p}^{(0)}(z;A_{0}(p)).
\]

Normalize $\Phi_{p}$ as follows. The local exponents of the third-order
equation at $z=0$ are $-2n,1$ and $2n+2.$ Since $\Phi_{p}$ is even, the
exponent $1$ cannot occur. The exponent $2n+2$ is also impossible: otherwise
$\Phi_{p}$ would have a zero of order $2n+2$ at $z=0$, whereas, by
Lemma~\ref{Lemma of Takemura}, its only possible poles away from $0$ are
simple poles at $\pm p$, contradicting the divisor relation for a nonzero
elliptic function. Hence $\Phi_{p}$ has a pole of order $2n$ at $0$.

After rescaling, normalize
\[
\Phi_{p}(z) = \wp(z)^{n} + \sum_{j=1}^{n} b_{j,p}(A_{0}(p))\wp(z)^{n-j} +
\frac{d_{n,p}(A_{0}(p))} {\wp(z)-\wp(p)}.
\]
We keep the notation $b_{j,p}$, $d_{n,p}$, and $Q_{n,p}^{(0)}$ after this
normalization. Such a rescaling changes $d_{n,p}$ by a nonzero factor and
$Q_{n,p}^{(0)}(A_{0}(p))$ by its square, without changing their zero sets.

We analyze simultaneously $d_{n,p}(A_{0}(p))$ and $Q_{n,p}^{(0)}(A_{0}(p))$ as
$p\to0$. Set
\[
z=px, \qquad F_{p}(x):=p^{2n}\Phi_{p}(px).
\]
From \eqref{id1},
\[
A_{0}(p) = \frac{3}{4p}+O(p).
\]
By Remark~\ref{remark, weight}, we obtain
\[
F_{p}(x) \longrightarrow F(x)
\]
locally uniformly on $\mathbb{C}\setminus\{0,\pm1\}$, where
\begin{equation}
\label{eq, common limiting F}F(x) = x^{-2n} + \sum_{j=1}^{n}\beta
_{j}x^{-2(n-j)} + \delta_{n}\frac{x^{2}}{1-x^{2}}.
\end{equation}

Moreover,
\[
p^{2}q_{n}^{(0)}(px;p,A_{0}(p)) \longrightarrow q_{*}(x)
\]
locally uniformly on $\mathbb{C}\setminus\{0,\pm1\}$, where
\[
q_{*}(x) = n(n+1)\left( \frac1{x^{2}}-1\right)  + \frac{3}{(x^{2}-1)^{2}}.
\]
Since $F_{p}$ satisfies the rescaled third-order equation, passing to the
limit gives
\begin{equation}
\label{eq, common limiting third order}F^{\prime\prime\prime}-4q_{*}F^{\prime
}-2q_{*}^{\prime}F=0.
\end{equation}

Since $F$ is rational and even, the limit
\[
c:=\lim_{x\to\infty}F(x)
\]
exists. We claim that $c\neq0$. Otherwise, for some $m\ge1$,
\[
F(x) = a_{m}x^{-2m}+O(x^{-2m-2}), \qquad a_{m}\neq0.
\]
Since
\[
q_{*}(x)=-n(n+1)+O(x^{-2}) \qquad(x\to\infty),
\]
the leading term in the left-hand side of
\eqref{eq, common limiting third order} would be
\[
-8mn(n+1)a_{m}x^{-2m-1},
\]
a contradiction. Hence
\begin{equation}
\label{eq, c nonzero}c\neq0.
\end{equation}

Define
\[
\mathcal{Q}_{*}[F] := \frac12FF^{\prime\prime}-\frac14F^{\prime2 }-q_{*}%
F^{2},
\]
which is a constant. By \eqref{eq, c nonzero},
\begin{equation}
\label{eq, limiting Q nonzero}\mathcal{Q}_{*}[F] = \lim_{x\to\infty
}\mathcal{Q}_{*}[F] = n(n+1)c^{2} \neq0.
\end{equation}
The scaling $z=px$ gives
\[
p^{4n+2}Q_{n,p}^{(0)}(A_{0}(p)) = \frac12F_{p}F_{p}^{\prime\prime}%
-\frac14F_{p}^{\prime2 }- p^{2}q_{n}^{(0)}(px;p,A_{0}(p))F_{p}^{2}.
\]
Passing to the limit and using \eqref{eq, limiting Q nonzero}, we obtain
\begin{equation}
\label{eq, asymptotic Q}\lim_{p\to0} p^{4n+2}Q_{n,p}^{(0)}(A_{0}(p)) =
n(n+1)c^{2} \neq0.
\end{equation}
In particular,
\[
Q_{n,p}^{(0)}(A_{0}(p))\not \equiv 0.
\]

It remains to show that $\delta_{n}\neq0$. Suppose, to the contrary, that
$\delta_{n}=0$. Then \eqref{eq, common limiting F} reduces to
\[
F(x)=P(x^{-2}),
\]
where $P(t)$ is a monic polynomial in $t$ of degree $n$.

Set
\[
\widetilde{q}(t) := n(n+1)(t-1)+\frac{3t^{2}}{(1-t)^{2}}.
\]
Then $P(t)$ satisfies
\begin{equation}
2t^{3}P^{\prime\prime\prime2}P^{\prime\prime}+ \bigl(6t-2\widetilde{q}%
(t)\bigr)P^{\prime}- \widetilde{q}^{\prime}(t)P = 0.\label{P's ode}%
\end{equation}
After clearing denominators, we obtain
\begin{equation}
\begin{aligned} {}& 2t^3(t-1)^3P''' +9t^2(t-1)^3P''\\ &\quad -2(t-1) \Bigl[ n(n+1)(t-1)^3-3t(t^2-3t+1) \Bigr]P'\\ &\quad -\Bigl[ n(n+1)(t-1)^3-6t \Bigr]P =0. \end{aligned}\label{P's ode}%
\end{equation}
Write
\begin{equation}
P(t)=t^{n}+\sum_{j=0}^{n-1}a_{j}t^{j}.\label{P(t) is poly}%
\end{equation}
Comparing coefficients in \eqref{P's ode} yields a recursive system for
$a_{0},\ldots,a_{n-1}$ with rational coefficients. Since $P$ is monic, this
recursion implies $a_{j}\in\mathbb{Q}, 0\le j\le n-1$. Therefore,
\[
P(t)\in\mathbb{Q}[t].
\]

Under the assumption $\delta_{n}=0$, $F$ is holomorphic at $x=1$. Since
\[
q_{*}(x) = \frac{3}{4(x-1)^{2}} - \frac{3}{4(x-1)} + O(1),
\]
the coefficient of $(x-1)^{-3}$ in \eqref{eq, common limiting third order}
gives
\[
F(1)=0.
\]
Hence
\[
P(1)=0, \qquad F^{\prime}(1)=-2P^{\prime}(1).
\]
Evaluating $\mathcal{Q}_{*}[F]$ at $x=1$ gives
\[
\mathcal{Q}_{*}[F] = -\frac{1}{4}F^{\prime2 }= -P^{\prime2}.
\]
Together with \eqref{eq, limiting Q nonzero}, and noting that $c=P(0)$ in the
present case, we obtain
\[
\left(  \frac{P^{\prime}(1)}{P(0)} \right) ^{2} = -n(n+1).
\]
This is impossible, since the left-hand side is a square in $\mathbb{Q}$,
whereas the right-hand side is negative. Thus $\delta_{n}\neq0.$

By Remark~\ref{remark, weight},
\begin{equation}
\label{eq, asymptotic d final}d_{n,p}(A_{0}(p)) = \delta_{n}p^{-2n-2}(1+o(1)),
\quad\text{as }p\to0.
\end{equation}
In particular,
\[
d_{n,p}(A_{0}(p))\not \equiv 0.
\]

By the classification \emph{(a-i)}--\emph{(a-iii)}, the product $y_{a_{0}%
(p),c(p)}(z)y_{a_{0}(p),c(p)}(-z)$ has a simple pole at $z=p$ in case
\emph{(a-i)}, whereas it is holomorphic there in cases \emph{(a-ii)} and
\emph{(a-iii)}. Since $\wp^{\prime}(p)\neq0$ for $p\notin E_{\tau}[2]$, it
follows that
\[
\Theta_{1}(\tau) = \left\{  p\in E_{\tau}\setminus E_{\tau}[2] \;\middle|\;
d_{n,p}(A_{0}(p))=0 \right\} .
\]
By \eqref{eq, asymptotic d final}, $d_{n,p}(A_{0}(p))$ is a nonzero elliptic
function of $p$, and hence $\Theta_{1}(\tau)$ is finite.

Likewise, Lemma~\ref{lem, theta2} gives
\[
\Theta_{2}(\tau) = \left\{  p\in E_{\tau}\setminus E_{\tau}[2] \;\middle|\;
Q_{n,p}^{(0)}(A_{0}(p))=0 \right\} .
\]
By \eqref{eq, asymptotic Q}, $Q_{n,p}^{(0)}(A_{0}(p))$ is a nonzero elliptic
function of $p$. Hence $\Theta_{2}(\tau)$ is finite as well.
\end{proof}

\begin{proof}
[Proof of Theorem \ref{thm, theta is finite}]The result follows from
Proposition~\ref{prop, theta1 theta2 finite}.
\end{proof}

\begin{proof}
[Proof of Theorem~\ref{theorem, noneven's apparent}]Since $\Theta(\tau)$ is
finite, for any $p\in\Theta(\tau)$ we may choose a sequence
\[
p_{\ell}\in E_{\tau}\setminus\bigl(E_{\tau}[2]\cup\Theta(\tau)\bigr), \qquad
p_{\ell}\to p.
\]
For each $\ell$, we have
\[
\mathbf{a}_{0}(p_{\ell})\in\overset{o}{X}_{n,p_{\ell}}(\tau).
\]
Hence, by Theorem~\ref{thm, noneven apparent for generic p}, $\mathrm{GLE}%
_{n}^{(1)}(p_{\ell},T,\tau)$ is log-free for every $T\in\mathbb{C}$.
Equivalently, the polynomial $f_{p_{\ell}}(T)$ defined in \eqref{def, fp(T)}
satisfies
\[
f_{p_{\ell}}(T)\equiv0.
\]

Since the coefficients of $f_{p}(T)$ belong to $\mathbb{Q}\bigl[\wp
(p),\wp^{\prime}(p),e_{j}(\tau)\bigr],$ they depend continuously on $p$ away
from $E_{\tau}[2]$. Letting $\ell\to\infty$, we obtain
\[
f_{p}(T)=\lim_{\ell\to\infty}f_{p_{\ell}}(T)\equiv0.
\]
Thus $\mathrm{GLE}_{n}^{(1)}(p,T,\tau)$ is log-free for every $T\in\mathbb{C}$.
\end{proof}

\section{Spectral Geometry of the Non-Even Component}

\label{spectral theory non-even}

Fix $\tau\in\mathbb{H}$ and $p\in E_{\tau}\setminus E_{\tau}[2]$, and let
$\mathbf{T}$ $\mathbf{=}$ $\left(  T_{1},T_{2}\right)  $ $\mathbf{\in}$
$V_{n,p}(\tau)$. The corresponding monodromy representation is a group
homomorphism
\[
\rho_{\tau,p}(\mathbf{T}):\pi_{1}(E_{\tau}\setminus\{0,\pm p\})\rightarrow
SL(2,\mathbb{C}).
\]
The local exponents at $z=0$ are $-n$ and $n+1$, while those at $z=\pm p$ are
$-1/2$ and $3/2$. Consequently, the corresponding local monodromy matrices are
$N_{0}$ $=$ $Id_{2\times2}$ and $N_{\pm p}$ $=$ $-Id_{2\times2}.$ Let
$z_{0}\not \in $ $\ \{0,\pm p\}+\Lambda_{\tau}$ be a base point such that
$\{0,\pm p\}+\Lambda_{\tau}$ $\not \in $ $z_{0}+\mathbb{R}$ and $z_{0}%
+\tau\cdot\mathbb{R}$. Let $Y(z;z_{0})$ be a fundamental system of solutions
of equation (\ref{GLEn}) in a neighborhood of $z_{0}$.

For $i=1,2$, let
\[
\ell_{i}:z\rightarrow z+\omega_{i},
\]
denote the standard generators of $\pi_{1}(E_{\tau})$ where $\omega_{1}=1$ and
$\omega_{2}=\tau$. The monodromy matrices%
\begin{equation}
M_{i}(\mathbf{T};p):=\rho_{\tau,p}(\mathbf{T})(\ell_{i})\in SL(2,\mathbb{C})
\label{M}%
\end{equation}
are determined by analytic continuation via%
\begin{equation}
Y(z+\omega_{i})=Y(z)M_{i}(\mathbf{T};p). \label{MM}%
\end{equation}
These matrices satisfy the monodromy relation
\begin{equation}
M_{1}M_{2}M_{1}^{-1}M_{2}^{-1}=N_{p}N_{0}N_{-p}=Id_{2\times2}, \label{abelian}%
\end{equation}
which implies that the representation $\rho_{\tau,p}(\mathbf{T})$ is
\textbf{abelian}. There are two cases.\textit{\medskip}

\noindent(i) \textbf{Completely reducible case}:\textit{\ }$M_{i}%
(\mathbf{T};p),$ $i$ $=$ $1,2,$ are simultaneously diagonalizable. More
precisely,
\begin{equation}
M_{1}(\mathbf{T};p)=\left(
\begin{array}
[c]{cc}%
e^{-2\pi is} & 0\\
0 & e^{2\pi is}%
\end{array}
\right)  ,\text{ }M_{2}(\mathbf{T};p)=\left(
\begin{matrix}
e^{2\pi ir} & 0\\
0 & e^{-2\pi ir}%
\end{matrix}
\right)  , \label{m1}%
\end{equation}
for some
\[
\left(  r,s\right)  \in\mathbb{C}^{2}/\mathbb{Z}^{2}.
\]
The pair $\left(  r,s\right)  $ is referred to as the monodromy data in this
case.\textit{\medskip}

\noindent(ii) \textbf{Non-completely reducible case}: $M_{i}(\mathbf{T};p),$
$i$ $=$ $1,2,$ are not simultaneously diagonalizable. Instead, after
conjugation they admit the normalized form:
\begin{equation}
M_{1}(\mathbf{T};p)=\varepsilon_{1}\left(
\begin{array}
[c]{cc}%
1 & 0\\
1 & 1
\end{array}
\right)  \text{, }M_{2}(\mathbf{T};p)=\varepsilon_{2}\left(
\begin{matrix}
1 & 0\\
\mathcal{D} & 1
\end{matrix}
\right)  ,\text{ } \label{m2}%
\end{equation}
where%
\[
\mathcal{D}\in\mathbb{C\cup\{\infty\}},
\]
and the pair
\[
\left(  \varepsilon_{1},\varepsilon_{2}\right)  \in
\{(+,+),(+,-),(-,+),(-,-)\}.
\]
These four sign choices correspond to the four half-periods
\[
r+s\tau=\frac{\omega_{k}}{2}\text{, }k=0,1,2,3,
\]
or equivalently,
\[
(r,s)\in\{(0,0),(1/2,0),(0,1/2),(1/2,1/2)\}.
\]

When $\mathcal{D=\infty}$, the normalized form is understood as
\begin{equation}
M_{1}(\mathbf{T};p)=\varepsilon_{1}\left(
\begin{array}
[c]{cc}%
1 & 0\\
0 & 1
\end{array}
\right)  \text{, }M_{2}(\mathbf{T};p)=\varepsilon_{2}\left(
\begin{matrix}
1 & 0\\
1 & 1
\end{matrix}
\right)  . \label{m3}%
\end{equation}
\textit{\medskip}Here, $\mathcal{D}$ serves as the monodromy data for this
case.\textit{\medskip}

We now construct the spectral curve of the non-even component, parallel to the
even theory developed in \cite{Chen-Kuo-Lin-Lame I}.

The associated third symmetric product equation is
\begin{equation}
\Phi^{\prime\prime\prime}(z)-4q_{n}^{(1)}(z;p,T)\,\Phi^{\prime}(z)-2q_{n}%
^{(1)\prime}(z;p,T)\,\Phi(z)=0,\text{ on }E_{\tau}\text{.}
\label{3rd ode, noneven}%
\end{equation}
A direct computation shows that if $y_{1}(z)$ and $y_{2}(z)$ are two solutions
of GLE$_{n}^{(1)}(p,T,\tau)$, then their product $\Phi(z)=y_{1}(z)y_{2}(z)$
satisfies the above third-order equation \eqref{3rd ode, noneven}.

By Theorem \ref{apparent,main thm}, the equation GLE$_{n}^{(1)}(p,T,\tau)$ is
log-free at all its singularities. Hence its monodromy representation is
abelian, so there exists a common Bloch solution $y_{1}(z;T)$ satisfying
\begin{equation}
y_{1}(z+\omega_{j};T)=\lambda_{j}(T)y_{1}(z;T), \label{y1 and eigenvalue}%
\end{equation}
where $\lambda_{j}(T)$ $\in\mathbb{C}$ is the eigenvalue of the monodromy
matrix $M_{j}(T)$, for $j=1,2$. We define the associated monodromy data
\[
(r(T),s(T))\in\mathbb{C}^{2}/\mathbb{Z}^{2}%
\]
by
\begin{equation}
\lambda_{1}(T)=e^{-2\pi is(T)},\quad\lambda_{2}(T)=e^{2\pi ir(T)}.
\label{r(T),s(T),def}%
\end{equation}


In the completely reducible case, it follows from (\ref{m1}) that there exists
another solution $y_{2}(z;T)$, linearly independent of $y_{1}(z;T)$,
satisfying
\begin{equation}
y_{2}(z+\omega_{j};T)=\lambda_{j}^{-1}(T)y_{2}(z;T),\quad j=1,2.
\label{y2 and eigenvalue}%
\end{equation}
Consequently, their product $y_{1}(z;T)y_{2}(z;T)$ is an elliptic solution of
equation \eqref{3rd ode, noneven}.

In the non-completely reducible case, it follows from (\ref{m2}) that
$y_{1}(z;T)$ may be normalized so that%

\begin{equation}
y_{1}(z+\omega_{j};T)=\lambda_{j}(T)y_{1}(z;T)=\varepsilon_{j}y_{1}%
(z;T),\quad\varepsilon_{j}\in\{\pm1\}. \label{y1,non-com reduc}%
\end{equation}
Consequently, $y_{1}^{2}(z;T)$ is an elliptic solution of \eqref{3rd ode, noneven}.

\begin{theorem}
\label{thm, 3rd's elliptic solution} Up to a nonzero scalar multiple, there
exists a unique nontrivial elliptic solution to \eqref{3rd ode, noneven},
denoted by $\Phi_{e,p}^{(1)}(z;T)$. Moreover, $\Phi_{e,p}^{(1)}(z;T)$ depends
polynomially on $T$ and admits the explicit expansion
\begin{align}
\Phi_{e,p}^{(1)}(z;T)  &  =c_{0}\wp^{n}(z)+\sum_{k=1}^{n-1}c_{2k}\wp
^{n-k}(z)+\sum_{k=1}^{n-1}c_{2k-1}\wp^{n-k-1}(z)\cdot\wp^{\prime
}(z)\label{Phi's form}\\
&  +c_{2n-1,+}(\zeta(z+p)-\zeta(z))+c_{2n-1,-}(\zeta(z-p)-\zeta(z))+c_{2n}%
.\nonumber
\end{align}
The coefficients satisfy
\begin{equation}
c_{k} = c_{k}(T) \in\mathbb{C}\bigl(\wp(p;\tau),\wp^{\prime}(p;\tau
),e_{j}(\tau)\bigr)[T], \label{cj's coeff is rational funs}%
\end{equation}
for $k\in\{0,1,2,\cdots,2n-2\}\cup\{2n-1,\pm\}$, and
\begin{equation}
c_{2n}=c_{2n}(T)\in\mathbb{C}\bigl(\wp(p;\tau),\wp^{\prime}(p;\tau
),\zeta(p),e_{j}(\tau)\bigr)[T], \label{c2n's coeff is rational funs}%
\end{equation}
which are primitive as polynomials in $T$ (i.e., they share no nontrivial
common divisor), and $c_{2n}(T)$ is monic.

Furthermore,
\begin{equation}
\deg_{T} c_{0}=0, \label{deg c0}%
\end{equation}
\begin{equation}
\deg_{T} c_{j}\leqslant j,\ 1\leqslant j\leqslant2n-2,\quad\deg_{T}
c_{2n-1,\pm}\leqslant2n-1, \label{deg cj}%
\end{equation}
and
\begin{equation}
\deg_{T} c_{2n}=2n. \label{deg c2n}%
\end{equation}

Consequently,
\[
\deg_{T} \Phi_{e,p}^{(1)}(z;T)=2n.
\]

\end{theorem}

The degree estimates \eqref{cj's coeff is rational funs}--\eqref{deg c2n}
follow the same  coefficient recursion as in the even case. Uniqueness
requires a different argument, since $q_{n}^{(1)}(z;p,$ $T)$ is generally non-even.

\begin{lemma}
\label{lem, elliptic sol is 1 dim} The elliptic solution of
\eqref{3rd ode, noneven} is unique up to a nonzero constant multiple.
\end{lemma}

Set
\begin{equation}
\Psi_{p}(z):=\frac{\sigma(z)}{\sqrt{\sigma(z-p)\sigma(z+p)}}. \label{Psi,p}%
\end{equation}
By the transformation law of the Weierstrass-$\sigma$ function:
\begin{equation}
\sigma(z+\omega_{j};\tau)=-e^{\eta_{j}\left(  z+\frac{\omega_{j}}{2}\right)
}\sigma(z;\tau),\quad\,j=1,2, \label{sigma's trans law}%
\end{equation}
it was shown in \cite[Lemma 2.2]{Chen-Kuo-Lin-Lame I} that $\Psi_{p}(z)$
satisfies
\begin{equation}
\Psi_{p}(z+\omega_{j})=\Psi_{p}(z),\quad j=1,2, \label{Psi is elliptic}%
\end{equation}
although $\Psi_{p}(z)$ has branch points at $\pm p$. Consequently, $\Psi
_{p}(z)^{2}$ is an elliptic function.

We use the following criterion from \cite[Proposition~2.3]{Chen-Kuo-Lin-Lame
I}.

\begin{lemma}
\cite{Chen-Kuo-Lin-Lame I} \label{lem,complete iff lambda neq pm1} The
equation GLE$_{n}^{(1)}(p,T,\tau)$ is completely reducible if and only if
\begin{equation}
(\lambda_{1}(T),\lambda_{2}(T))\notin\{\pm(1,1),\pm(1,-1)\}. \label{lamda}%
\end{equation}
Equivalently, the corresponding monodromy data $(r(T),s(T))$ satisfy
\begin{equation}
(r(T),s(T))\notin\frac{1}{2}\mathbb{Z}^{2}. \label{rs}%
\end{equation}

\end{lemma}

\begin{proof}
Suppose first that \eqref{lamda} holds. Then one period monodromy matrix has
two distinct eigenvalues. Since the commutativity \eqref{M1M2=M2M1}, there are
two linearly independent common eigenfunctions of $M_{j}(T),j=1,2$. Hence the
monodromy representation is completely reducible.

For the converse, assume that the equation is completely reducible but that
\[
(\lambda_{1}(T),\lambda_{2}(T)) \in\{\pm(1,1),\pm(1,-1)\}.
\]
Then all period multipliers are $\pm1$. Thus, for a common eigenbasis
$y_{1},y_{2}$, the symmetric products
\[
y_{1}^{2},\qquad y_{1}y_{2},\qquad y_{2}^{2}
\]
have trivial period monodromy. Since these three functions span the solution
space of the associated third-order equation \eqref{3rd ode, noneven}, every
solution of that equation is elliptic.

Now choose a Frobenius basis at $z=0$,
\[
\tilde y_{1}(z)=z^{-n}(1+O(z)), \qquad\tilde y_{2}(z)=z^{n+1}(1+O(z)).
\]
Their product is therefore an elliptic solution of \eqref{3rd ode, noneven}.
Dividing by the elliptic function $\Psi_{p}^{2}$ removes the half-order
singularities at $z=\pm p$, so that
\[
H(z):=\frac{\tilde y_{1}(z)\tilde y_{2}(z)}{\Psi_{p}(z)^{2}}
\]
is elliptic and holomorphic away from $z=0$. On the other hand, the local
behavior at $z=0$ gives
\[
H(z)=-\frac{\sigma(p)^{2}}{z}+O(1).
\]
Hence $H$ would have a unique simple pole, which is impossible for an elliptic
function. The contradiction proves the necessity.
\end{proof}

By Lemma~\ref{lem,complete iff lambda neq pm1}, if the monodromy data
$(r(T),s(T))$, defined by (\ref{r(T),s(T),def}), belong to $\frac{1}%
{2}\mathbb{Z}^{2}$, then the equation GLE$_{n}^{(1)}(p,T,\tau)$ is
non-completely reducible. Consequently, its monodromy matrices are of the form \eqref{m2}.

\begin{proof}
[Proof of Lemma~3.2]Let $y_{1}(z;T)$ and $y_{2}(z;T)$ be a fundamental system
of solutions of $\mathrm{GLE}_{n}^{(1)}(p,T,\tau)$, normalized according to
either \eqref{m1} or \eqref{m2}. Then
\[
y_{1}^{2}(z;T),\qquad y_{1}(z;T)y_{2}(z;T),\qquad y_{2}^{2}(z;T)
\]
form a basis of the solution space of the third-order equation
\eqref{3rd ode, noneven}. Consider the two cases.

\medskip\noindent\emph{Case 1. The monodromy representation is completely
reducible.} By Lemma~\ref{lem,complete iff lambda neq pm1}, there exists
$j\in\{1,2\}$ such that $\lambda_{j}^{2}(T)\neq1.$ By
\eqref{y1 and eigenvalue} and \eqref{y2 and eigenvalue}, translation by
$\omega_{j}$ acts on the above basis as
\[
\begin{gathered}
y_1^2(z+\omega_j;T)
=\lambda_j^2(T)y_1^2(z;T),\\
y_1(z+\omega_j;T)y_2(z+\omega_j;T)
=y_1(z;T)y_2(z;T),\\
y_2^2(z+\omega_j;T)
=\lambda_j^{-2}(T)y_2^2(z;T).
\end{gathered}
\]

Let $\Phi(z)$ be an elliptic solution of \eqref{3rd ode, noneven} and write
\[
\Phi(z) = a\,y_{1}^{2}(z;T) +b\,y_{1}(z;T)y_{2}(z;T) +c\,y_{2}^{2}(z;T).
\]
Since $\Phi(z+\omega_{j})=\Phi(z)$, comparison of coefficients gives
\[
\bigl(\lambda_{j}^{2}(T)-1\bigr)a=0, \qquad\bigl(\lambda_{j}^{-2}%
(T)-1\bigr)c=0.
\]
Hence $a=c=0$, and therefore
\[
\Phi(z)=b\,y_{1}(z;T)y_{2}(z;T).
\]
Thus the space of elliptic solutions is one-dimensional in the completely
reducible case.

\medskip\noindent\emph{Case 2. The monodromy representation is non-completely
reducible.} By \eqref{m2}, we may normalize the fundamental system so that
\[
y_{1}(z+1;T)=\varepsilon_{1}y_{1}(z;T), \qquad y_{2}(z+1;T) = \varepsilon
_{1}\bigl(y_{1}(z;T)+y_{2}(z;T)\bigr),
\]
where $\varepsilon_{1}\in\{\pm1\}$. Since $\varepsilon_{1}^{2}=1$, we obtain
\[
\begin{gathered}
y_1^2(z+1;T)
=y_1^2(z;T),\\
y_1(z+1;T)y_2(z+1;T)
=y_1^2(z;T)+y_1(z;T)y_2(z;T),\\
y_2^2(z+1;T)
=y_1^2(z;T)
+2y_1(z;T)y_2(z;T)
+y_2^2(z;T).
\end{gathered}
\]

Let
\[
\Phi(z) = a\,y_{1}^{2}(z;T) +b\,y_{1}(z;T)y_{2}(z;T) +c\,y_{2}^{2}(z;T)
\]
be an elliptic solution of \eqref{3rd ode, noneven}. The identity
$\Phi(z+1)=\Phi(z)$ gives
\[
b+c=0, \qquad2c=0.
\]
Hence $b=c=0$, and consequently
\[
\Phi(z)=a\,y_{1}^{2}(z;T).
\]

Thus in both cases the space of elliptic solutions of \eqref{3rd ode, noneven}
is one-dimensional.
\end{proof}

We introduce the spectral polynomial $Q_{n,p}^{(1)}(T)$ and the associated
spectral curve $\Gamma_{n,p}^{(1)}(\tau)$ of $\mathrm{GLE}_{n}^{(1)}%
(p,T,\tau)$.

\begin{lemma}
\label{lem, Phi inde of P'} The elliptic solution $\Phi_{e,p}^{(1)}(z;T)$ of
\eqref{3rd ode, noneven} is polynomial in $T$. Moreover, its coefficients
belong to
\[
\mathbb{C}(\wp(z),\wp^{\prime}(z),\wp(p),e_{j}(\tau)).
\]

\end{lemma}

\begin{proof}
Write
\[
\Phi_{e,p}^{(1)}(z;T)=\sum_{k=0}^{2n}\Phi_{k}(z,p)T^{k}.
\]
Each coefficient $\Phi_{k}(z,p)$ is elliptic in $z$.

The identities
\begin{equation}
q_{n}^{(1)}(z;-p,T)=q_{n}^{(1)}(z;p,T),\quad q_{n}^{(1)}(z;p+{\omega_{k}%
},T)=q_{n}^{(1)}(z;p,T), \label{q(z;p) is elliptic in p}%
\end{equation}
imply the third-order equation \eqref{3rd ode, noneven} is invariant under
\[
p\rightarrow-p\text{ and }p\rightarrow p+\omega_{k}.
\]
Since the elliptic solution is unique up to a nonzero constant multiple by
Theorem~\ref{thm, 3rd's elliptic solution} and is normalized by the leading
coefficient $c_{2n}(T)$, it is uniquely determined. Therefore
\begin{equation}
\Phi_{e,p}^{(1)}(z;T)=\Phi_{e,-p}^{(1)}(z;T)=\Phi_{e,p+{\omega_{k}}}%
^{(1)}(z;T). \label{Phi(z,-p)=Phi(z,p)}%
\end{equation}

Comparing coefficients of $T$ shows that each $\Phi_{k}(z,p)$ is even and
elliptic in $p$, so
\[
\Phi_{k}(z,p)\in\mathbb{C}(\wp(z),\wp^{\prime}(z),\wp(p),e_{j}(\tau)).
\]

\end{proof}

\begin{example}
When $n=1$,
\[
\begin{aligned} &\Phi_{e,p}^{(1)}(z;T)=T^2-\frac{1}{2}\left(\zeta(z+p)+\zeta(z-p)-2\zeta(z)\right)T\\ &+\frac{1}{2}\wp(z)-\frac{1}{4}\frac{\wp^{\prime\prime}(p)}{\wp^{\prime}(p)}\left(\zeta(z+p)-\zeta(z-p)\right)+\frac{1}{2}\frac{\wp^{\prime\prime}(p)}{\wp^{\prime}(p)}\zeta(p)-\frac{1}{2}\wp(p). \end{aligned}
\]

\end{example}

Define
\begin{equation}
\begin{aligned} Q_{n,p}^{(1)}(z;T):=&\frac{1}{2}\Phi_{e,p}^{(1)}(z;T)\Phi_{e,p}^{(1)\prime\prime}(z;T)-\frac{1}{4}(\Phi_{e,p}^{(1)\prime}(z;T))^{2}\\ &-q_{n}^{(1)}(z;p,T)\Phi_{e,p}^{(1)2}(z;T). \end{aligned} \label{spectral poly's def}%
\end{equation}

Since $\Phi_{e,p}^{(1)}(z;T)$ satisfies \eqref{3rd ode, noneven},
\[
\frac{d}{dz}Q_{n,p}^{(1)}(z;T)\equiv0.
\]
We therefore write it simply as $Q_{n,p}^{(1)}(T)$.

\begin{proposition}
$Q_{n,p}^{(1)}(T)$ is a polynomial in $T$ of degree $4n+2$ whose coefficients
are rational functions of $\wp(p)$ and $e_{k}(\tau)$, and are independent of
$\wp^{\prime}(p)$ and $\zeta(p)$.
\end{proposition}

We call $Q_{n,p}^{(1)}(T)$ the \textit{spectral polynomial}. Its independence
of $\wp^{\prime}(p)$ contrasts with the even spectral polynomial; see~\eqref{spectral poly for even n=1}.

\begin{example}
For $n=1$, the spectral polynomial $Q_{1,p}^{(1)}(T)$ is given by
\[
Q_{1,p}^{(1)}(T)=\bigl(T^{2}-2\wp(p)-e_{1}\bigr)\bigl(T^{2}-2\wp
(p)-e_{2}\bigr)\bigl(T^{2}-2\wp(p)-e_{3}\bigr).
\]

\end{example}

The associated \textit{spectral curve} is defined by
\begin{equation}
\Gamma_{n,p}^{(1)}(\tau):=\{(T,C)\in\mathbb{C}^{2}|\,C^{2}=Q_{n,p}^{(1)}(T)\}.
\label{spectral curve,def}%
\end{equation}
The curve admits a natural hyperelliptic involution%
\[
P=(T,C)\rightarrow P^{\ast}:=(T,-C).
\]
The branch points of $\Gamma_{n,p}^{(1)}(\tau)$ correspond to the zeros of the
spectral polynomial, $Q_{n,p}^{(1)}(T)=0.$ Equivalently, these are the points
of the spectral curve where $C=0$.

For $P=(T,C)\in\Gamma_{n,p}^{(1)}(\tau)$, set
\begin{equation}
\phi(P;z):=\frac{iC+\frac{1}{2}\Phi_{e,p}^{(1)\prime}(z;T)}{\Phi_{e,p}%
^{(1)}(z;T)},\quad z\in\mathbb{C}. \label{phi's,def}%
\end{equation}
It satisfies the Riccati equation
\begin{equation}
\phi^{\prime}(P;z)=q_{n}^{(1)}(z;p,T)-\phi^{2}(P;z). \label{phi,Riccati equ}%
\end{equation}

Fix a base point $z_{0}\in\mathbb{C}\setminus\{0,\pm p\}$. For each
$P\in\Gamma_{n,p}^{(1)}(\tau)$, we define the Baker--Akhiezer function by
\begin{equation}
\psi(P;z,z_{0}):=\exp\left(  \int_{z_{0}}^{z}\phi(P;\xi)d\xi\right)  ,\quad
z\in\mathbb{C}, \label{BA function,def}%
\end{equation}
where the integration path is chosen to avoid the singularities of the
meromorphic function $\phi(P;\xi)$.

By the Riccati equation \eqref{phi,Riccati equ}, it is easy to verify that
$\psi(P;z,z_{0})$ solves GLE$_{n}^{(1)}(p,T(P),\tau)$, where $T(P)$ is the
$T$-coordinate of the point $P=(T(P),$ $C(P))$ $\in\Gamma_{n,p}^{(1)}(\tau)$.
Thus every point $P\in$ $\Gamma_{n,p}^{(1)}(\tau)$ corresponds to a log-free
equation GLE$_{n}^{(1)}(p,T(P),\tau)$. Since $T(P^{\ast})=T(P)$, both
$\psi(P;z,z_{0})$ and $\psi(P^{\ast};z,$ $z_{0})$ are solutions of the same
GLE$_{n}^{(1)}(p,T(P),\tau)$. Consequently, the point $P$ and its dual
$P^{\ast}$ represent the same GLE$_{n}^{(1)}(p,T(P),\tau)$.

As in the even-symmetry case, $\psi(P;z,z_{0})$ and $\psi(P^{\ast};z,z_{0})$
satisfy
\begin{equation}
\psi(P;z,z_{0})\psi(P^{\ast};z,z_{0})=\frac{\Phi_{e,p}^{(1)}(z;T(P))}%
{\Phi_{e,p}^{(1)}(z_{0};T(P))}, \label{two psi's product}%
\end{equation}
and
\begin{equation}
W(\psi(P;z,z_{0}),\psi(P^{\ast};z,z_{0}))=\frac{2iC(P)}{\Phi_{e,p}^{(1)}%
(z_{0};T(P))}. \label{wronskian of psi}%
\end{equation}
See \cite{Kuo}. Since changing $z_{0}$ only rescales $\psi$, we henceforth
suppress the base point.

Equation \eqref{two psi's product} gives the following symmetries.

\begin{proposition}
\label{prop, sym} (i) $q_{n}^{(1)}(z;p,-T)=q_{n}^{(1)}(-z;p,T).\medskip$

(ii) $\Phi_{e,p}^{(1)}(z;-T)=\Phi_{e,p}^{(1)}(-z;T).\medskip$

(iii) $Q^{(1)}_{n,p}(T)=Q^{(1)}_{n,p}(-T).\medskip$

Consequently, $Q^{(1)}_{n,p}(T)$ is an even polynomial in $T$ with
coefficients that are rational functions of $\wp(p)$ and $e_{k}(\tau)$. Hence
there exists a polynomial%
\[
\widehat{Q}_{n,p}^{(1)}(X)\in\mathbb{C}(\wp(p),e_{k}(\tau))[X]
\]
such that
\begin{equation}
Q^{(1)}_{n,p}(T)=\widehat{Q}_{n,p}^{(1)}(T^{2}). \label{tilde Q is even in T}%
\end{equation}

\end{proposition}

\begin{proof}
Assertion (i) is immediate from (\ref{potential, noneven}). Let $P=(T,C)\in
\Gamma_{n,p}^{(1)}(\tau)$. By \eqref{two psi's product}, we have
\begin{equation}
\psi(P;-z)\,\psi(P^{\ast};-z)=\frac{\Phi_{e,p}^{(1)}(-z;T)}{\Phi_{e,p}%
^{(1)}(z_{0};T)}, \label{two psi product -z}%
\end{equation}
is an elliptic solution to \eqref{3rd ode, noneven} associated with potential
$q_{n}^{(1)}(z;p,-T)$. Hence both $\Phi_{e,p}^{(1)}(z;-T)$ and $\Phi
_{e,p}^{(1)}(-z;T)$ are elliptic solutions of \eqref{3rd ode, noneven}
corresponding to the same potential $q_{n}^{(1)}(z;p,-T)$. Since they are
monic polynomials in $T$, uniqueness in
Theorem~\ref{thm, 3rd's elliptic solution} yields
\[
\Phi_{e,p}^{(1)}(z;-T)=\Phi_{e,p}^{(1)}(-z;T).
\]
This proves (ii).

Finally, (iii) follows immediately from (i), (ii), and
(\ref{spectral poly's def}).
\end{proof}

\begin{theorem}
\label{thm, com red iff BA are indep} Let $P=(T,C)\in\Gamma_{n,p}^{(1)}(\tau
)$. The Baker--Akhiezer functions $\psi(P;z)$ and $\psi(P^{\ast};z)$ are
linearly independent if and only if
\[
Q_{n,p}^{(1)}(T)\neq0.
\]
Equivalently, they are linearly independent if and only if $P$ is not a branch
point of the spectral curve $\Gamma_{n,p}^{(1)}(\tau)$.
\end{theorem}

\begin{theorem}
\label{thm,Q(T) neq 0 iff com.red.} For $P=(T,C)\in\Gamma_{n,p}^{(1)}(\tau)$,
$\mathrm{GLE}_{n}^{(1)}(p,T,\tau)$ is completely reducible if and only if
$Q_{n,p}^{(1)}(T)\neq0$. Hence the branch points of the spectral curve
$\Gamma_{n,p}^{(1)}(\tau)$ correspond exactly to the values of $T$ for which
the monodromy representation is non-completely reducible.
\end{theorem}

\begin{proof}
The sufficiency follows directly from Theorem
\ref{thm, com red iff BA are indep}. It remains to prove the necessity. Assume
that GLE$_{n}^{(1)}(p,T,\tau)$ is completely reducible with $Q_{n,p}^{(1)}(T)$
$=0$. By Theorem \ref{thm, com red iff BA are indep}, two Baker--Akhiezer
functions $\psi(P;z)$ and $\psi(P^{\ast};z)$ are linearly dependent; hence
they share a common zero $z_{1}$.

Since GLE$_{n}^{(1)}(p,T,\tau)$ is completely reducible, there exists a basis
of solutions $y_{1}(z;T)$ and $y_{2}(z;T)$ such that the product
$y_{1}(z;T)\cdot y_{2}(z;T)$ is an elliptic solution of
\eqref{3rd ode, noneven}. By Theorem \ref{thm, 3rd's elliptic solution} and
\eqref{two psi's product}, we obtain
\begin{equation}
y_{1}(z;T)y_{2}(z;T)=\Phi_{e,p}^{(1)}(z;T)=\psi(P;z)\psi(P^{\ast};z)
\label{y1y2=product of two psi}%
\end{equation}
up to a nonzero multiple. Since
\[
\psi(P;z_{1})=\psi(P^{\ast};z_{1})=0,
\]
it follows from \eqref{y1y2=product of two psi} that  either one of $y_{1}$
and $y_{2}$ has a zero of order at least two at $z_{1}$, or both vanish at
$z_{1}$. In either case,
\[
W(y_{1},y_{2})(z_{1})=0,
\]
contradicting the linear independence of $y_{1}$ and $y_{2}$.

This contradiction gives
\[
Q_{n,p}^{(1)}(T)\neq0,
\]
proving necessity.
\end{proof}

In either monodromy type, the common monodromy eigenfunction may be normalized
as
\begin{equation}
\psi(P;z)=y_{1}(z;T). \label{1111}%
\end{equation}
Indeed, when $Q_{n,p}^{(1)}(T)\neq0$, the pair $\psi(P)$ and $\psi(P^{*})$
gives the two common eigenfunctions; when $Q_{n,p}^{(1)}(T)=0$, they coincide.

For each $P=(T,C)\in\Gamma_{n,p}^{(1)}(\tau)$, define $\lambda_{j}%
(P)\in\mathbb{C}$ by
\[
\lambda_{j}(P):=\lambda_{j}(T),
\]
where $\lambda_{j}(T),j=1,2$, is the eigenvalue of $M_{j}(T)$ corresponding to
$y_{1}(z;T)$, satisfying \eqref{y1 and eigenvalue}. Then (\ref{1111}) implies
that
\begin{equation}
\psi(P;z+\omega_{j})=\lambda_{j}(P)\psi(P;z),\ j=1,2. \label{BA is 2nd}%
\end{equation}
Moreover,
\begin{equation}
\lambda_{j}(P^{\ast})=\lambda_{j}(P)^{-1}. \label{0120equ1}%
\end{equation}

For notational consistency, write $(r(P),s(P))$ instead of $(r(T),s(T))$.
Namely,
\begin{equation}
\lambda_{1}(P)=e^{-2\pi is(P)}\quad\text{and}\quad\lambda_{2}(P)=e^{2\pi
ir(P)}. \label{def,r(P),s(P)}%
\end{equation}
It follows from \eqref{0120equ1} that
\begin{equation}
r(P^{\ast})=-r(P)\quad\text{and}\quad s(P^{\ast})=-s(P). \label{0120equ2}%
\end{equation}

Relation \eqref{BA is 2nd} shows that $\psi(P;z)$ is an elliptic function of
the second kind. Consequently, up to a nonzero constant, it admits a
Hermite-Halphen type representation of the form%
\begin{equation}
\psi(P;z)=y_{\mathbf{a}(P),c(P)}(z)=\dfrac{e^{c(P)z}\,\prod\limits_{i=1}%
^{n+1}\sigma(z-a_{i}(P))}{~\sigma^{n}(z)\sqrt{\sigma(z+p)\,\sigma(z-p)}\,},
\label{Ba is given by a(P) and c}%
\end{equation}
where $y_{\mathbf{a}(P),c(P)}(z)$ is given by \eqref{Hermite noneven}, with
$c(P)\in\mathbb{C}$ and the zero divisor
\[
\mathbf{a}(P)=\{a_{1}(P),\cdots,a_{n+1}(P)\}\in\operatorname{Sym}^{n+1}%
E_{\tau}^{\times}.
\]

Since the product $\psi(P;z)\psi(P^{\ast};z)$ is elliptic by
\eqref{two psi's product}, it follows that
\[
c(P^{\ast})=-c(P).
\]
Moreover, if the $C$-coordinate of the point $P=(T,$ $C)$ satisfies $C\neq0$,
then \eqref{wronskian of psi} implies that $\psi(P;z)$ and $\psi(P^{\ast};z)$
are linearly independent, and hence
\begin{equation}
\mathbf{a}(P)\cap\mathbf{a}(P^{\ast})=\emptyset.
\label{two a's intersection is empty}%
\end{equation}
Conversely, \eqref{two a's intersection is empty} implies that the Wronskian
\[
W(\psi(P;z,z_{0}),\psi(P^{\ast};z,z_{0}))\neq0,
\]
which in turn yields the independence of $\psi(P;z)$ and $\psi(P^{\ast};z)$.

\begin{theorem}
\label{thm,criterion for} Let $P=(T,C)\in\Gamma_{n,p}^{(1)}(\tau)$. The
following statements are equivalent:

(i) GLE$_{n}^{(1)}(p,T,\tau)$ is completely reducible.\medskip

(ii) $Q_{n,p}^{(1)}(T)\neq0$.\medskip

(iii) $(r(P),s(P))$ $\notin\frac{1}{2}\mathbb{Z}^{2}$.\medskip

(iv) $\mathbf{a}(P)\cap\mathbf{a}(P^{\ast})=\emptyset.$
\end{theorem}

We now define the addition map $\sigma_{n,p}^{(1)}$ of $\mathrm{GLE}_{n}%
^{(1)}(p,T,\tau)$. Recall that $Y_{n,p}^{(1)}(\tau)$ is defined in
\eqref{Yn's def} and characterized by \eqref{Yn,p, apparent}. For each
$P\in\Gamma_{n,p}^{(1)}(\tau)$, there exists a unique $\mathbf{a}%
(P)\in\operatorname{Sym}^{n+1}E_{\tau}^{\times}$ such that the Baker--Akhiezer
function $\psi(P;z)$ of GLE$_{n}^{(1)}(p,T,\tau)$ satisfies
\eqref{Ba is given by a(P) and c}. Namely, $\mathbf{a}(P)$ is the zero divisor
of $\psi(P;z)$, that is,
\begin{equation}
\mathbf{a}(P)\in Y_{n,p}^{(1)}(\tau). \label{a(P) in Yn,p}%
\end{equation}

Define a map
\[
i_{n,p}:\Gamma_{n,p}^{(1)}(\tau)\mapsto\operatorname{Sym}^{n+1}E_{\tau}%
\]
by
\begin{equation}
i_{n,p}(P):=\mathbf{a}(P)=\{a_{1}(P),\cdots,a_{n+1}(P)\}\in\operatorname{Sym}%
^{n+1}E_{\tau}.
\end{equation}

The above discussion implies that $i_{n,p}$ is well-defined.

\begin{proposition}
\label{prop, inp is an embedding} The map $i_{n,p}$ is an embedding from
$\Gamma_{n,p}^{(1)}(\tau)$ into $\operatorname{Sym^{n+1}}E_{\tau}$. Moreover,
\begin{equation}
i_{n,p}(\Gamma_{n,p}^{(1)}(\tau))=Y_{n,p}^{(1)}(\tau).
\label{image of Gamma n,p is Yn,p}%
\end{equation}

\end{proposition}

\begin{proof}

The embeddedness follows from \cite[Proposition~3.7]{Chen-Kuo-Lin-Lame I}. It
suffices to prove \eqref{image of Gamma n,p is Yn,p}.

By \eqref{a(P) in Yn,p},
\[
i_{n,p}\bigl(\Gamma_{n,p}^{(1)}(\tau)\bigr)\subset Y_{n,p}^{(1)}(\tau).
\]

Conversely, let $\mathbf{a}\in Y_{n,p}^{(1)}(\tau)$. As discussed in
Section~\ref{Log-free variety}, $\mathbf{a}$ satisfies one of the three
conditions (a-i)--(a-iii). Without loss of generality, assume that
$\mathbf{a}$ satisfies (a-i). Then by Theorem~\ref{thm, ya solves noneven ode}%
, there exists $c\in\mathbb{C}$ such that $y_{\mathbf{a},c}(z)$ is a solution
of the associated GLE$_{n}^{(1)}(p,T,\tau)$, where $c$ and $T$ are uniquely
determined by \eqref{c in terms of aj} and \eqref{T in terms of aj},
respectively. Since $y_{\mathbf{a},c}(z)$ is elliptic of second kind, it must
coincide with $\psi(P;z)$ or $\psi(P^{\ast};z)$, for $P=(T,C)\in\Gamma
_{n,p}^{(1)}(\tau)$. Hence $\mathbf{a=a}(P)$ or $\mathbf{a=a}(P^{\ast})$, and
therefore
\[
i_{n,p}\bigl(\Gamma_{n,p}^{(1)}(\tau)\bigr)\supseteq Y_{n,p}^{(1)}(\tau).
\]

Consequently,
\[
i_{n,p}\bigl(\Gamma_{n,p}^{(1)}(\tau)\bigr)=Y_{n,p}^{(1)}(\tau).
\]

\end{proof}

Since the spectral polynomial $Q_{n,p}^{(1)}(T)$ has even degree $\deg
_{T}Q_{n,p}^{(1)}(T)=4n+2$, the spectral curve $\Gamma_{n,p}^{(1)}(\tau)$
admits two points at infinity $\infty_{\pm}^{(1)}$. Its compactification
\begin{equation}
\overline{\Gamma_{n,p}^{(1)}(\tau)}=\Gamma_{n,p}^{(1)}(\tau)\cup\{\infty_{\pm
}^{(1)}\} \label{closure of Gamma n,p}%
\end{equation}
is an algebraic curve.

To extend $i_{n,p}$ to the compactification, we study the closure
$\overline{Y_{n,p}^{(1)}(\tau)}$. The following is proved as in \cite[Theorem
4.1]{Chen-Kuo-Lin-Lame I}.

\begin{lemma}
\label{lem, closure of Yn,p} Let $n\in\mathbb{Z}_{>0}$ and fix $p\in E_{\tau
}\setminus E_{\tau}[2]$. Then
\[
\overline{Y_{n,p}^{(1)}(\tau)}=Y_{n,p}^{(1)}(\tau)\cup\{\infty_{\pm}(p)\},
\]
where
\begin{align}
\infty_{+}(p)  &  :=(\overbrace{0,\cdots,0}^{n-1},p,-p), \label{infty point 1}%
\\
\infty_{-}(p)  &  :=(\overbrace{0,\cdots,0}^{n+1}). \label{infty point 2}%
\end{align}
Moreover, $\overline{Y_{n,p}^{(1)}(\tau)}$ is smooth at $\infty_{\pm}(p)$.
\end{lemma}

By Lemma~\ref{lem, closure of Yn,p}, the embedding $i_{n,p}$ has a natural
extension (still denoted by $i_{n,p}$) to
\[
i_{n,p}: \overline{\Gamma_{n,p}^{(1)}(\tau)}\to\overline{Y_{n,p}^{(1)}(\tau)}
\]
by setting
\begin{equation}
i_{n,p}(\infty_{+}^{(1)})=\infty_{+}(p),\text{ and }i_{n,p}(\infty_{-}%
^{(1)})=\infty_{-}(p).
\end{equation}

Define the addition map
\[
\sigma_{n,p}^{(1)}:\overline{\Gamma_{n,p}^{(1)}(\tau)}\to E_{\tau}%
\]
defined by%
\begin{equation}
\sigma_{n,p}^{(1)}(P):=%
\begin{cases}
\sum_{i=1}^{n+1}a_{i}(P),\  & P\in\Gamma_{n,p}^{(1)}(\tau),\\
0, & P\in\{\infty_{\pm}^{(1)}\},
\end{cases}
\label{def, add map}%
\end{equation}
which is the composition of $i_{n,p}$ and the natural addition map
\[
\{a_{1},\cdots,a_{n+1}\}\to\sum_{i=1}^{n+1}a_{i}.
\]

It follows that the addition map $\sigma_{n,p}^{(1)}$ is a finite morphism
from $\overline{\Gamma_{n,p}^{(1)}(\tau)}$ to $E_{\tau}$ satisfying
\[
\sigma_{n,p}^{(1)}(\infty_{\pm}^{(1)})=0.
\]

Hence the degree of $\sigma_{n,p}^{(1)}$
\[
\deg\sigma_{n,p}^{(1)}=\#\,\sigma_{n,p}^{(1)-1}(z), \qquad z\in E_{\tau},
\]
is well-defined.

\begin{theorem}
\label{thm, sigma's degree} Let $n\in\mathbb{Z}_{>0}$ and $p\in E_{\tau
}\setminus E_{\tau}[2]$. Then
\begin{equation}
\deg\sigma_{n,p}^{(1)}=n(n+1). \label{degree of sigma}%
\end{equation}

\end{theorem}

Theorem~\ref{thm, sigma's degree} also follows immediately from the global
degree formula of~\cite{Chou-Wang-Wu-II} and the even-component
degree~\cite{Chen-Kuo-Lin-Lame I}:
\[
2n(n+1)+1 = \bigl(n(n+1)+1\bigr)
+ \deg\sigma_{n,p}^{(1)}.
\]
Section~~\ref{Section, Degree of the Non-Even Addition Map} gives an
independent componentwise proof, without using the global degree formula.

\section{Degree of the Non-Even Addition Map}

\label{Section, Degree of the Non-Even Addition Map}


Since the projection $\pi:\overline{\Gamma_{n,p}^{(1)}(\tau)}\mapsto
\overline{\hat{\Gamma}_{n,p}^{(1)}(\tau)}$ is degree two, the associated
addition maps satisfy
\begin{equation}
\deg\sigma_{n,p}^{(1)}=2\deg\hat{\sigma}_{n,p}^{(1)}.\label{deg1=2deg hat}%
\end{equation}
By \eqref{deg1=2deg hat}, Theorem~\ref{thm, sigma's degree} is equivalent to:

\begin{theorem}
\label{thm, degree of hat sigma in appendix} Let $n\in\mathbb{Z}_{>0}$ and
$p\in E_{\tau}\setminus E_{\tau}[2]$. Then
\[
\deg\hat{\sigma}_{n,p}^{(1)}=\frac{n(n+1)}{2}.
\]

\end{theorem}

Let $P=(T,C)\in\Gamma_{n,p}^{(1)}(\tau)$. Recall $(r(P),s(P))$ defined in
\eqref{def,r(P),s(P)}, and the Hermite-Halphen ansatz $y_{\mathbf{a}%
(P),c(P)}(z)$ given by \eqref{Hermite noneven}, where
\[
\mathbf{a}(P)=\{a_{1}(P),\ldots,a_{n+1}(P)\}
\]
and $c(P)$ is given by \eqref{c in terms of aj} or \eqref{c in terms of aj,2}.

\begin{lemma}
\label{lem, relation bet r(P),s(P) and aj} Suppose $y_{\mathbf{a}(P),c(P)}(z)$
given by \eqref{Hermite noneven} is the Hermite-Halphen solution of the
equation $\mathrm{GLE}_{n}^{(1)}(p,T,\tau)$. Then one has
\begin{equation}%
\begin{cases}
r(P)+s(P)\tau=\sum_{j=1}^{n+1}a_{j}(P)=\sigma_{n,p}^{(1)}(P),\\[5pt]%
\\
r(P)\,\eta_{1}(\tau)+s(P)\,\eta_{2}(\tau)=c(P).
\end{cases}
\label{algebraic relations between ai,r(P),s(P),c(P)}%
\end{equation}

\end{lemma}

\begin{proof}
This follows directly from the transformation law \eqref{sigma's trans law} of
the Weierstrass sigma function applied to the ansatz \eqref{Hermite noneven}.
\end{proof}

Define
\[
\kappa:\Gamma_{n,p}^{(1)}(\tau)\rightarrow\mathbb{C}\cup\{\infty\},
\]%
\begin{equation}
\kappa(P):=\zeta\left(  \sigma_{n,p}^{(1)}(P)\right)  -c(P),
\label{kappa's def}%
\end{equation}
where $c(P)$ is defined in \eqref{c in terms of aj} or \eqref{c in terms of aj,2}.

For $\left(  r,s\right)  $ $\in$\textit{ }$\mathbb{C}^{2}\setminus\frac{1}%
{2}\mathbb{Z}^{2}$, set
\begin{equation}
Z(r,s,\tau):=\zeta(r+s\tau)-r\eta_{1}(\tau)-s\eta_{2}(\tau), \label{zrs}%
\end{equation}
This is the classical Hecke premodular function \cite{Heck}.

\begin{proposition}
\label{prop, kappa is Zrs} With the above notations, we have
\begin{equation}
\kappa(P)=Z(r(P),s(P),\tau). \label{kappa(P)'s express}%
\end{equation}

\end{proposition}

\begin{proof}
By Lemma~\ref{lem, relation bet r(P),s(P) and aj}, we obtain
\begin{align*}
\kappa(P)  &  =\zeta(\sigma_{n,p}^{(1)}(P))-c(P)\\
&  =\zeta(r(P)+s(P)\tau)-(r(P)\eta_{1}(\tau)+s(P)\eta_{2}(\tau))\\
&  =Z(r(P),s(P),\tau).
\end{align*}

\end{proof}

\begin{theorem}
\label{thm, P and P' rational in T} Let $P=(T,C)\in\Gamma_{n,p}^{(1)}(\tau)$.
Then there exist coprime polynomials $P_{2j-1}(T,p;\tau)$ and$\,P_{2j}%
(T,p;\tau)$ in the variable $T$, whose coefficients are polynomials in
$\wp(p),$ $\wp^{\prime}(p)$ and $e_{k}(\tau)$, for $j=1,2,3$; that is,%
\[
P_{2j-1}(T,p;\tau),\,P_{2j}(T,p;\tau)\in\mathbb{C}[\wp(p),\wp^{\prime
}(p),e_{k}(\tau)][T],
\]
such that%
\begin{equation}
\wp(\sigma_{n,p}^{(1)}(P))=\frac{P_{1}(T,p;\tau)}{P_{2}(T,p;\tau)},
\label{P(sigma) in T}%
\end{equation}

\begin{equation}
\wp^{\prime}(\sigma_{n,p}^{(1)}(P))=\frac{P_{3}(T,p;\tau)}{P_{4}(T,p;\tau
)}\sqrt{-Q_{n,p}^{(1)}(T)}, \label{P'(sigma) in T}%
\end{equation}
and%
\begin{equation}
\kappa(P)=\frac{P_{5}(T,p;\tau)}{P_{6}(T,p;\tau)}\sqrt{-Q_{n,p}^{(1)}(T)}.
\label{kappa rational in T}%
\end{equation}

\end{theorem}

\begin{proof}
The functions
\[
\wp(\sum_{i=1}^{n+1}a_{i}),\quad\wp^{\prime}(\sum_{i=1}^{n+1}a_{i})
\]
are rational on $\Gamma_{n,p}^{(1)}(\tau)$. Since
\begin{equation}
\sigma_{n,p}^{(1)}(P^{*}) = -\sigma_{n,p}^{(1)}(P),\label{sigma P* and P}%
\end{equation}
the first is invariant under $P\mapsto P^{*}$, whereas the second is
anti-invariant. Hence
\[
\wp\bigl(\sigma_{n,p}^{(1)}(P)\bigr),\wp^{\prime}(\sigma_{n,p}^{(1)}%
(P))/C\in\mathbb{C}(T),
\]
giving \eqref{P(sigma) in T}--\eqref{P'(sigma) in T}.

Likewise, $\kappa(P^{*})=-\kappa(P)$, so ${\kappa}/{C}\in\mathbb{C}(T),$
giving \eqref{kappa rational in T}.
\end{proof}

\begin{remark}
\label{remark, P2,4,6's common zeros} Fix $\tau\in\mathbb{H}$ and $p\in
E_{\tau}\setminus E_{\tau}[2]$. The rational functions $\wp(\sigma_{n,p}%
^{(1)}(P))$, $\wp^{\prime}(\sigma_{n,p}^{(1)}(P))$ and $\kappa(P)$ have poles
only at those points $P$ for which $\sigma_{n,p}^{(1)}(P)=0$. Consequently,
the associated polynomials $P_{j}(T,p;\tau)$, $j=2,4,6,$ have the same set of
zeros in the variable $T$ with possibly different multiplicities.
\end{remark}

\begin{lemma}
\label{lem, Pj is even in T} The polynomials $P_{j}(T,p;\tau),$ $j=1,2,\cdots
,6$, introduced in Theorem~\ref{thm, P and P' rational in T}, are even
functions in the variable $T$. Moreover, their coefficients are independent of
$\wp^{\prime}(p)$. Equivalently, there exist polynomials
\begin{equation}
\hat{P}_{j}(X,p;\tau)\in\mathbb{C}[\wp(p),e_{k}(\tau)][X]
\label{hat Pj in ring of poly}%
\end{equation}
such that
\begin{equation}
P_{j}(T,p;\tau)=\hat{P}_{j}(T^{2},p;\tau),\ j=1,\cdots,6. \label{Pj is even}%
\end{equation}

\end{lemma}

The key symmetry is the invariance of the monodromy data under $P\to-P^{*}$.

\begin{proposition}
\label{prop, data P and -P} Let $P\in\Gamma_{n,p}^{(1)}(\tau)$. Then
\begin{equation}
(r(-P^{*}),s(-P^{*}))=(r(P),s(P)). \label{data -P and data P}%
\end{equation}

\end{proposition}

\begin{proof}
From \eqref{phi's,def} and \eqref{BA function,def}, we have
\begin{align*}
\psi(P;z)=  &  \exp\left(  \int_{z_{0}}^{z}\phi(P;\xi)d\xi\right) \\
=  &  \exp\left(  \int_{z_{0}}^{z}\frac{iC(P)+\frac{1}{2}\Phi_{e,p}%
^{(1)\prime}(\xi;T)}{\Phi_{e,p}^{(1)}(\xi;T)}d\xi\right)  .
\end{align*}
Then
\begin{align}
\psi(P;-z)=  &  \exp\left(  \int_{z_{0}}^{-z}\frac{iC(P)+\frac{1}{2}\Phi
_{e,p}^{(1)\prime}(\xi;T)}{\Phi_{e,p}^{(1)}(\xi;T)}d\xi\right) \nonumber\\
=  &  \exp\left(  -\int_{-z_{0}}^{z}\frac{iC(P)+\frac{1}{2}\Phi_{e,p}%
^{(1)\prime}(-\xi;T)}{\Phi_{e,p}^{(1)}(-\xi;T)}d\xi\right)  .
\label{psi(P;-z)}%
\end{align}

By Proposition~\ref{prop, sym}-(ii),
\[
\Phi_{e,p}(z;-T)=\Phi_{e,p}(-z;T).
\]
Differentiating with respect to $z$ yields
\begin{equation}
\Phi_{e,p}^{\prime}(z;-T)=-\Phi_{e,p}^{\prime}(-z;T).\label{Phi is symmetric}%
\end{equation}
Combining \eqref{psi(P;-z)} and \eqref{Phi is symmetric} gives
\begin{equation}
\psi(P;-z)=\exp\left(  \int_{-z_{0}}^{z}\frac{-iC(P)+\frac{1}{2}\Phi
_{e,p}^{\prime}(\xi;-T)}{\Phi_{e,p}(\xi;-T)}\right)  =\psi(-P;z)
\label{psi(P;-z) equals psi(-P;z)}%
\end{equation}
up to a nonzero multiple.

Replacing $P$ by $P^{*}$, we obtain from \eqref{psi(P;-z) equals psi(-P;z)}
that
\[
\begin{aligned} \psi(-P^*;z+\omega_j)&=\psi(P^*;-z-\omega_j)=\lambda^{-1}_j(P^*)\psi(P^*;-z)\\ &=\lambda_j(P)\psi(-P^*;z), \end{aligned}
\]
which implies
\[
\lambda_{j}(-P^{*})=\lambda_{j}(P),\ j=1,2.
\]
Then \eqref{data -P and data P} follows from \eqref{def,r(P),s(P)}.
\end{proof}

\begin{proof}
[Proof of Lemma~\ref{lem, Pj is even in T}]By
Lemma~\ref{lem, relation bet r(P),s(P) and aj} and
Proposition~\ref{prop, data P and -P},
\begin{equation}
\sigma_{n,p}^{(1)}(-P)=r(-P)+s(-P)\tau=-(r(P)+s(P)\tau)=-\sigma_{n,p}%
^{(1)}(P).
\end{equation}
Together with \eqref{sigma P* and P}, this yields
\begin{equation}
\sigma_{n,p}^{(1)}(-P^{\ast})=\sigma_{n,p}^{(1)}(P), \label{sigma P=sigma -P*}%
\end{equation}
and hence
\begin{equation}
\wp(\sigma_{n,p}^{(1)}(P))=\wp(\sigma_{n,p}^{(1)}(-P^{\ast})).
\end{equation}
Since $P\mapsto-P^{\ast}$ corresponds to $T\mapsto-T$, it follows that
$\wp(\sigma_{n,p}(P))$ is invariant under $T\mapsto-T$. Consequently, there
exist coprime polynomials
\[
\hat{P}_{j}(X,p;\tau)\in\mathbb{C}[\wp(p),\wp^{\prime}(p),e_{k}(\tau
)][X],\ j=1,2,
\]
such that
\[
P_{j}(T,p;\tau)=\hat{P}_{j}(T^{2},p;\tau),\ j=1,2.
\]

The same argument applies to $\wp^{\prime}(\sigma_{n,p}^{(1)}(P))$ and
$\kappa(P)$, yielding \eqref{Pj is even} for $j=3,4,5,6$.

It remains to show that the coefficients of $\hat{P}_{j}(X,p;\tau
),j=1,\cdots,6$ are independent of $\wp^{\prime}(p)$. By
Proposition~\ref{prop, sym}, the spectral polynomial satisfies
\[
Q_{n,p}^{(1)}(T)\in\mathbb{C}(\wp(p),e_{k}(\tau))[T].
\]
Hence the spectral curve $\Gamma_{n,p}^{(1)}(\tau)$ is invariant under
$p\mapsto-p$, that is,
\begin{equation}
\Gamma_{n,p}^{(1)}(\tau)=\Gamma_{n,-p}^{(1)}(\tau),
\label{spectral curve is inde of wp'}%
\end{equation}
and therefore does not depend on $\wp^{\prime}(p)$.

Let%
\[
P=(T(P),C(P))\in\Gamma_{n,p}^{(1)}(\tau).
\]
Write
\[
\psi(P;z)=\psi(P;p,z),\text{ }\mathbf{a}(P)=\mathbf{a}(P;p)=\{a_{1}%
(P;p),\cdots,a_{n+1}(P;p)\}
\]
to emphasize the dependence on $p$. By
Theorem~\ref{thm, ya solves noneven ode}, the zero divisor $\mathbf{a}(P;p)$
is characterized by the conditions \eqref{ya is sol. condition1} and
\eqref{ya is sol. condition2}, both of which depend only on $\wp(p)$ and are
hence invariant under $p\rightarrow-p$. Thus
\[
\mathbf{a}(P;-p)\in Y_{n,-p}^{(1)}(\tau)=Y_{n,p}^{(1)}(\tau).
\]

By \eqref{Phi(z,-p)=Phi(z,p)} and \eqref{BA function,def}, we have
\begin{align}
\psi(P;p,z)  &  =\exp\left(  \int_{z_{0}}^{z}\frac{iC(P)+\frac{1}{2}\Phi
_{e,p}^{\prime}(z;T(P))}{\Phi_{e,p}(z;T(P))}d\xi\right) \nonumber\\
&  =\exp\left(  \int_{z_{0}}^{z}\frac{iC(P)+\frac{1}{2}\Phi_{e,-p}^{\prime
}(z;T(P))}{\Phi_{e,-p}(z;T(P))}d\xi\right)  =\psi(P;-p,z).\nonumber
\end{align}
Thus
\[
\mathbf{a}(P;p)=\mathbf{a}(P;-p)
\]
and hence
\begin{equation}
\sigma_{n,p}^{(1)}(P)=\sum_{i=1}^{n+1}a_{i}(P;p)=\sum_{i=1}^{n+1}%
a_{i}(P;-p)=\sigma_{n,-p}^{(1)}(P). \label{add map is inde of wp'}%
\end{equation}
Therefore, $\wp(\sigma_{n,p}^{(1)}(P))$, $\wp^{\prime}(\sigma_{n,p}^{(1)}(P))$
and $\kappa(P)$ are all invariant under $p\to-p$, so their coefficients are
independent of $\wp^{\prime}(p)$.
\end{proof}

Let $\hat{P}=(X,C)=\pi(P)\in\hat{\Gamma}_{n,p}^{(1)}(\tau)$. By
Lemma~\ref{lem, Pj is even in T}, $\wp(\hat{\sigma}_{n,p}^{(1)}(\hat P))$
depends only on $X=T^{2}$. Hence it can be written in the form
\begin{equation}
\wp(\hat{\sigma}_{n,p}^{(1)}(\hat{P}))=\frac{\hat{P}_{1}(X,p;\tau)}{\hat
{P}_{2}(X,p;\tau)}.
\end{equation}

Moreover, Lemma~\ref{lem, closure of Yn,p} implies that
\begin{equation}
\hat{\sigma}_{n,p}^{(1)}(\hat{P})\to0\ \text{as }X\to\infty.
\label{hat sigma at infty}%
\end{equation}
Hence,
\begin{equation}
\deg_{X} \hat{P}_{1}(X,p;\tau)>\deg_{X} \hat{P}_{2}(X,p;\tau).
\label{deg P1>P2}%
\end{equation}
Applying the argument \cite[Proposition~3.2]{Takemura4}, we conclude that
\begin{equation}
\wp\bigl(\hat{\sigma}_{n,p}^{(1)}(\hat{P})\bigr)\sim\frac{4}{n^{2}(n+1)^{2}%
}X\quad\text{as }X\to\infty. \label{P(sigma hat)'s asym}%
\end{equation}

Consequently,
\begin{equation}
\deg_{X} \hat{P}_{1}(X,p;\tau)=\deg_{X} \hat{P}_{2}(X,p;\tau)+1.
\label{deg hat P1 and P2}%
\end{equation}

\medskip

The next proposition characterizes several basic properties of $\hat{\sigma
}_{n,p}^{(1)}$. Follow the idea established in \cite[Lemma~6.2]%
{Chen-Kuo-Lin-Lame I}, the degree $\deg\hat{\sigma}_{n,p}^{(1)}$ is
characterized by
\begin{equation}
\deg\hat\sigma_{n,p}^{(1)} = \deg_{X}\hat P_{1}(X,p;\tau
).\label{deg sigma is deg P1}%
\end{equation}

\begin{proposition}
\label{prop, degree of hat sigma} Let $\tau\in\mathbb{H}$ and $p\in E_{\tau
}\setminus E_{\tau}[2]$. Then $\deg\hat{\sigma}_{n,p}^{(1)}$ is independent of
$p$ and $\tau$.
\end{proposition}

\begin{proof}
By \eqref{deg sigma is deg P1}, it suffices to prove that $\deg_{X}\hat
P_{1}(X,p;\tau)$ is independent of $(p,\tau)$. Write
\[
\hat P_{1}(X,p;\tau)=\sum_{j=0}^{m}a_{j}(p;\tau)X^{j}, \qquad\hat
P_{2}(X,p;\tau)=\sum_{j=0}^{m-1}b_{j}(p;\tau)X^{j},
\]
where $m$ is the generic degree of $\hat P_{1}$.

We claim that whenever the coefficients of $\hat P_{1}$ and $\hat P_{2}$ do
not all vanish at $(p_{0},\tau_{0})$, one has $a_{m}(p_{0};\tau_{0})\neq0$.
Otherwise, choose generic $\sigma_{0}\in E_{\tau_{0}}\setminus\{0\}$ such
that
\[
P_{0}(X):= \hat P_{1}(X,p_{0};\tau_{0}) -\wp(\sigma_{0};\tau_{0})\hat
P_{2}(X,p_{0};\tau_{0}) \not \equiv 0.
\]
Then $\deg_{X}P_{0}\leqslant m-1$. For $(p_{\ell},\tau_{\ell})\to(p_{0},\tau_{0})$
with $a_{m}(p_{\ell};\tau_{\ell})\neq0$, set
\[
P_{\ell}(X):= \hat P_{1}(X,p_{\ell};\tau_{\ell}) -\wp(\sigma_{0};\tau_{\ell
})\hat P_{2}(X,p_{\ell};\tau_{\ell}).
\]
Since $\deg_{X}P_{\ell}=m$ and $P_{\ell}\to P_{0}$, some zero $X_{\ell}$ of
$P_{\ell}$ tends to $\infty$. If $\hat P_{\ell}$ lies above $X_{\ell}$, then
\[
\hat\sigma_{n,p_{\ell}}^{(1)}(\hat P_{\ell})=\sigma_{0},
\]
whereas \eqref{hat sigma at infty} yields $\hat\sigma_{n,p_{\ell}}^{(1)}(\hat
P_{\ell})\to0$, a contradiction. Thus the claim holds.

Now fix $(p_{0},\tau_{0})$. After dividing $\hat P_{1}$ and $\hat P_{2}$ by
their maximal common vanishing factor in $p-p_{0}$, and, if necessary, in
$\tau-\tau_{0}$, their coefficients are not all zero at $(p_{0},\tau_{0})$.
The claim therefore gives
\[
\deg_{X}\hat P_{1}(X,p_{0};\tau_{0})=m.
\]
Hence
\[
\deg\hat\sigma_{n,p}^{(1)} = \deg_{X}\hat P_{1}(X,p;\tau) = m
\]
for all $\tau\in\mathbb{H}$ and $p\in E_{\tau}\setminus E_{\tau}[2]$.
\end{proof}

For $p\in E_{\tau}\setminus E_{\tau}[2]$, let
\[
P(p)=(T(p),C(p))\in\Gamma_{n,p}^{(1)}(\tau),\ \hat{P}(p)=(X(p),C(p))\in
\hat{\Gamma}_{n,p}^{(1)}(\tau),
\]
and denote
\begin{equation}
i_{n,p}(P(p))=\mathbf{a}(p):=\{a_{1}(p),\cdots,a_{n+1}(p)\}\in Y_{n,p}%
^{(1)}(\tau), \label{in,p=a(p)}%
\end{equation}
which is the zero divisor of Baker--Akhiezer function $\psi(P(p);z)$
associated with $\mathrm{GLE}_{n}^{(1)}(p,T(p),\tau)$.

Let $m=\deg\hat{\sigma}_{n,p}^{(1)}$. Due to
Proposition~\ref{prop, degree of hat sigma}, we consider a sequence $p\to0$,
then $\wp\bigl(\hat{\sigma}_{n,p}^{(1)}(\hat{P}(p))\bigr)$ admits the
factorization
\begin{equation}
\wp\bigl(\hat{\sigma}_{n,p}^{(1)}(\hat{P}(p))\bigr)=\dfrac{4}{n^{2}(n+1)^{2}%
}\dfrac{\prod_{i=1}^{m}(X(p)-X_{i}^{0}(p))}{\prod_{j=1}^{m-1}(X(p)-X_{j}%
^{\infty}(p))}, \label{P(hat sigma(P(p)))}%
\end{equation}
where $X_{i}^{0}(p)$'s and $X_{j}^{\infty}(p)$'s are the zeros of $\hat{P}%
_{1}(X(p),p;\tau)$ and $\hat{P}_{2}(X(p),p;\tau)$ respectively.

We analyze \eqref{P(hat sigma(P(p)))} as $p\to0$, beginning with the
convergence of GLE$_{n}^{(1)}(p,$ $T(p),\tau)$. Since $\mathbf{a}(p)\in
Y_{n,p}^{(1)}(\tau)$, we may assume that
\begin{equation}
\mathbf{a}(p)\to\mathbf{a}^{0}=\{a_{1}^{0},\cdots,a_{n+1}^{0}\}\in
\operatorname{Sym}^{n+1} E_{\tau}\text{ as }p\to0. \label{a(p) as p to 0}%
\end{equation}

The argument of \cite[Proposition~4.2]{Chen-Kuo-Lin-Lame I} gives:

\begin{lemma}
\cite{Chen-Kuo-Lin-Lame I} \label{lem, q converges iff c converges} Let $p\in
E_{\tau}\setminus E_{\tau}[2]$ and assume that $p\to0\in E_{\tau}$. Suppose
that $(T,B)=(T(p),B(p))$ is chosen so that $\mathrm{GLE}_{n}^{(1)}%
(p,T(p),\tau)$ is log-free. Then the family $\mathrm{GLE}_{n}^{(1)}%
(p,T(p),\tau)$ converges if and only if the corresponding parameter $c(p)$ in
the Hermite--Halphen solution $y_{\mathbf{a}(p),c(p)}(z)$ converges, if and
only if $B(p)$ converges.
\end{lemma}

\begin{lemma}
\label{lem, gle converges to lame} As $p\to0$, $B(p)$ converges to $\tilde
{B}\in\mathbb{C}$ if and only if $X(p)=T(p)^{2}$ satisfies
\begin{equation}
X(p)=\frac{n(n+1)}{p^{2}}+\widetilde{B}+O(p). \label{X(p)'s behavior at 0}%
\end{equation}
In such case, $\mathrm{GLE}_{n}^{(1)}(p,T(p),\tau)$ converges to
$\mathrm{L}_{n}(\widetilde{B},\tau)$ uniformly on any compact subsets of
$E_{\tau}\setminus\{0\}$.
\end{lemma}

\begin{proof}
From \eqref{B(T),noneven},
\begin{equation}
\begin{aligned} B(p)&=T^2(p)+\frac{\wp^{\prime\prime}(p)}{2\wp^{\prime}(p)}\zeta(p)-\frac{1}{2}(2n^2+2n-3)\wp(p)\\ &= T^2(p)-\frac{n(n+1)}{p^2}+O(p^2). \end{aligned} \label{0415equ1}%
\end{equation}
Hence $B(p)$ converges to $\widetilde{B}\in\mathbb{C}$ if and only if $T(p)$
satisfies \eqref{X(p)'s behavior at 0}.

As $p\to0$, we have the following expansions
\begin{equation}
\wp(z+p)+\wp(z-p)=2\wp(z)+O(p^{2}),
\end{equation}
\begin{equation}
\zeta(z+p)+\zeta(z-p)-2\zeta(z)=O(p^{2}),
\end{equation}
\begin{equation}
\zeta(z+p)-\zeta(z-p)=-2\wp(z)p+O(p^{3}). \label{0415equ2}%
\end{equation}

Using \eqref{0415equ1}--\eqref{0415equ2} yields
\begin{equation}
\begin{aligned} q_{n}^{(1)}(z;p,T(p))&=\left[ \begin{array} [c]{l}n(n+1)\wp(z)+\frac{3}{4}(\wp(z+p)+\wp(z-p))\\ +T(p)(\zeta(z+p)+\zeta(z-p)-2\zeta(z))\\ -\frac{\wp^{\prime\prime}(p)}{4\wp^{\prime}(p)}(\zeta(z+p)-\zeta(z-p))+B(p) \end{array} \right]\\ &=n(n+1)\wp(z)+T^2(p)-\frac{n(n+1)}{p^2}+T(p)\,O(p^2)+O(p). \end{aligned}
\end{equation}

Therefore, $q_{n}^{(1)}(z;p,T(p))$ converges as $p\to0$ if and only if
\eqref{X(p)'s behavior at 0} holds. In such case,
\[
q_{n}^{(1)}(z;p,T(p))\to n(n+1)\wp(z)+\widetilde{B},
\]
which implies that $\mathrm{GLE}_{n}^{(1)}(p,T(p),\tau)$ converges uniformly
on every compact subset of $E_{\tau}\setminus\{0\}$ to $\mathrm{L}%
_{n}(\widetilde{B},\tau)$.
\end{proof}

By Lemma~\ref{lem, gle converges to lame}, as $p\to0$, under the constraint
\eqref{X(p)'s behavior at 0}, $\mathrm{GLE}_{n}^{(1)}(p,T(p),\tau)$ uniformly
converges to the classical Lam\'{e} equation $\mathrm{L}_{n}(\widetilde{B}%
,\tau)$. We recall the classical Lam\'e facts needed below
\cite{CLW,LW-AnnMath,LW}.

\medskip

\noindent(1) The classical Lam\'{e} potential $n(n+1)\wp(z)$ is a stationary
elliptic KdV potential. Let $\widetilde{Q}_{n}(\widetilde{B};\tau)$ denote the
associated spectral polynomial, which has degree $2n+1$ in $\widetilde{B}$.
The corresponding spectral curve is defined by
\begin{equation}
\widetilde{\Gamma}_{n}(\tau):=\{(\widetilde{B},\widetilde{C})\in\mathbb{C}%
^{2}\mid\widetilde{C}^{2}=\widetilde{Q}_{n}(\widetilde{B};\tau)\}.
\label{lame's spectral curve}%
\end{equation}
Its natural compactification is given by
\[
\overline{ \widetilde{\Gamma}_{n}(\tau)}= \widetilde{\Gamma}_{n}(\tau
)\cup\widetilde{\infty},
\]
where $\widetilde{\infty}$ denotes the unique point at infinity of
$\widetilde{\Gamma}_{n}(\tau)$.

\medskip

\noindent(2) For each $\widetilde{P}=(\widetilde{B},\widetilde{C}%
)\in\widetilde{\Gamma}_{n}(\tau)$, there exists a Baker--Akhiezer function
$\psi(\widetilde{P};z)$ associated with $\mathrm{L}_{n}(\widetilde{B},\tau)$,
which admits a Hermite-Halphen representation. Denote its zero divisor by
\begin{equation}
\tilde{\mathbf{a}}(\widetilde{P})=\{\tilde{a}_{1}(\widetilde{P}),\cdots
,\tilde{a}_{n}(\widetilde{P})\}\in\operatorname{Sym}^{n}E_{\tau}^{*}.
\label{lame's zero set}%
\end{equation}

\medskip

\noindent(3) Define the addition map
\begin{equation}
\tilde\sigma_{n}: \overline{\widetilde{\Gamma}_{n}(\tau)} \to E_{\tau},
\label{add map of lame}%
\end{equation}
by
\[
\tilde\sigma_{n}(\widetilde{P}) = \sum_{j=1}^{n} \tilde a_{j}(\widetilde{P}%
),\quad\tilde{\sigma}_{n}(\widetilde{\infty})=0.
\]
Then $\tilde\sigma_{n}$ is a finite morphism and
\[
\deg\tilde{\sigma}_{n}=\frac{n(n+1)}{2}.
\]

Furthermore, $\wp(\tilde{\sigma}_{n}(\widetilde{P}))$ is a rational function
on $\widetilde{\Gamma}_{n}(\tau)$ and admits the expression
\begin{equation}
\wp(\tilde\sigma_{n}(\widetilde{B},\widetilde{C}))=\dfrac{4}{n^{2}(n+1)^{2}%
}\dfrac{\prod_{i=1}^{\frac{n(n+1)}{2}}(\widetilde{B}-B_{i}^{0})}{\prod
_{j=1}^{\frac{n(n+1)}{2}-1}(\widetilde{B}-B_{j}^{\infty})},
\label{lame's P(sigma n)}%
\end{equation}
where $\{B_{i}^{0}\}_{i=1}^{\frac{n(n+1)}{2}}$ and $\{B_{j}^{\infty}%
\}_{j=1}^{\frac{n(n+1)}{2}-1}$ denote the $\widetilde{B}$--coordinates of the
zeros and poles of $\wp(\tilde\sigma_{n}(\widetilde{B},\widetilde{C}))$, respectively.

\medskip

Comparing \eqref{P(hat sigma(P(p)))} and \eqref{lame's P(sigma n)}, it remains
to prove $m=\frac{n(n+1)}{2}.$ The following proposition is a consequence of
\cite[Propositions~5.3 and~5.5]{Chen-Kuo-Lin-Lame I}; its proof is therefore omitted.

\begin{proposition}
\label{prop, B and a} \noindent\textup{(i)} Suppose that $B(p)\to
\widetilde{B}_{0}\in\mathbb{C}$ as $p\to0$. Then
\begin{equation}
\mathbf{a}^{0}=\{a_{1}^{0},\ldots,a_{n}^{0},0\},
\label{a0 lies in lame's zero}%
\end{equation}
where $\{a_{1}^{0},\ldots,a_{n}^{0}\}\in\operatorname{Sym}^{n}E_{\tau}^{*}$ is
the zero divisor of the Baker--Akhiezer function $\psi(\widetilde{P}_{0};z)$
associated with $\mathrm{L}_{n}(\widetilde{B}_{0},\tau)$, and $\widetilde{P}%
_{0}=(\widetilde{B}_{0},\widetilde{C}_{0}) \in\widetilde{\Gamma}_{n}(\tau)$.

\noindent\textup{(ii)} Suppose that $B(p)\to\infty$ as $p\to0$. Then
\begin{equation}
\mathbf{a}^{0}\in\{\infty_{+}(p),\infty_{-}(p)\}, \label{mathbf a=0}%
\end{equation}
where $\infty_{+}(p)$ and $\infty_{-}(p)$ are given by \eqref{infty point 1}
and \eqref{infty point 2}, respectively.
\end{proposition}

The next theorem shows that any nonzero limit of $\hat{\sigma}_{n,p}%
^{(1)}(\hat P(p))$ induces the convergence of $\mathrm{GLE}_{n}^{(1)}(p,$
$T(p),\tau)$ to the Lam\'e equation $\mathrm{L}_{n}(\widetilde{B},\tau)$.

\begin{theorem}
\label{thm, sigma neq 0 implies convergence} Let $\hat{P}(p)=(X(p),C(p))\in
\hat{\Gamma}_{n,p}^{(1)}(\tau)$ with $X(p)=T(p)^{2}$. Assume that $\hat
{\sigma}_{n,p}^{(1)}(\hat{P}(p))\to\sigma_{0}\neq0$ as $p\to0$. Then,
$\mathrm{GLE}_{n}^{(1)}(p,T(p),\tau)$ converges to $\mathrm{L}_{n}(\tilde
{B},\tau)$, where $\widetilde{B}$ satisfies
\begin{equation}
\tilde{\sigma}_{n}(\widetilde{B},\widetilde{C})=\sigma_{0}.
\label{lame's sigma is sigma0}%
\end{equation}
with $(\widetilde{B},\widetilde{C})\in\widetilde{\Gamma}_{n}(\tau)$.
\end{theorem}

\begin{proof}
We claim that $B(p)$ is uniformly bounded as $p\to0$. If not, we may assume
$B(p)\to\infty$. Then Proposition~\ref{prop, B and a}~(ii) implies that
$\mathbf{a}^{0}$ satisfies \eqref{mathbf a=0}. Consequently,
\[
\sigma_{0} = \lim_{p\to0} \hat\sigma_{n,p}^{(1)}(\hat P(p)) = 0,
\]
which contradicts the assumption $\sigma_{0}\neq0$. Therefore, $B(p)$ is
bounded near $p=0$. Passing to a subsequence if necessary, we may assume that
$B(p)\to\widetilde{B}\in\mathbb{C}$.

By Proposition~\ref{prop, B and a}~(i), it follows that
\eqref{a0 lies in lame's zero} holds and hence
\[
\sigma_{0} = \lim_{p\to0} \hat\sigma_{n,p}^{(1)}(\hat P(p)) = \sum_{j=1}^{n}
a_{j}^{0} = \pm\tilde\sigma_{n}(\widetilde{B}_{0},\widetilde{C}_{0}).
\]
Replacing $\widetilde{C}_{0}$ by $-\widetilde{C}_{0}$ if necessary, we obtain \eqref{lame's sigma is sigma0}.
\end{proof}

We determine the zeros and poles in \eqref{P(hat sigma(P(p)))}. For the zeros,
Theorem~\ref{thm, sigma neq 0 implies convergence} gives:

\begin{lemma}
\label{lem, zero of P(sigma hat)} As $p\to0$, for each $i=1,2,\cdots,m$,
\begin{equation}
X_{i}^{0}(p)=\dfrac{n(n+1)}{p^{2}}+\widetilde{B}_{i}^{0}+o(1),
\label{X_i^0's behavior}%
\end{equation}
with $\widetilde{B}_{i}^{0}\in\{B_{1}^{0},\cdots,B_{\frac{n(n+1)}{2}}^{0}\}$.
\end{lemma}

For the poles, let
\[
\hat{P}_{j}^{\infty}(p):=(X_{j}^{\infty}(p),C(\hat{P}_{j}^{\infty}(p)))\in
\hat{\Gamma}_{n,p}^{(1)}(\tau)
\]
and $T_{j}^{\infty}(p)\in\mathbb{C}$ with $T_{j}^{\infty}(p)^{2}=X_{j}%
^{\infty}(p)$.

Since
\begin{equation}
\wp\bigl(\hat{\sigma}_{n,p}^{(1)}(\hat{P}_{j}^{\infty}(p))\bigr)=\infty,
\end{equation}
we have $\hat{\sigma}_{n,p}^{(1)}(\hat{P}_{j}^{\infty}(p))=0$. By
Theorem~\ref{thm, sigma neq 0 implies convergence} and
Proposition~\ref{prop, B and a}--(ii), the associated equation $\mathrm{GLE}%
_{n}^{(1)}(p,T_{j}^{\infty}(p),\tau)$ may fail to converge as $p\to0$. There
are three possibilities:

\begin{itemize}
\item[(1)] $X_{j}^{\infty}(p)=X_{j}^{\infty,c}(p)$, in which case
$\mathrm{GLE}_{n}^{(1)}(p,T_{j}^{\infty}(p),\tau)$ converges as $p\to0$. By
Lemma~\ref{lem, gle converges to lame},
\begin{equation}
X_{j}^{\infty,c}(p)=\dfrac{n(n+1)}{p^{2}}+\widetilde{B}_{j}^{\infty,c}+o(1)
\label{0421equ2}%
\end{equation}
for some $\widetilde{B}_{j}^{\infty,c}\in\mathbb{C}$.

\item[(2)] $\mathrm{GLE}_{n}^{(1)}(p,T_{j}^{\infty}(p),\tau)$ diverges as
$p\to0$, which can be further divided into two cases:

\item[(D-i)] $X_{j}^{\infty}(p)=X_{j}^{\infty_{1}}(p)$, where
\begin{equation}
X_{j}^{\infty_{1}}(p)=\dfrac{n(n+1)}{p^{2}}+\dfrac{\alpha_{j}^{\infty_{1}}}%
{p}+O(1), \label{X infty1's asy}%
\end{equation}
with $\alpha_{j}^{\infty_{1}}\neq0$.

\item[(D-ii)] $X_{j}^{\infty}(p)=X_{j}^{\infty_{2}}(p)$, where
\begin{equation}
p^{2}X_{j}^{\infty_{2}}(p)\to\alpha_{j}^{\infty_{2}}\in\mathbb{C}\cup
\{\infty\}\setminus\{n(n+1)\}. \label{X infty2's asy}%
\end{equation}

\end{itemize}

Then for $\hat{P}(p)=(X(p),C(p))\in\hat{\Gamma}_{n,p}^{(1)}(\tau)$, we have
\begin{equation}
\begin{aligned} &\wp\bigl(\hat{\sigma}_{n,p}^{(1)}(\hat{P}(p))\bigr)\\ =&\dfrac{4}{n^2(n+1)^2}\dfrac{\prod_{i=1}^{m}(X(p)-X_{i}^{0}(p))}{\prod_{j=1}^{\hat{m}}(X(p)-X_{j}^{\infty,c}(p))\prod_{l=1}^{2}\prod_{j=1}^{k_{l}}(X(p)-X_{j}^{\infty_l}(p))}. \end{aligned} \label{P(sigma hat(p)) 1}%
\end{equation}

\begin{lemma}
\label{lem, Xi infty's asy} As $p\to0$, for each $j=1,2,\cdots,m-1$,
\begin{equation}
X_{j}^{\infty}(p)=\dfrac{n(n+1)}{p^{2}}+\widetilde{B}_{j}^{\infty}+o(1),
\end{equation}
with $\widetilde{B}_{j}^{\infty}\in\{B_{1}^{\infty},\cdots,B_{\frac{n(n+1)}%
{2}-1}^{\infty}\}$.
\end{lemma}

\begin{proof}
Assume that among $\{X_{j}^{\infty}(p)\}$, there are $k_{1}$ terms satisfying
\textrm{(D-i)} and $k_{2}$ terms satisfying \textrm{(D-ii)}. It suffices to
show that $k_{1}=k_{2}=0$.

Fix $\sigma_{0}\in E_{\tau}\setminus E_{\tau}[2]$ with $\wp(\sigma_{0}%
)\notin\{0,\infty\}$. Let $\sigma_{p}\in E_{\tau}\setminus E_{\tau}[2]$ such
that $\lim_{p\to0}\sigma_{p}=\sigma_{0}$. By
Theorem~\ref{thm, sigma neq 0 implies convergence},
\begin{equation}
\mathrm{GLE}_{n}^{(1)}(p,T(p),\tau)\ \text{converges to }\mathrm{L}%
_{n}(\widetilde{B},\tau)\ \text{as }p\to0, \label{GLE converges to Lame}%
\end{equation}
where $T(p)^{2}=X(p)$, $(X(p),C(p))\in\hat{\sigma}_{n,p}^{(1)-1}(\sigma_{p})$.

Consequently,
\begin{equation}
\wp\bigl(\sigma_{p}\bigr)=\dfrac{4}{n^{2}(n+1)^{2}}\dfrac{\prod_{i=1}%
^{m}(X(p)-X_{i}^{0}(p))}{\prod_{j=1}^{\hat{m}}(X(p)-X_{j}^{\infty,c}%
(p))\prod_{l=1}^{2}\prod_{j=1}^{k_{l}}(X(p)-X_{j}^{\infty_{l}}(p))},
\label{P(sigma p) 0}%
\end{equation}
where $\hat{m}=m-1-k_{1}-k_{2}$.

From Lemma~\ref{lem, gle converges to lame}, \eqref{GLE converges to Lame}
implies that
\[
X(p)=\frac{n(n+1)}{p^{2}}+\widetilde{B}+o(1),
\]
where
\begin{equation}
\widetilde{B}\notin\{\widetilde{B}_{i}^{0},\widetilde{B}_{j}^{\infty,c}\}
\label{range of tilde B}%
\end{equation}
due to $\wp(\sigma_{0})\notin\{0,\infty\}$.

Observe that
\begin{gather}
X(p)-X_{i}^{0}(p)=\widetilde{B}-\widetilde{B}_{i}^{0}+o(1),\label{0421equ3}\\
X(p)-X_{j}^{\infty,c}(p)=\widetilde{B}-\widetilde{B}_{j}^{\infty
,c}+o(1),\label{0422equ1}\\
X(p)-X_{j}^{\infty_{1}}(p)=-\frac{\alpha_{j}^{\infty_{1}}}{p}+o(1),\\
p^{2}(X(p)-X_{j}^{\infty_{2}}(p))=n(n+1)-\alpha_{j}^{\infty_{2}}+o(1).
\label{0421equ4}%
\end{gather}
Substituting \eqref{0421equ3}--\eqref{0421equ4} into \eqref{P(sigma p) 0}
gives
\begin{equation}
\wp(\sigma_{p})=\dfrac{(-1)^{k_{1}}}{n^{2}(n+1)^{2}}\dfrac{4p^{k_{1}+2k_{2}%
}\prod_{i=1}^{m}(\widetilde{B}-\widetilde{B}_{i}^{0})}{\prod_{j=1}^{\hat{m}%
}(\tilde{B}-\widetilde{B}_{j}^{\infty,c})\prod_{j=1}^{k_{1}}\alpha_{j}%
^{\infty_{1}}\prod_{j=1}^{k_{2}}(n(n+1)-\alpha_{j}^{\infty_{2}})}+o(1),
\label{P(sigma p)}%
\end{equation}
where \eqref{X infty1's asy}, \eqref{X infty2's asy} and
\eqref{range of tilde B} imply that
\[
\dfrac{\prod_{i=1}^{m}(\widetilde{B}-\widetilde{B}_{i}^{0})}{\prod_{j=1}%
^{\hat{m}}(\widetilde{B}-\widetilde{B}_{j}^{\infty,c})\prod_{j=1}^{k_{1}%
}\alpha_{j}^{\infty_{1}}\prod_{j=1}^{k_{2}}(n(n+1)-\alpha_{j}^{\infty_{2}}%
)}\neq0.
\]
Since $\wp(\sigma_{p})\to\wp(\sigma_{0})\neq0,\infty$, it follows from
\eqref{P(sigma p)} that
\[
k_{1}+2k_{2}=0,
\]
which forces $k_{1}=k_{2}=0$.
\end{proof}

\begin{proof}
[Proof of Theorem~\ref{thm, degree of hat sigma in appendix}]Fix
$\widetilde{B}\in\mathbb{C}\setminus\{B_{i}^{0},B_{j}^{\infty}\}$ and let
$X(p)=T(p)^{2}$ satisfy $\mathrm{GLE}_{n}^{(1)}$ $(p,T(p),\tau)$ converges to
$\mathrm{L}_{n}(\widetilde{B},\tau)$ as $p\to0$. Then
Lemma~\ref{lem, gle converges to lame} implies that $\widetilde{B}$ and $X(p)$
satisfy
\begin{equation}
X(p)=\frac{n(n+1)}{p^{2}}+\widetilde{B}+o(1)\ \text{near }p=0,
\label{0422equ2}%
\end{equation}
and $B(p)\to\widetilde{B}$ as $p\to0$. Proposition~\ref{prop, B and a}--(i)
then yields
\[
\mathbf{a}(p)\to\mathbf{a}^{0}=\{a_{1}^{0},\cdots,a_{n}^{0},0\}\in
\widetilde{Y}_{n}(\tau).
\]
Hence,
\[
\lim_{p\to0}\hat{\sigma}_{n,p}(\hat{P}(p))=\lim_{p\to0}\sum_{j=1}^{n+1}%
a_{j}(p)=\sum_{j=1}^{n}a_{j}^{0}=\tilde{\sigma}_{n}(\widetilde{B}%
,\widetilde{C}),
\]
where $(\widetilde{B},\widetilde{C})\in\widetilde{\Gamma}_{n}(\tau)$.

Consequently,
\begin{equation}
\lim_{p\to0}\wp\bigl(\hat{\sigma}_{n,p}^{(1)}(X(p),C(p))\bigr)=\wp
(\tilde{\sigma}_{n}(\widetilde{B},\widetilde{C})) . \label{two P as p to 0}%
\end{equation}

Substituting \eqref{0422equ2} into \eqref{P(hat sigma(P(p)))}, and using
Lemmas~\ref{lem, zero of P(sigma hat)}--~\ref{lem, Xi infty's asy}, the limit
\eqref{two P as p to 0} gives
\[
\wp(\sigma_{n}(\widetilde{B},\widetilde{C}))=\dfrac{4}{n^{2}(n+1)^{2}}%
\dfrac{\prod_{i=1}^{m}(\widetilde{B}-\widetilde{B}_{i}^{0})}{\prod_{j=1}%
^{m-1}(\tilde{B}-\widetilde{B}_{j}^{\infty})}.
\]
Comparing with \eqref{lame's P(sigma n)} yields
\[
\deg\hat{\sigma}_{n,p}^{(1)}=m=\deg\tilde{\sigma}_{n}=\frac{n(n+1)}{2}.
\]

\end{proof}

\begin{remark}
The proof of Theorem~\ref{thm, degree of hat sigma in appendix} also shows
that
\[
\{\widetilde{B}_{1}^{0},\cdots,\widetilde{B}_{\frac{n(n+1)}{2}}^{0}%
\}=\{B_{1}^{0},\cdots,B_{\frac{n(n+1)}{2}}^{0}\}
\]
and
\[
\{\widetilde{B}_{1}^{\infty},\cdots,\widetilde{B}_{\frac{n(n+1)}{2}-1}%
^{\infty}\}=\{B_{1}^{\infty},\cdots,B_{\frac{n(n+1)}{2}-1}^{\infty}\}.
\]

\end{remark}

\section{Monodromy Rigidity and $p$-Independence}

\label{Monodromy Theory}

Quotienting ${\Gamma}_{n,p}^{(1)}(\tau)$ by the involution $P\mapsto-P^{\ast}%
$:
\[
\hat{\Gamma}_{n,p}^{(1)}(\tau):=\Gamma_{n,p}^{(1)}(\tau)/\sim,
\]
which, on the $T$-coordinate, is given by $T\sim-T$. Setting
\[
X=T^{2},
\]
we obtain
\begin{equation}
\hat{\Gamma}_{n,p}^{(1)}(\tau)=\{(X,C)\in\mathbb{C}^{2}\mid C^{2}%
=\widehat{Q}_{n,p}^{(1)}(X)\}. \label{quotient curve's express}%
\end{equation}
Since $\deg_{X}\widehat{Q}_{n,p}^{(1)}(X)=2n+1$, $\hat{\Gamma}_{n,p}%
^{(1)}(\tau)$ is a hyperelliptic curve with one point at infinity, denoted by
$\hat\infty^{(1)}$. Therefore, its compactification is given by
\begin{equation}
\overline{\hat\Gamma_{n,p}^{(1)}(\tau)}=\hat\Gamma_{n,p}^{(1)}(\tau
)\cup\{\widehat{\infty}^{(1)}\}.
\end{equation}

Let
\[
\pi:\Gamma_{n,p}^{(1)}(\tau)\to\ \hat{\Gamma}_{n,p}^{(1)}(\tau)
\]
be the projection map given by
\begin{equation}
\pi(P)=\pi(-P^{\ast})=\hat{P}=(X,C).
\end{equation}

Moreover, since the two points at infinity $\infty_{\pm}^{(1)}$ of
$\Gamma_{n,p}^{(1)}(\tau)$ satisfy
\[
-\infty_{+}^{(1)\ast}=\infty_{-}^{(1)},
\]
they are identified under the projection. Hence the map $\pi$ extends to a
holomorphic map of degree two (still denoted by $\pi$)
\[
\pi: \overline{\Gamma_{n,p}^{(1)}(\tau)} \to\overline{\hat{\Gamma}_{n,p}%
^{(1)}(\tau)},
\]
satisfying
\[
\pi(\infty_{+}^{(1)})=\pi(\infty_{-}^{(1)})=\widehat{\infty}^{(1)}.
\]

By \eqref{sigma P=sigma -P*}, the addition map $\sigma_{n,p}^{(1)}%
:\overline{\Gamma_{n,p}^{(1)}(\tau)}\mapsto E_{\tau}$ is invariant under the
involution $P\rightarrow-P^{\ast}$. Hence it descends to a well-defined map
\begin{equation}
\hat{\sigma}_{n,p}^{(1)}:\overline{\hat{\Gamma}_{n,p}^{(1)}(\tau)}\mapsto
E_{\tau}, \label{hat sigma's def}%
\end{equation}
defined by
\begin{equation}
\hat{\sigma}_{n,p}^{(1)}(\hat{P}):=\sigma_{n,p}^{(1)}(P),
\label{add map's def on quotient}%
\end{equation}
where $\hat{P}=\pi(P)$.

By Theorem~\ref{thm, P and P' rational in T} and
Lemma~\ref{lem, Pj is even in T}, the above analysis yields the following theorem.

\begin{theorem}
\label{thm, monodromy sys in X} Let $P=(T,C)\in\Gamma_{n,p}^{(1)}(\tau)$ and
$X=T^{2}$. Then $(r(P),s(P))$ is determined by the system
\begin{equation}
\left\{
\begin{array}
[c]{l}%
\wp(r(P)+s(P)\tau)=\dfrac{\hat{P}_{1}(X,p;\tau)}{\hat{P}_{2}(X,p;\tau)},\\
\wp^{\prime}(r(P)+s(P)\tau)=\dfrac{\hat{P}_{3}(X,p;\tau)}{\hat{P}_{4}%
(X,p;\tau)}\sqrt{-\widehat{Q}_{n,p}^{(1)}(X)},\\
\kappa(P)=Z(r(P),s(P),\tau)=\dfrac{\hat{P}_{5}(X,p;\tau)}{\hat{P}_{6}%
(X,p;\tau)}\sqrt{-\widehat{Q}_{n,p}^{(1)}(X)}.
\end{array}
\right.  \label{mono system for X, r(P) and s(P)}%
\end{equation}

\end{theorem}

Since $(\wp,\wp^{\prime})$ satisfies the classical differential equation
\begin{equation}
\wp^{\prime2}=4\wp^{3}-g_{2}\wp-g_{3}, \label{diff equ bet P and P'}%
\end{equation}
it follows from Theorem~\ref{thm, monodromy sys in X} that
the polynomials $\hat{P}_{j}(X,p;\tau)$, simply denoted by $\hat{P}_{j}(X)$,
$j=1,2,3,4$, satisfy the algebraic identity
\begin{equation}
-\frac{\hat{P}_{3}^{2}(X)}{\hat{P}_{4}^{2}(X)}\,\widehat{Q}_{n,p}%
^{(1)}(X)=4\left(  \frac{\hat{P}_{1}(X)}{\hat{P}_{2}(X)}\right)  ^{3}%
-g_{2}\frac{\hat{P}_{1}(X)}{\hat{P}_{2}(X)}-g_{3}. \label{alge identity on Pj}%
\end{equation}

For $(\sigma,\kappa)\in\mathbb{C}\times\mathbb{CP}^{1}$, consider
\begin{equation}%
\begin{cases}
\wp(\sigma)=\dfrac{\hat{P}_{1}(X,p;\tau)}{\hat{P}_{2}(X,p;\tau)},\\
\wp^{\prime}(\sigma)=\dfrac{\hat{P}_{3}(X,p;\tau)}{\hat{P}_{4}(X,p;\tau)}%
\sqrt{-\widehat{Q}_{n,p}^{(1)}(X)},\\
\kappa=\dfrac{\hat{P}_{5}(X,p;\tau)}{\hat{P}_{6}(X,p;\tau)}\sqrt
{-\widehat{Q}_{n,p}^{(1)}(X)}.
\end{cases}
\label{system for X, r and s}%
\end{equation}

In the degenerate case $(\sigma,\kappa)=(0,\infty)$, the system should be
interpreted as
\begin{equation}
D(X,p;\tau)=0, \label{gcd of P2,4,6}%
\end{equation}
where
\begin{equation}
D(X,p;\tau):=\gcd(\hat{P}_{2}(X,p;\tau),\hat{P}_{4}(X,p;\tau),\hat{P}%
_{6}(X,p;\tau)). \label{gcd's def}%
\end{equation}

\begin{proposition}
Let $P=(T(P),C(P))\in\Gamma_{n,p}^{(1)}(\tau)$ and set $X(P)=T^{2}(P)$,
$\hat{P}=\pi(P)$.

\begin{enumerate}
\item[(i)] $D(X(P),p,\tau)=0$ if and only if $\hat{\sigma}_{n,p}^{(1)}(\hat
{P})=0$.

\item[(ii)] Suppose $D(X(P),p,\tau)=0$. Then $\mathrm{GLE}_{n}^{(1)}%
(p,T(P),\tau)$ is completely reducible if and only if $(r(P),s(P))\neq(0,0)$.
\end{enumerate}
\end{proposition}

\begin{proof}
\noindent(i) For sufficiency, since
\[
\sigma_{n,p}^{(1)}(P)=\hat{\sigma}_{n,p}^{(1)}(\hat{P})=r+s\tau=0,
\]
the functions $\wp(\sigma_{n,p}^{(1)}(P))$, $\wp^{\prime}(\sigma_{n,p}%
^{(1)}(P))$ and $\kappa(P)$ have a common pole at $P$. Equivalently, $X(P)$ is
a common root of $P_{j}(X)$ for $j=2,4,6$. Hence
\[
D(X(P),p;\tau)=0.
\]

For necessity, suppose, for a contradiction, that $\hat\sigma_{n,p}^{(1)}(\hat
P)\neq0$. Then 
$$\wp(\hat\sigma_{n,p}^{(1)}(\hat P))\neq\infty.$$
By
Theorem~\ref{thm, monodromy sys in X}, this implies that $\hat{P}_{2}(X)\neq
0$, contradicting the assumption $D(X,p;\tau)=0$. This proves (i).

\noindent(ii) Suppose $D(X,p;\tau)=0$. By (i), we have
\[
r(P)+s(P)\tau=\hat\sigma_{n,p}^{(1)}(\hat P)=0.
\]
The assertion then follows directly from Theorem~\ref{thm,criterion for}.
\end{proof}

From now on, we focus on the non-degenerate case $(\sigma,\kappa)\in
\mathbb{C}^{*}\times\mathbb{C}$. \medskip

\begin{proposition}
\label{prop, system implies data exists} Given $(\sigma,\kappa)\in
\mathbb{C}^{*}\times\mathbb{C}$. Suppose that $X\in\mathbb{C}$ satisfies the
system \eqref{system for X, r and s}. Then there exists a point $\hat{P}%
\in\hat{\Gamma}_{n,p}^{(1)}(\tau)$ such that
\[
\left(  \sigma_{n,p}^{(1)}(P),\kappa(P)\right)  =(\sigma,\kappa),
\]
for any $P\in\pi^{-1}(\hat{P})$.
\end{proposition}

\begin{proof}
Given $(\sigma,\kappa)\in\mathbb{C}^{*}\times\mathbb{C}$, suppose
$X\in\mathbb{C}$ is a solution of system \eqref{system for X, r and s}. Let
$\hat{P}=(X,C)\in\hat{\Gamma}_{n,p}^{(1)}(\tau)$ and $P\in\pi^{-1}(\hat{P})$.
By Theorem~\ref{thm, monodromy sys in X}, we obtain%
\[
\wp(\sigma_{n,p}^{(1)}(P))=\wp(\sigma),
\]%
\[
\wp^{\prime}(\sigma_{n,p}^{(1)}(P))=\wp^{\prime}(\sigma),
\]
and
\[
\kappa(P)=\kappa.
\]
This completes the proof.
\end{proof}

Moreover, by Lemma~\ref{lem, relation bet r(P),s(P) and aj} and
Proposition~\ref{prop, data P and -P}, we have
\[
\sigma_{n,p}^{(1)}(-P^{*})=\sigma_{n,p}^{(1)}(P), \qquad c(-P^{*})=c(P).
\]
Hence
\[
\kappa(-P^{*}) = \zeta\bigl(\sigma_{n,p}^{(1)}(-P^{*})\bigr)-c(-P^{*}) =
\zeta\bigl(\sigma_{n,p}^{(1)}(P)\bigr)-c(P) = \kappa(P).
\]
Thus $\kappa(P)$ descends to a rational function on the quotient curve
$\hat{\Gamma}_{n,p}^{(1)}(\tau)$, which we denote by $\kappa(\hat{P})$. Let
\begin{equation}
\widehat{W}_{n,p}(\sigma;\kappa)\in\mathbb{C}(\wp(\sigma),\wp^{\prime}%
(\sigma),\wp(p),\wp^{\prime}(p),e_{j}(\tau))[\kappa]
\end{equation}
be the monic minimal polynomial of $\kappa$.

Proposition~\ref{prop, system implies data exists} gives:

\begin{theorem}
\label{thm, equivalence of W and system} Let $(\sigma,\kappa)\in\mathbb{C}%
^{*}\times\mathbb{C}.$ Then the following statements are equivalent:

\begin{enumerate}
\item[(i)] $\widehat{W}_{n,p}(\sigma;\kappa)=0.$

\item[(ii)] There exists $\hat{P}\in\hat{\Gamma}_{n,p}^{(1)}(\tau)$ such that
\[
\bigl(
\hat{\sigma}_{n,p}^{(1)}(\hat{P}), {\kappa}(\hat{P}) \bigr)
= (\sigma,\kappa).
\]

\item[(iii)] The system \eqref{system for X, r and s} with the given pair
$(\sigma,\kappa)$ admits a solution $X$.
\end{enumerate}
\end{theorem}

By Lemma~\ref{lem, Pj is even in T} and \eqref{add map is inde of wp'},
$\widehat{W}_{n,p}(\sigma,\kappa)$ is already independent of $\wp^{\prime}%
(p)$. We prove that it is in fact independent of $p$.

\begin{theorem}
\label{thm, Wn inde of p} The polynomial $\widehat{W}_{n,p}(\sigma;\kappa)$ is
independent of $\wp(p)$.
\end{theorem}

The key is the following rigidity statement.

\begin{theorem}
\label{thm, no isomonodromic deformation} Let $\tau\in\mathbb{H}$ and $p\in
E_{\tau}\setminus E_{\tau}[2]$. Then equation $\mathrm{GLE}_{n}^{(1)}%
(p(\tau),T(\tau),$ $\tau)$ admits no nontrivial monodromy-preserving deformation
as $\tau$ varies.
\end{theorem}

As in the even case \cite{Chen-Kuo-Lin-Hamiltonian}, consider the third-order equation:%

\begin{equation}
\begin{aligned}
    \Omega^{\prime\prime\prime}(z)-&4q_{n}^{(1)}(z;p,T)\Omega^{\prime}%
(z)-2q_{n}^{(1)\prime}(z;p,T)\Omega(z)\\
+&2\frac{\partial}{\partial\tau}%
q_{n}^{(1)}(z;p,T)=0\ \text{in }\mathbb{C}\times\mathbb{H},
\end{aligned}
\label{monodromy 3rd ode}%
\end{equation}
where $^{\prime}=\frac{\partial}{\partial z}$.

\begin{lemma}
\cite{Chen-Kuo-Lin-Hamiltonian} \label{lem, mono pre as tau iff 3rd ode}
Equation $\operatorname{GLE}$$_{n}^{(1)}$($p(\tau),T(\tau),\tau$) is monodromy
preserving as $\tau$ deforms if and only if there exists a single-valued
solution $\Omega(z;\tau)$ to \eqref{monodromy 3rd ode} which satisfies
\begin{align}
\Omega(z+1;\tau)=\Omega(z;\tau),\label{Omega z+1}\\
\Omega(z+\tau;\tau)=\Omega(z;\tau)-1. \label{Omega z+tau}%
\end{align}

\end{lemma}

Recall $\Phi_{e,p}^{(1)}(z;T)$ is the unique elliptic solution of
\eqref{3rd ode, noneven}. If $\Omega(z;\tau)$ is a single-valued solution of
\eqref{monodromy 3rd ode}, then, for any $c\in\mathbb{C}$,
\begin{equation}
\widetilde{\Omega}(z;\tau) := \Omega(z;\tau)+c\,\Phi_{e,p}^{(1)}(z;T)
\end{equation}
is again a solution of \eqref{monodromy 3rd ode}.

Since the local exponents of \eqref{3rd ode, noneven} and
\eqref{monodromy 3rd ode} at $z=0$ coincide, we may choose $c\in\mathbb{C}$ so
that $\widetilde{\Omega}(z;\tau)$ is holomorphic at $z=0$. Replacing
$\Omega(z;\tau)$ by this $\widetilde{\Omega}(z;\tau)$, we may assume that
$\Omega(z;\tau)$ is holomorphic at $z=0$.

Equations \eqref{Omega z+1} and \eqref{Omega z+tau} imply that $\Omega
^{\prime}(z;\tau)$ is elliptic, with possible poles only at $\pm p$, of order
at most two. Hence there exist $A_{1},A_{2},R,C\in\mathbb{C}$ such that
\[
\Omega^{\prime}(z;\tau)=-A_{1}\wp(z+p)-A_{2}\wp(z-p)+R(\zeta(z+p)-\zeta
(z-p))+C.
\]
Then
\[
\Omega(z;\tau)=A_{1}\zeta(z+p)+A_{2}\zeta(z-p)+R(\ln\sigma(z+p)-\ln
\sigma(z-p))+Cz+D
\]
for some $D\in\mathbb{C}$. Since $\Omega(z;\tau)$ is single-valued, it forces
$R=0$. Then
\begin{equation}
\Omega(z;\tau)=A_{1}\zeta(z+p)+A_{2}\zeta(z-p)+Cz+D.
\end{equation}

Equations \eqref{Omega z+1}--\eqref{Omega z+tau} give
\begin{align}
(A_{1}+A_{2})\eta_{1}+C=0,\\
(A_{1}+A_{2})\eta_{2}+C\tau=-1.
\end{align}

By the Legendre relation
\begin{equation}
\tau\eta_{1}-\eta_{2}=2\pi i, \label{legendre relation}%
\end{equation}
we obtain
\begin{equation}
A_{1}+A_{2}=-\frac{i}{2\pi},\quad C=\frac{i}{2\pi}\eta_{1}.
\label{A1+A2 and C}%
\end{equation}
Hence
\begin{equation}
\Omega(z;\tau)=A_{1}\zeta(z+p)+A_{2}\zeta(z-p)+\frac{i\eta_{1}}{2\pi}z+D.
\label{Omega's form}%
\end{equation}

\begin{proof}
[Proof of Theorem~\ref{thm, no isomonodromic deformation}]Suppose, to the
contrary, that $\mathrm{GLE}_{n}^{(1)}(p(\tau),T(\tau),\tau)$ admits a
nontrivial monodromy-preserving deformation in $\tau$. By
Lemma~\ref{lem, mono pre as tau iff 3rd ode}, there exists a single-valued
solution $\Omega(z;\tau)$ of \eqref{monodromy 3rd ode} satisfying
\eqref{Omega z+1} and \eqref{Omega z+tau}. By the preceding discussion, we may
assume that $\Omega$ takes the form \eqref{Omega's form}.

Since
\[
q_{n}^{(1)}(z+1;p,T)=q_{n}^{(1)}(z+\tau;p,T)=q_{n}^{(1)}(z;p,T),
\]
differentiating with respect to $\tau$ gives
\begin{equation}
\frac{\partial}{\partial\tau}q_{n}^{(1)}(z+1;p,T)=\frac{\partial}{\partial
\tau}q_{n}^{(1)}(z;p,T), \label{0601equ1}%
\end{equation}
and
\begin{equation}
q_{n}^{(1)\prime}(z;p,T)+\frac{\partial}{\partial\tau}q_{n}^{(1)}%
(z+\tau;p,T)=\frac{\partial}{\partial\tau}q_{n}^{(1)}(z;p,T). \label{0601equ2}%
\end{equation}

Define
\begin{equation}
\begin{aligned} &U(z;\tau):=\Omega^{\prime\prime\prime}(z;\tau)-4q_{n}^{(1)}(z;p,T)\Omega^{\prime}(z;\tau)\\ &-2q_{n}^{(1)\prime}(z;p,T)\Omega(z;\tau)+2\frac{\partial}{\partial \tau}q_{n}^{(1)}(z;p,T). \end{aligned} \label{U(z)'s def}%
\end{equation}
By \eqref{0601equ1} and \eqref{0601equ2}, $U(z;\tau)$ is elliptic and, since
$\Omega(z;\tau)$ solves \eqref{monodromy 3rd ode}, $U(z;\tau)$ $\equiv0$.  It
therefore suffices to examine the principal parts at $z=\pm p$.

Near $z=p$,
\begin{equation}
q_{n}^{(1)}(z;p,T)=\sum_{l=-2}^{1}q_{l,p}(z-p)^{l}+O((z-p)^{2}),
\label{qn at z=p}%
\end{equation}
where
\begin{equation}
\begin{aligned} q_{-2,p}=\frac{3}{4},\ q_{-1,p}=T+\frac{\wp''(p)}{4\wp'(p)},\ q_{0,p}=\left(T+\frac{\wp''(p)}{4\wp'(p)}\right)^2, \end{aligned}
\end{equation}
and
\begin{equation}
\begin{aligned} q_{1,p} &= n(n+1)\wp'(p) + \frac{3}{4}\wp'(2p) \\ &\quad + T\bigl(2\wp(p)-\wp(2p)\bigr) + \frac{\wp''(p)}{4\wp'(p)}\wp(2p). \end{aligned} \label{q1,p}%
\end{equation}

Differentiating \eqref{qn at z=p} with respect to $\tau$ gives
\begin{equation}
\frac{\partial}{\partial\tau}q_{n}^{(1)}(z;p,T)=\sum_{l=-3}^{-1}\hat{q}%
_{l,p}(z-p)^{l}+O(1),
\end{equation}
where
\begin{gather}
\hat{q}_{-3,p}=\frac{3}{2}\frac{dp(\tau)}{d\tau},\\
\hat{q}_{-2,p}=\left(  T+\frac{\wp^{\prime\prime}(p)}{4\wp^{\prime}%
(p)}\right)  \frac{dp(\tau)}{d\tau},\\
\hat{q}_{-1,p}=\frac{d}{d\tau}\left(  T+\frac{\wp^{\prime\prime}(p)}%
{4\wp^{\prime}(p)}\right)  .
\end{gather}

Similarly, write
\begin{equation}
\Omega(z;\tau)=\sum_{l=-1}^{2}w_{l,p}(z-p)^{l}+O((z-p)^{3}),
\end{equation}
where
\begin{gather*}
w_{-1,p}=A_{2},\\
w_{0,p}=\frac{i\eta_{1}}{2\pi}p+A_{1}\zeta(2p) - (A_{1}-A_{2})\zeta(p),\\
w_{1,p}=\frac{i\eta_{1}}{2\pi}-\wp(2p)A_{1},\ w_{2,p}=-\frac{A_{1}}{2}%
\wp^{\prime}(2p).
\end{gather*}

Substituting these expansions into \eqref{U(z)'s def}, one obtains
\[
U(z;\tau) = U_{-3,p}(z-p)^{-3} + U_{-2,p}(z-p)^{-2} + U_{-1,p}(z-p)^{-1}
+O(1),
\]
with
\begin{gather*}
U_{-3,p}=2\left(  3q_{-1,p}w_{-1,p}+2q_{-2,p}w_{0,p}+\hat q_{-3,p}\right)  ,\\
U_{-2,p}=2\left(  2q_{0,p}w_{-1,p}+q_{-1,p}w_{0,p}+\hat q_{-2,p}\right)  ,\\
\end{gather*}
and
\[
U_{-1,p}=2\left(  q_{1,p}w_{-1,p}-q_{-1,p}w_{1,p}-2q_{-2,p}w_{2,p}+\hat
q_{-1,p}\right)  .
\]
The equations $U_{-3,p}=U_{-2,p}=0$ are equivalent to
\begin{equation}
\frac{dp(\tau)}{d\tau}=-2TA_{2}-\frac{i}{2\pi}(p\eta_{1}-\zeta(p))+\frac
{i\wp^{\prime\prime}(p)}{4\pi\wp^{\prime}(p)}, \label{U-3,p=0}%
\end{equation}

The coefficient $U_{-1,p}=0$ gives
\begin{equation}
\begin{aligned} &\frac{d}{d\tau} \left( T+\frac{\wp''(p)}{4\wp'(p)} \right)\\ =& -\frac{3}{4} A_1\wp'(2p)+ \left( T+\frac{\wp''(p)}{4\wp'(p)} \right) \left( \frac{i\eta_1}{2\pi} -A_1\wp(2p) \right) -A_2 q_{1,p}. \end{aligned} \label{d tau T+P''/4p'}%
\end{equation}
where $q_{1,p}$ is given by \eqref{q1,p}.

Similarly, the expansion at $z=-p$ gives
\begin{equation}
\frac{dp(\tau)}{d\tau} = 2TA_{1} - \frac{i}{2\pi}(p\eta_{1}-\zeta(p)) +
\frac{i\wp^{\prime\prime}(p)}{4\pi\wp^{\prime}(p)}, \label{U-3,-p=0}%
\end{equation}
and
\begin{equation}
\begin{aligned} &\frac{d}{d\tau} \left( T-\frac{\wp''(p)}{4\wp'(p)} \right)\\ ={}& \frac{3}{4} A_2\wp'(2p) + \left( T-\frac{\wp''(p)}{4\wp'(p)} \right) \left( \frac{i\eta_1}{2\pi} - A_2\wp(2p) \right) - A_1q_{1,-p}, \end{aligned} \label{d tau T-P''/4p'}%
\end{equation}
where
\begin{equation}
\begin{aligned} q_{1,-p} &= -n(n+1)\wp'(p) - \frac{3}{4}\wp'(2p) \\ &\quad + T\bigl(2\wp(p)-\wp(2p)\bigr) - \frac{\wp''(p)}{4\wp'(p)}\wp(2p). \end{aligned} \label{q1,-p}%
\end{equation}

Comparing \eqref{U-3,p=0} with \eqref{U-3,-p=0}, we obtain
\begin{equation}
(A_{1}+A_{2})T=0, \label{0602equ1}%
\end{equation}
and
\begin{equation}
\frac{dp(\tau)}{d\tau} = T(A_{1}-A_{2}) + \frac{i}{2\pi}\bigl(\zeta
(p)-p\eta_{1}\bigr) + \frac{i}{4\pi}\frac{\wp^{\prime\prime}(p)}{\wp^{\prime
}(p)}. \label{0602equ2}%
\end{equation}
Since $A_{1}+A_{2}\neq0$ by \eqref{A1+A2 and C}, it follows from
\eqref{0602equ1} that
\begin{equation}
T(\tau)\equiv0. \label{T(tau)=0}%
\end{equation}
Hence \eqref{0602equ2} reduces to
\begin{equation}
\frac{dp(\tau)}{d\tau} = \frac{i}{2\pi}\bigl(\zeta(p)-p\eta_{1}\bigr) +
\frac{i}{4\pi}\frac{\wp^{\prime\prime}(p)}{\wp^{\prime}(p)}.
\label{pdot reduced}%
\end{equation}

Combining \eqref{d tau T+P''/4p'} and \eqref{d tau T-P''/4p'}, and then using
$T=0$, we obtain
\begin{equation}
(A_{1}-A_{2}) \left(  \frac{3}{4}\wp^{\prime}(2p) + \frac{1}{4}\frac
{\wp^{\prime\prime}(p)}{\wp^{\prime}(p)}\wp(2p) \right)  + A_{1}q_{1,-p} +
A_{2}q_{1,p} = 0.
\end{equation}

Since
\[
q_{1,p}=-q_{1,-p} \qquad\text{when } T=0,
\]
we get
\begin{equation}
(A_{1}-A_{2}) \left(  \frac{3}{4}\wp^{\prime}(2p) + \frac{1}{4}\frac
{\wp^{\prime\prime}(p)}{\wp^{\prime}(p)}\wp(2p) - q_{1,p} \right)  = 0.
\end{equation}
Together with \eqref{q1,p}, this gives
\[
n(n+1)(A_{1}-A_{2})\wp^{\prime}(p)=0.
\]
Since $p\notin E_{\tau}[2]$, we have $\wp^{\prime}(p)\neq0$, and hence
\[
A_{1}=A_{2}.
\]
Together with \eqref{A1+A2 and C}, this yields
\begin{equation}
A_{1}=A_{2}=-\frac{i}{4\pi}. \label{A1=A2}%
\end{equation}

By \eqref{T(tau)=0} and \eqref{A1=A2}, equation \eqref{d tau T+P''/4p'} gives
\begin{equation}
\begin{aligned} &\frac{d}{d\tau} \left( \frac{\wp''(p)}{\wp'(p)} \right)\\ ={}& \frac{i}{\pi} \left[ n(n+1)\wp'(p) + \frac{3}{2}\wp'(2p) + \frac{1}{2} \frac{\wp''(p)}{\wp'(p)} \bigl(\wp(2p)+\eta_1\bigr) \right]. \end{aligned} \label{ratio derivative from U}%
\end{equation}

Using the addition formulae
\begin{equation}
\wp(2p) = -2\wp(p) + \frac{\wp^{\prime\prime}(p)^2}{4\wp^{\prime}(p)^2},
\nonumber\label{wp2p formula}%
\end{equation}
and
\begin{equation}
\wp^{\prime}(2p) = -\wp^{\prime}(p) + \frac{3\wp^{\prime\prime}(p)}%
{\wp^{\prime}(p)}\wp(p) - \frac{\wp^{\prime\prime}(p)^3}{4\wp^{\prime}(p)^3},
\nonumber\label{wp2p prime formula}%
\end{equation}
we reduce \eqref{ratio derivative from U} to
\begin{equation}
\begin{aligned} &\frac{d}{d\tau} \left( \frac{\wp''(p)}{\wp'(p)} \right) \\ ={}& \frac{i}{\pi} \left[ \left(n(n+1)-\frac{3}{2}\right)\wp'(p) + \frac{7}{2} \frac{\wp(p)\wp''(p)}{\wp'(p)} - \frac{1}{4} \frac{\wp''(p)^3}{\wp'(p)^3} + \frac{\eta_1\wp''(p)}{2\wp'(p)} \right]. \end{aligned} \label{ratio derivative from U reduced}%
\end{equation}

Using the identity
\begin{equation}
\frac{\partial}{\partial\tau}\wp(z) = -\frac{i}{4\pi} \left[  2\bigl(\zeta
(z)-z\eta_{1}\bigr)\wp^{\prime}(z) + 4\bigl(\wp(z)-\eta_{1}\bigr)\wp(z) -
\frac{2}{3}g_{2} \right] , \label{heat equation wp}%
\end{equation}
we compute
\[
\begin{aligned} \frac{\partial}{\partial\tau}\wp'(p) = -\frac{i}{4\pi} \left[ 2\bigl(\zeta(p)-p\eta_1\bigr)\wp''(p) + 6\bigl(\wp(p)-\eta_1\bigr)\wp'(p) \right], \end{aligned}\label{partial tau wp prime}%
\]
and
\[
\begin{aligned} \frac{\partial}{\partial\tau}\wp''(p) = -\frac{i}{4\pi} \left[ 2\bigl(\zeta(p)-p\eta_1\bigr)\wp'''(p) + 6\wp'(p)^2 + \bigl(4\wp(p)-8\eta_1\bigr)\wp''(p) \right]. \end{aligned}\label{partial tau wp second}%
\]
Combining these identities with \eqref{pdot reduced}, we obtain
\[
\begin{aligned} \frac{d}{d\tau}\wp'(p) &= \wp''(p)\frac{dp(\tau)}{d\tau} + \frac{\partial}{\partial\tau}\wp'(p) \\ &= \frac{i}{4\pi} \left[ \frac{\wp''(p)^2}{\wp'(p)} - 6\bigl(\wp(p)-\eta_1\bigr)\wp'(p) \right], \end{aligned}\label{total tau wp prime}%
\]
and
\[
\begin{aligned} \frac{d}{d\tau}\wp''(p) &= \wp'''(p)\frac{dp(\tau)}{d\tau} + \frac{\partial}{\partial\tau}\wp''(p) \\ &= \frac{i}{4\pi} \left[ \frac{\wp''(p)\wp'''(p)}{\wp'(p)} - 6\wp'(p)^2 - \bigl(4\wp(p)-8\eta_1\bigr)\wp''(p) \right]. \end{aligned}\label{total tau wp second}%
\]
Then
\begin{equation}
\begin{aligned} &\frac{d}{d\tau} \left( \frac{\wp''(p)}{\wp'(p)} \right) \\ ={}& \frac{ \left(\frac{d}{d\tau}\wp''(p)\right)\wp'(p) - \wp''(p)\left(\frac{d}{d\tau}\wp'(p)\right) } {\wp'(p)^2} \\ ={}& \frac{i}{\pi} \left[ -\frac{3}{2}\wp'(p) + \frac{7}{2} \frac{\wp(p)\wp''(p)}{\wp'(p)} - \frac{1}{4} \frac{\wp''(p)^3}{\wp'(p)^3} + \frac{\eta_1\wp''(p)}{2\wp'(p)} \right], \end{aligned} \label{ratio derivative direct}%
\end{equation}

Comparing \eqref{ratio derivative from U reduced} with
\eqref{ratio derivative direct}, we obtain
\[
n(n+1)\wp^{\prime}(p)=0,
\]
which is impossible because $n\geqslant1$ and $p\notin E_{\tau}[2]$. This
contradiction proves Theorem~\ref{thm, no isomonodromic deformation}.
\end{proof}

\begin{proof}
[Proof of Theorem~\ref{thm, Wn inde of p}]We argue by contradiction. Suppose
that $\widehat{W}_{n,p}(\sigma;\kappa)$ depends nontrivially on $\wp(p)$. Set
$x=\wp(p)$. Then we may choose
\[
\tau_{0}\in\mathbb{H},\qquad p_{0}\in E_{\tau_{0}}\setminus E_{\tau_{0}%
}[2],\qquad\sigma_{0}\in\mathbb{C}^{*},\qquad\kappa_{0}\in\mathbb{C }%
\]
such that, with $x_{0}=\wp(p_{0};\tau_{0})$,
\begin{equation}
\widehat{W}_{n,p_{0}}(\sigma_{0};\kappa_{0})=0,\label{0603equ1}%
\end{equation}
and
\begin{equation}
\left.  \frac{\partial\widehat{W}_{n,p}(\sigma_{0};\kappa_{0})} {\partial x}
\right| _{x=x_{0},\ \tau=\tau_{0}} \neq0.\label{partial x neq 0}%
\end{equation}
Choose the data also so that the corresponding system
\eqref{system for X, r and s} has a nonzero unramified solution $X_{0}$, since
the exceptional locus is a proper algebraic subset.

By Theorem~\ref{thm, equivalence of W and system}, \eqref{0603equ1} yields a
point $P_{0}\in\Gamma_{n,p_{0}}^{(1)}(\tau_{0})$ such that
\begin{equation}
\bigl( \sigma_{n,p_{0}}^{(1)}(P_{0}),\kappa(P_{0}) \bigr) = (\sigma_{0}%
,\kappa_{0}).\label{0603equ2}%
\end{equation}
Set
\[
(r_{0},s_{0}):=(r(P_{0}),s(P_{0})).
\]
Then
\[
\sigma_{0}=r_{0}+s_{0}\tau_{0}, \qquad\kappa_{0}=Z(r_{0},s_{0},\tau_{0}),
\]
and hence
\begin{equation}
\widehat{W}_{n,p_{0}} \bigl( r_{0}+s_{0}\tau_{0};\, Z(r_{0},s_{0},\tau_{0})
\bigr) =0.\label{0603equ3}%
\end{equation}

For the fixed pair $(r_{0},s_{0})$, define
\[
F(p,\tau):= \widehat{W}_{n,p} \bigl(
r_{0}+s_{0}\tau;\, Z(r_{0},s_{0},\tau) \bigr).
\]
Then \eqref{0603equ3} gives
\[
F(p_{0},\tau_{0})=0.
\]
Moreover, since $\wp^{\prime}(p_{0};\tau_{0})\neq0$, \eqref{partial x neq 0}
implies
\[
\frac{\partial F}{\partial p}(p_{0},\tau_{0}) = \left.  \frac{\partial
\widehat{W}_{n,p}(\sigma_{0};\kappa_{0})} {\partial x} \right| _{x=x_{0}%
,\ \tau=\tau_{0}} \wp^{\prime}(p_{0};\tau_{0}) \neq0.
\]

Thus, by the implicit function theorem, there exist a neighborhood
$O_{\tau_{0}}$ of $\tau_{0}$ and a holomorphic function
\[
p=p(\tau),\qquad p(\tau_{0})=p_{0},
\]
such that
\begin{equation}
\widehat{W}_{n,p(\tau)} \bigl( r_{0}+s_{0}\tau;\, Z(r_{0},s_{0},\tau)
\bigr) =0, \qquad\tau\in O_{\tau_{0}}.\label{0603equ4}%
\end{equation}
After shrinking $O_{\tau_{0}}$ if necessary, we may assume that
\[
p(\tau)\notin E_{\tau}[2], \qquad\tau\in O_{\tau_{0}}.
\]

Since $X_{0}\neq0$ is an unramified solution of \eqref{system for X, r and s}
at $\tau=\tau_{0}$, Theorem~\ref{thm, equivalence of W and system}, together
with the implicit function theorem, yields a holomorphic function
\[
X=X(\tau),\qquad X(\tau_{0})=X_{0},
\]
such that $X(\tau)\neq0$ and $X(\tau)$ solves \eqref{system for X, r and s}
for every $\tau\in O_{\tau_{0}}$.

After shrinking $O_{\tau_{0}}$ once more, we may choose a holomorphic branch
\[
T(\tau)=X(\tau)^{1/2}
\]
and a corresponding point
\[
P(\tau)=(T(\tau),C(\tau)) \in\Gamma_{n,p(\tau)}^{(1)}(\tau).
\]
By construction,
\[
(r(P(\tau)),s(P(\tau)))=(r_{0},s_{0}), \qquad\tau\in O_{\tau_{0}}.
\]
Hence $\mathrm{GLE}_{n}^{(1)} \bigl(p(\tau),T(\tau),\tau\bigr)$ forms a local
isomonodromic family with monodromy data $(r_{0},s_{0})$, contradicting
Theorem~\ref{thm, no isomonodromic deformation}.
\end{proof}

By Theorem \ref{thm, Wn inde of p}, write
\[
\widehat{W}_{n}(\sigma;\kappa)\in\mathbb{C}[\wp(\sigma),\wp^{\prime}%
(\sigma),e_{k}(\tau)][\kappa]. \label{Wn is poly in kappa}%
\]

\section{Identification with the Classical Lam\'e Curve}

\label{Isomonodromic deformation}

Let \(S\) denote the parameter space of pairs \((\tau,p)\) with
\(\tau\in\mathbb H\) and \(p\in E_\tau\setminus E_\tau[2]\), and let
\[
S^{\mathrm{sm}}\subset S
\]
be the open locus on which both
\(\hat{\Gamma}^{(1)}_{n,p}(\tau)\) and
\(\widetilde{\Gamma}_n(\tau)\) are nonsingular.
In this section we first construct the spectral isomorphism over
\(S^{\mathrm{sm}}\). The extension across \(\mathcal S\setminus\mathcal S^{\mathrm{sm}}\) will be proved in Proposition~\ref{prop:extension-singular-fibers}, using holomorphic dependence and the explicit coordinate translation.

 At the beginning, we
review the pre-modular form ${Z}_{n}(r,s,\tau)$ for the classical Lam\'{e}
potential from \cite{LW}. Recall the addition map
\[
\tilde{\sigma}_{n}:\overline{\widetilde{\Gamma}_{n}(\tau)}\to E_{\tau}
\]
for the classical Lam\'{e} potential introduced in \eqref{add map of lame}.
Since $\deg\tilde{\sigma}_{n}=\frac{n(n+1)}{2}$, the field extension
$K(\overline{\widetilde{\Gamma}_{n}(\tau)})/K(E_{\tau})$ has degree
\[
[K(\overline{\widetilde{\Gamma}_{n}(\tau)}):K(E_{\tau})]=\frac{n(n+1)}{2}.
\]

Define
\[
\tilde{\kappa}(\widetilde{P}):=\zeta(\tilde{\sigma}_{n}(\widetilde{P}%
))-c(\widetilde{P}),
\]
where $\widetilde{P}\in\overline{\widetilde{\Gamma}_{n}(\tau)}$ and
\[
c(\widetilde{P}):=\sum_{j=1}^{n}\zeta(\tilde{a}_{j}(\widetilde{P})).
\]
Here
\[
\tilde{\mathbf{a}}(\widetilde{P}) = \{\tilde a_{1}(\widetilde{P}%
),\ldots,\tilde a_{n}(\widetilde{P})\}
\]
denotes the zero set of the Baker--Akhiezer function $\psi(\widetilde{P};z)$,
as in \eqref{lame's zero set}. Then $\tilde{\kappa}(\widetilde{P})$ is a
rational function on $\overline{\widetilde{\Gamma}_{n}(\tau)}$. It is shown in
\cite{LW} that $\tilde\kappa(\widetilde{P})$ is in fact a primitive element of
the field extension and therefore its minimal polynomial, denoted by
$\widetilde{W}_{n}(\sigma;\kappa)$, certainly exists and belongs to
\[
\mathbb{C}[\wp(\sigma),\wp^{\prime}(\sigma),g_{2},g_{3}][\kappa]
\]
of degree $\frac{n(n+1)}{2}$ in $\kappa$.

Set the weights of $\kappa$, $\wp(\sigma)$, $e_{k}(\tau)$ and $\wp^{\prime
}(\sigma)$ to be $1,2,2,3$, respectively. Then $\widetilde{W}_{n}%
(\sigma;\kappa)$ is of homogeneous weight $\frac{n(n+1)}{2}$.

For $(r,s)\in\mathbb{C}^{2}$ with $\sigma=r+s\tau\neq0$, the pre--modular form
associated with the Lam\'{e} equation L$_{n}(\widetilde{B},\tau)$ is given by
\begin{equation}
{Z}_{n}(r,s,\tau):=\widetilde{W}_{n}\bigl(\sigma;Z(r,s,\tau)\bigr).
\label{def, lame's pre modular form}%
\end{equation}

We use the following result from \cite{LW}.

\begin{lemma}
\cite{LW} \label{lem, def, lame's pre modular form} Let $\tau\in\mathbb{H}$
and $(r,s)\in\mathbb{C}^{2}\setminus\frac{1}{2}\mathbb{Z}^{2}$ with
$r+s\tau\neq0$. Then ${Z}_{n}(r,s,\tau)=0$ if and only if there exists a
unique $\widetilde{B}\in\mathbb{C}$ such that the Lam\'{e} equation
L$_{n}(\widetilde{B},\tau)$ is completely reducible with monodromy data
$(r,s)$.
\end{lemma}

The main result of this section shows that the two polynomials $\widehat{W}%
_{n}(\sigma;\kappa)$ and $\widetilde{W}_{n}(\sigma;\kappa)$, which arise from
different algebraic constructions, in fact coincide.

\begin{theorem}
\label{thm, three W coincide} We have
\begin{equation}
\widehat{W}_{n}(\sigma;\kappa)=\widetilde{W}_{n}(\sigma;\kappa).
\label{two Wn coincides}%
\end{equation}

Consequently, $\kappa(\hat P)$ is a primitive generator of the finite
extension of rational function field $K(\overline{\hat{\Gamma}_{n,p}%
^{(1)}(\tau)})$ over $K(E_{\tau})$. Equivalently,
\[
K\bigl(\overline{\hat{\Gamma}_{n,p}^{(1)}(\tau)}\bigr)
= K(E_{\tau})\bigl(\kappa(\hat P)\bigr).
\]

\end{theorem}

For the proof, let $\sigma_{0}=r+s\tau\notin E_{\tau}[2]$ and $\kappa_{0}%
\in\mathbb{C}$ satisfy
\begin{equation}
\widehat{W}_{n}(\sigma_{0};\kappa_{0})=0. \label{Wn(sigma,kappa)=0}%
\end{equation}
Since $\widehat{W}_{n}(\sigma;\kappa)$ is independent of $p$,
Theorem~\ref{thm, equivalence of W and system} implies that, for each $p\in
E_{\tau}\setminus E_{\tau}[2]$, there exists $\hat{P}(p)=(X(p),C(p))\in
\hat{\Gamma}_{n,p}^{(1)}(\tau)$ such that
\begin{equation}
(\hat{\sigma}_{n,p}^{(1)}(\hat{P}(p)),\hat{\kappa}(\hat{P}(p)))=(\sigma
_{0},\kappa_{0}). \label{sigma and kappa are inva}%
\end{equation}

Choose $P(p)=(T(p),C(p))\in\pi^{-1}(\hat{P}(p))$ and consider the
corresponding equation $\mathrm{GLE}_{n}(p,T(p),\tau)$. Since
\begin{equation}
\sigma_{n,p}^{(1)}(P(p))=\hat{\sigma}_{n,p}^{(1)}(\hat{P}(p))=\sigma
_{0}\not \in E_{\tau}[2],
\end{equation}
$\mathrm{GLE}_{n}^{(1)}(p,T(p),\tau)$ is completely reducible. Let
$(r(P(p)),s(P(p)))$ denote its monodromy data. It then follows from
\eqref{sigma and kappa are inva} that
\begin{equation}
(r(P(p)),s(P(p)))\equiv(r,s).
\end{equation}
Thus the family $\mathrm{GLE}_{n}^{(1)}(p,T(p),$ $\tau)$ is isomonodromic as
$p\in E_{\tau}\setminus E_{\tau}[2]$ varies.

\begin{lemma}
\label{lem, gle to lame r,s} Under the same assumptions above, let
GLE$_{n}^{(1)}(p,T(p),$ $\tau)$ be the isomonodromic family constructed there.
Then, as $p\rightarrow0$, GLE$_{n}^{(1)}(p,T(p),\tau)$ converges to
L$_{n}(\widetilde{B},\tau)$ for some $\widetilde{B}\in\mathbb{C}$. Moreover,
the limiting equation L$_{n}(\widetilde{B},\tau)$ is completely reducible and
has the same monodromy data $\left(  r,s\right)  $.
\end{lemma}

\begin{proof}

Let $\psi(P(p);z)$ be the associated Baker--Akhiezer function of
GLE$_{n}^{(1)}$ $(p,T(p),\tau)$, where $P(p)=(T(p),C(p))\in\Gamma_{n,p}%
^{(1)}(\tau)$. By \eqref{Ba is given by a(P) and c} and
Lemma~\ref{lem, relation bet r(P),s(P) and aj}, the ansatz $(\mathbf{a}%
(p),c(p))$ satisfies
\begin{equation}%
\begin{cases}
\sum_{j=1}^{n+1}a_{j}(p)=r+s\tau,\\
c(p)=r\,\eta_{1}+s\,\eta_{2}.
\end{cases}
\label{fixed (r,s) and aj(p)}%
\end{equation}

In particular,
\begin{equation}
c^{0}:=\lim_{p\to0}c(p)=r\eta_{1}+s\eta_{2}%
\end{equation}
converges since $(r,s)$ is fixed.

By Lemmas~\ref{lem, q converges iff c converges}
and~\ref{lem, gle converges to lame}, it follows that GLE$_{n}^{(1)}%
(p,T(p),\tau)$ converges to L$_{n}(\widetilde{B},\tau)$ as $p\to0$ for some
$\widetilde{B}\in\mathbb{C}$, where $\widetilde{B}$ is determined by \eqref{X(p)'s behavior at 0}.

As in \eqref{a(p) as p to 0}, we may assume that $\mathbf{a}(p)\to
\mathbf{a}^{0}$ as $p\to0$ . Then Proposition~\ref{prop, B and a}(i) implies
$\mathbf{a}^{0}$ satisfies \eqref{a0 lies in lame's zero}, where $\{a_{1}%
^{0},\ldots,a_{n}^{0}\}$ is the zero divisor of the Baker--Akhiezer function
associated with $\mathrm{L}_{n}(\widetilde{B},\tau)$. Therefore,
\begin{equation}
\psi_{0}(z):=e^{c^{0}z}\frac{\prod_{j=1}^{n}\sigma(z-a_{j}^{0})}{\sigma
(z)^{n}}%
\end{equation}
is precisely the Baker--Akhiezer function of $\mathrm{L}_{n}(\widetilde{B}%
,\tau)$.

Let $(\tilde r,\tilde s)\in\mathbb{C}^{2}/\mathbb{Z}^{2}$ denote the monodromy
data of $\mathrm{L}_{n}(\widetilde{B},\tau)$. Then
\begin{equation}%
\begin{cases}
\tilde{r}+\tilde{s}\tau=\sum_{j=1}^{n}a_{j}^{0}=\lim_{p\to0}\sum_{j=1}%
^{n+1}a_{j}(p)\equiv r+s\tau,\\
\tilde{r}\eta_{1}+\tilde{s}\eta_{2}=c^{0}=\lim_{p\to0}c(p)\equiv r\eta
_{1}+s\eta_{2}.
\end{cases}
\label{lame's mono data}%
\end{equation}

Comparing this with \eqref{lame's mono data} yields
\begin{equation}%
\begin{cases}
\displaystyle \tilde{r}+\tilde{s}\tau=r+s\tau,\\
\tilde{r}\eta_{1}+\tilde{s}\eta_{2}=r\eta_{1}+s\eta_{2}.
\end{cases}
\end{equation}

By Legendre relation \eqref{legendre relation}, we conclude that
\[
(\tilde{r},\tilde{s})=(r,s).
\]

Since $(r,s)\in\mathbb{C}^{2}\setminus\frac12\mathbb{Z}^{2}$, the limiting
Lam\'e equation $\mathrm{L}_{n}(\widetilde{B},\tau)$ is completely reducible
with the same monodromy data $(r,s)$.
\end{proof}

\begin{proof}
[Proof of Theorem~\ref{thm, three W coincide}]It suffices first to consider
generic $\sigma_{0}\notin E_{\tau}[2]$ for which both polynomials have only
simple zeros.

Let $(\sigma_{0},\kappa_{0})\in E_{\tau}\setminus E_{\tau}[2]\times\mathbb{C}$
satisfy \eqref{Wn(sigma,kappa)=0}. By the previous analysis, the family
$\mathrm{GLE}_{n}^{(1)}(p,T(p),\tau)$ is completely reducible with monodromy
data $(r,s)$, where $\sigma_{0}=r+s\tau$. Hence, by
Theorem~\ref{thm, monodromy sys in X} and
Theorem~\ref{thm, equivalence of W and system}, $X(p)=T(p)^{2}$ solves the
system \eqref{system for X, r and s} with the pair $(\sigma_{0},\kappa_{0})$.
Moreover, comparing \eqref{mono system for X, r(P) and s(P)} with
\eqref{system for X, r and s} leads to
\begin{equation}
\kappa_{0}=\frac{\hat{P}_{5}(X(p),p;\tau)}{\hat{P}_{6}(X(p),p;\tau)}%
\sqrt{-\widehat{Q}_{n,p}^{(1)}(X(p))}=\kappa(P(p))=Z(r,s,\tau).
\label{kappa0=kappa(P)}%
\end{equation}

As $p\to0$, Lemma~\ref{lem, gle to lame r,s} implies that this isomonodromic
family converges to $\mathrm{L}_{n}(\widetilde{B},\tau)$ for some
$\widetilde{B}\in\mathbb{C}$. Moreover, the limiting equation $\mathrm{L}%
_{n}(\widetilde{B},\tau)$ is also completely reducible with the same monodromy
data $(r,s)$. Therefore, Lemma~\ref{lem, def, lame's pre modular form} yields
\begin{equation}
{Z}_{n}(r,s,\tau)=0. \label{lame's Z=0}%
\end{equation}

Combining \eqref{kappa0=kappa(P)}, \eqref{lame's Z=0} and
\eqref{def, lame's pre modular form}, we obtain
\begin{equation}
\widetilde{W}_{n}(\sigma_{0};\kappa_{0})=\widetilde{W}_{n}(\sigma
_{0};Z(r,s,\tau))={Z}_{n}(r,s,\tau)=0.
\end{equation}
Thus, for generic $\sigma_{0}$, every zero of $\widehat{W}_{n}(\sigma
_{0};\kappa)$ is also a zero of $\widetilde{W}_{n}(\sigma_{0};\kappa)$. It
follows that
\[
\widehat{W}_{n}(\sigma_{0};\kappa)\mid\widetilde{W}_{n}(\sigma_{0}%
;\kappa)\quad\text{in }K(E_{\tau})[\kappa].
\]

Since both $\widetilde{W}_{n}(\sigma;\kappa)$ and $\widehat{W}_{n}%
(\sigma;\kappa)$ are irreducible polynomials over $K(E_{\tau})$ and monic in
$\kappa$, we conclude that
\[
\widehat{W}_{n}(\sigma;\kappa) = \widetilde{W}_{n}(\sigma;\kappa).
\]
This proves \eqref{two Wn coincides}.

Consequently,
\[
\deg_{\kappa}\widehat{W}_{n}(\sigma;\kappa) =\deg_{\kappa}\widetilde{W}%
_{n}(\sigma;\kappa)= \frac{n(n+1)}2.
\]
Hence,
\[
[K(E_{\tau})(\kappa(\hat P)):K(E_{\tau})] =[K(\overline{\hat{\Gamma}%
_{n,p}^{(1)}(\tau)}):K(E_{\tau})]= \frac{n(n+1)}{2} .
\]

Observe that
\[
K(E_{\tau})\subset K(E_{\tau})(\kappa(\hat P))\subset K(\overline{\hat{\Gamma
}_{n,p}^{(1)}(\tau)}),
\]
it follows that
\begin{equation}
[K(\overline{\hat{\Gamma}_{n,p}^{(1)}(\tau)}):K(E_{\tau})(\kappa(\hat P)]=1.
\label{primitive generator}%
\end{equation}
Then
\[
K(E_{\tau})(\kappa(\hat P))=K(\overline{\hat{\Gamma}_{n,p}^{(1)}(\tau)}).
\]

Therefore, $\kappa(\hat P)$ is a primitive generator of $K(\overline
{\hat{\Gamma}_{n,p}^{(1)}(\tau)})$ over $K(E_{\tau})$.
\end{proof}

Since
\[
K\!\left( \hat{\Gamma}^{(1)}_{n,p}(\tau)\right)  = K(E_{\tau})(\kappa),
\]
and the minimal polynomial of $\kappa$ over $K(E_{\tau})$ is $W_{n}%
(\sigma;\kappa)$, the field $K(E_{\tau})(\kappa)$ is precisely the function
field of the algebraic curve
\[
C_{n}:= \left\{  (\sigma,\kappa)\in E_{\tau}\times\mathbb{P}^{1}\mid W_{n}%
(\sigma;\kappa)=0 \right\} .
\]
Thus,
\[
K\!\left( \hat{\Gamma}^{(1)}_{n,p}(\tau)\right)  = K(C_{n}).
\]
Hence both $\hat\Gamma_{n,p}^{(1)}(\tau)$ and $\widetilde{\Gamma}_{n}(\tau)$
are normalizations of $\mathcal{C}_{n}$.
By the uniqueness of normalization, there is a unique isomorphism
\begin{equation}
\iota_{p}: \overline{\hat{\Gamma}_{n,p}^{(1)}(\tau)} \longrightarrow
\overline{\widetilde{\Gamma}_{n}(\tau)},\ \hat{P}\longmapsto\widetilde{P}%
,\label{isomorphism}%
\end{equation}
compatible with the normalization maps. In particular,
\begin{equation}
    \tilde{\sigma}_{n}\circ\iota_{p} = \hat{\sigma}^{(1)}_{n,p}, \qquad
\tilde{\kappa}\circ\iota_{p} = \kappa,
\label{isomorphism sigma and kappa}
\end{equation}
and
\begin{equation}
(\hat{\sigma}_{n,p}^{(1)}(\hat{P}),\kappa(\hat{P}))=(\tilde{\sigma}%
_{n}(\widetilde{P}),\tilde{\kappa}(\widetilde{P}%
)).\label{isomorphism property}%
\end{equation}

\begin{theorem}
\label{thm, com red iff Zn=0} Let $\tau\in\mathbb{H}$ and $(r,s)\in
\mathbb{C}^{2}\setminus\frac12\mathbb{Z}^{2}$. Then
\[
Z_{n}(r,s,\tau)=0
\]
if and only if, for every $p\in E_{\tau}\setminus E_{\tau}[2]$, there exists a
unique point
\[
\hat P(p)=(X(p),C(p)) \in\hat\Gamma_{n,p}^{(1)}(\tau)
\]
such that the corresponding equation
\[
\mathrm{GLE}_{n}^{(1)}(p,T(p),\tau), \qquad T(p)^{2}=X(p),
\]
is completely reducible with monodromy data
\[
(r(P(p)),s(P(p)))=(r,s).
\]
In this case, $\mathrm{GLE}_{n}^{(1)}(p,T(p),\tau)$ forms an isomonodromic
family with monodromy data $(r,s)$ as $p$ varies in $E_{\tau}\setminus
E_{\tau}[2]$.
\end{theorem}

\begin{proof}
Suppose first that
\[
Z_{n}(r,s,\tau)=0.
\]
Then Lemma~\ref{lem, def, lame's pre modular form} yields a unique point
\[
\widetilde{P}_{0}=(\widetilde{B},\widetilde{C}) \in\widetilde{\Gamma}_{n}%
(\tau)
\]
such that $\mathrm{L}_{n}(\widetilde{B},\tau)$ is completely reducible with
monodromy data $(r,s)$. In particular,
\[
\bigl(
\tilde\sigma_{n}(\widetilde{P}_{0}), \tilde\kappa(\widetilde{P}_{0}) \bigr)
= (\sigma_{0},\kappa_{0}),
\]
where
\begin{equation}
\sigma_{0}=r+s\tau, \qquad\kappa_{0}=Z(r,s,\tau).\label{sigma0 and kappa0}%
\end{equation}

By the uniqueness of $\widetilde{P}$, \eqref{isomorphism} and
\eqref{isomorphism property},
\[
\hat{P}(p)=\iota_{p}^{-1}(\widetilde{P})
\]
is the unique point satisfying
\[
\bigl(
\hat\sigma_{n,p}^{(1)}(\hat P(p)), \kappa(\hat P(p)) \bigr)
= (\sigma_{0},\kappa_{0}).
\]

Choose any $P(p)$ over $\hat P(p).$ By
Lemma~\ref{lem, relation bet r(P),s(P) and aj},
\[
r(P(p))+s(P(p))\tau= \hat\sigma_{n,p}^{(1)}(\hat P(p)) = r+s\tau,
\]
and
\[
r(P(p))\eta_{1}+s(P(p))\eta_{2} = \zeta(\sigma_{0})-\kappa_{0} = r\eta
_{1}+s\eta_{2}.
\]
Hence
\[%
\begin{cases}
(r(P(p))-r)+(s(P(p))-s)\tau=0,\\[2mm]%
(r(P(p))-r)\eta_{1}+(s(P(p))-s)\eta_{2}=0.
\end{cases}
\]
By Legendre's relation~\eqref{legendre relation},
\[
(r(P(p)),s(P(p)))=(r,s).
\]

Thus, for every $p\in E_{\tau}\setminus E_{\tau}[2]$, the corresponding
equation $\mathrm{GLE}_{n}^{(1)}(p,T(p),$ $\tau)$ is completely reducible with
monodromy data $(r,s)$. Consequently, these equations form an isomonodromic
family as $p$ varies in $E_{\tau}\setminus E_{\tau}[2]$.

Conversely, suppose that for every $p\in E_{\tau}\setminus E_{\tau}[2]$ there
exists such a point $\hat P(p)$ and that the corresponding equation is
completely reducible with monodromy data $(r,s)$. Then
\[
\hat\sigma_{n,p}^{(1)}(\hat P(p)) = r+s\tau=\sigma_{0}
\]
and
\[
\kappa(\hat P(p)) = Z(r,s,\tau)=\kappa_{0}.
\]
Hence
\[
\widehat{W}_{n}(\sigma_{0};\kappa_{0})=0.
\]
By Theorem~\ref{thm, three W coincide},
\[
\widehat{W}_{n}(\sigma_{0};\kappa_{0}) = \widetilde{W}_{n}(\sigma_{0}%
;\kappa_{0}).
\]
Therefore,
\[
Z_{n}(r,s,\tau) = \widetilde{W}_{n} \bigl(
r+s\tau;Z(r,s,\tau) \bigr)
=0.
\]
This completes the proof.
\end{proof}

\section{The Monodromy Correspondence}

\label{Section, Mono Equiv} 
We now identify the monodromy correspondence
between $\mathrm{GLE}_{n}^{(1)}(p,T,\tau)$ and $\mathrm{L}_{n}(\widetilde{B}%
,\tau)$ explicitly.

We first treat the completely reducible case. Suppose $\mathrm{GLE}_{n}%
^{(1)}(p,T,\tau)$ is completely reducible with monodromy data $(r,s)$. By
Theorem~\ref{thm, com red iff Zn=0},
\[
Z_{n}(r,s,\tau)=0.
\]
Moreover, for each $p\in E_{\tau}\setminus E_{\tau}[2]$, there exists a unique
point
\[
\hat{P}(p)=(X(p),C(p))\in\hat{\Gamma}_{n,p}^{(1)}(\tau),
\]
such that for $P(p)=(T(p),C(p))\in\Gamma_{n,p}^{(1)}(\tau)$ with
$T(p)^{2}=X(p)$, the family
\[
\mathrm{GLE}_{n}^{(1)}(p,T(p),\tau), \qquad p\in E_{\tau}\setminus E_{\tau
}[2],
\]
is completely reducible with common monodromy data $(r,s)$.  Hence $X(p)$ is
single-valued on $E_{\tau}\setminus E_{\tau}[2]$.  Since $X(p)$ solves the
algebraic system \eqref{system for X, r and s} with $(\sigma,\kappa
)=(\sigma_{0},\kappa_{0}),$ where
\[
\sigma_{0}=r+s\tau, \qquad\kappa_{0}=Z(r,s,\tau),
\]
it follows that $X(p)$ is holomorphic on $E_{\tau}\setminus E_{\tau}[2]$.
Lemma~\ref{lem, p to wk/2} gives the required extension across the nonzero half-periods.

\begin{lemma}
\label{lem, p to wk/2} Fix $(r,s)\in\mathbb{C}^{2}\setminus\frac{1}%
{2}\mathbb{Z}^{2}$ such that $Z_{n}(r,s,\tau)=0$. Then the corresponding
parameter $T(p)$ of the isomonodromic family $\mathrm{GLE}_{n}^{(1)}%
(p,T(p),\tau)$ converges as $p\to\frac{\omega_{k}}{2},k=1,2,3$.
\end{lemma}

\begin{proof}
The argument of \cite{Chen-Kuo-Lin-Lame I} gives the first equivalence of
Lemma~\ref{lem, q converges iff c converges} also holds as $p\to\frac
{\omega_{k}}{2}$: the family $\mathrm{GLE}_{n}^{(1)}(p,T(p),\tau)$ converges
if and only if $c(p)$ converges.

Since the monodromy data $(r,s)$ are fixed,
Lemma~\ref{lem, relation bet r(P),s(P) and aj} gives
\[
c(p)=r\eta_{1}+s\eta_{2},
\]
which is independent of $p$. Hence $\mathrm{GLE}_{n}^{(1)}(p,T(p),\tau)$
converges as $p\to\frac{\omega_{k}}{2}$. In particular, the convergence of the
potential $q_{n}^{(1)}(z;p,T(p))$ implies that $T(p)$ converges.
\end{proof}

\begin{theorem}
\label{thm, X(p)=n*n+1} Let $\tau\in\mathbb{H}$ and $(r,s)\in\mathbb{C}%
^{2}\setminus\frac{1}{2}\mathbb{Z}^{2}$ be fixed such that the family
$\mathrm{GLE}_{n}^{(1)}(p,T(p),\tau)$ is completely reducible with monodromy
data $(r,s)$. Then $X(p)=T(p)^{2}$ satisfies
\begin{equation}
X(p)=n(n+1)\wp(p)+\widetilde{B}, \label{260715equ1}%
\end{equation}
where $\widetilde{B}$ is the parameter of the limiting Lam\'e equation
$\mathrm{L}_{n}(\widetilde{B},\tau)$ obtained in
Lemma~\ref{lem, gle to lame r,s}. Furthermore, equations $\mathrm{GLE}%
_{n}^{(1)}(p,T(p),\tau)$ and $\mathrm{L}_{n}(\widetilde{B},\tau)$ are
monodromy equivalent.
\end{theorem}

\begin{proof}
By the preceding discussion, $X(p)$ is single-valued and holomorphic on $p\in
E_{\tau}\setminus\{0\}$. We analyze the behavior of $X(p)$ as $p\to0$.

By Lemma~\ref{lem, gle to lame r,s}, as $p\to0$, $\mathrm{GLE}_{n}%
^{(1)}(p,T(p),\tau)$ converges to $\mathrm{L}_{n}(\widetilde{B},\tau)$ for some
$\tilde{B}\in\mathbb{C}$, which is completely reducible with the same
monodromy data $(r,s)$. Then Lemma~\ref{lem, q converges iff c converges} and
Lemma~\ref{lem, gle converges to lame} imply that $X(p)$ satisfies the
asymptotic form
\begin{equation}
X(p)=\frac{n(n+1)}{p^{2}}+\widetilde{B}+O(p). \label{260724equ1}%
\end{equation}
Hence, $p=0$ is the unique pole of $X(p)$, and it is a double pole.

Then the elliptic function $X(p)-n(n+1)\wp(p) $ has no poles on $E_{\tau}$ and
is therefore constant. Comparing this with \eqref{260724equ1}, we conclude
\[
X(p)-n(n+1)\wp(p)=\widetilde{B}.
\]
This proves \eqref{260715equ1}.

By Lemma~\ref{lem, gle to lame r,s}, both $\mathrm{GLE}_{n}^{(1)}%
(p,T(p),\tau)$ and $\mathrm{L}_{n}(\widetilde{B},$ $\tau)$ have monodromy data
$(r,s)$, and are therefore monodromy equivalent.
\end{proof}

\begin{proposition}[Extension across singular fibers]\label{prop:extension-singular-fibers}
The spectral isomorphism \eqref{isomorphism}, with its compatibility with the addition
maps and $\kappa$, extends to singular fibers for every $\tau\in\mathbb H$ and
$p\in E_\tau\setminus E_\tau[2]$.
\end{proposition}

\begin{proof}
The normalized elliptic solution $\Phi^{(1)}_{e,p}(z;T)$ and hence
the coefficients of $\widehat Q^{(1)}_{n,p}(X;\tau)$ depend
holomorphically on $(\tau,p)$ away from $E_\tau[2]$. On the dense
nonsingular locus, \eqref{isomorphism} and Theorem~\ref{thm, X(p)=n*n+1} identify the branch values
under $\widetilde B=X-n(n+1)\wp(p;\tau)$. The leading
coefficients agree, so
\[
\widehat Q^{(1)}_{n,p}(X;\tau)
=\widetilde Q_n\bigl(X-n(n+1)\wp(p;\tau);\tau\bigr).
\]
Holomorphic continuation extends this identity to all admissible
parameters. The roots correspond with multiplicities, and
\[
(X,C)\mapsto\bigl(X-n(n+1)\wp(p;\tau),C\bigr)
\]
and its inverse define isomorphisms on singular fibers as well.
Choose the sign of $C$ consistently with \eqref{isomorphism}.

On $C\neq0$, the Baker--Akhiezer representation gives locally
holomorphic period multipliers. Their equality extends from the
dense nonsingular locus. By \eqref{algebraic relations between ai,r(P),s(P),c(P)} and
\eqref{kappa(P)'s express}, compatibility with
the addition maps and $\kappa$ holds on a dense open subset of
each fiber, hence on the whole fiber as identities of holomorphic
maps and meromorphic functions.
\end{proof}

The intrinsic construction of Sections~\ref{Log-free variety}--\ref{Isomonodromic deformation} determines the
monodromy correspondence on $V^{(1)}_{n,p}(\tau)\simeq\mathbb A^1_T$.
Its coordinate formula \eqref{260715equ1} now yields a first-order realization.

Theorem~\ref{thm, X(p)=n*n+1} identifies the accessory parameter without assuming a
transformation between the two equations. Its conclusion now reveals
such a transformation. Set
\begin{equation}\label{rev:eq:darboux-parameters}
 N=n(n+1),\qquad f_p(z)=\wp(z)-\wp(p),\qquad
 \widetilde B=T^2-N\wp(p).
\end{equation}
Then $q_L=N\wp(z)+\widetilde B=T^2+Nf_p$, and
\begin{equation}\label{rev:eq:factorization}
 (\partial_z+T)(\partial_z-T)y=Nf_p\,y.
\end{equation}
With $u=y'-Ty$, this becomes
\begin{equation}\label{rev:eq:first-order-system}
 y'=Ty+u,\qquad u'=Nf_p y-Tu.
\end{equation}
Eliminating $y$ introduces the logarithmic derivative $f_p'/f_p$.
Since the divisor $\operatorname{div}(f_p)=[p]+[-p]-2[0]$,
 the new scalar equation has
the prescribed additional pair $\pm p$. The normalization
$v=f_p^{-1/2}u$ removes its first derivative.

The inverse step is encoded in the apparency conditions. For a
solution $v$ of $\mathrm{GLE}^{(1)}_n(p,T,\tau)$, put $u=f_p^{1/2}v$
and $a_p=\wp''(p)/(4\wp'(p))$. The exponents of $u$ at
$\pm p$ are $0,2$, and the first Frobenius coefficients give
\[
 u'(p)+(T_2-a_p)u(p)=0,\qquad
 u'(-p)+(T_1+a_p)u(-p)=0.
\]
On the non-even component $T_2-a_p=T_1+a_p=T$; thus
$(u'+Tu)/f_p$ is regular at both points.

Write $g_p=f_p'/f_p$ and $h_p=f_p^{-1/2}$. The square root is
continued with the convention of $\Psi_p$ in~\eqref{Psi,p}--\eqref{Psi is elliptic}.

\begin{proposition}[Darboux realization]
\label{rev:prop:darboux}
Fix $n\geqslant1$, $\tau\in\mathbb H$, and
$p\in E_\tau\setminus E_\tau[2]$. For $T\in\mathbb C$, let
$\widetilde B$ be given by~\eqref{rev:eq:darboux-parameters}. Then the
first-order operators
\begin{equation}
\label{Darboux transformation}
\begin{aligned}
\mathcal D_{p,T}y
&=h_p(y'-Ty),\
\mathcal E_{p,T}v
&=\frac{h_p}{N}
\left(v'+\left(T+\frac{g_p}{2}\right)v\right)
\end{aligned}
\end{equation}
are mutually inverse isomorphisms between the solution spaces of
$L_n(\widetilde B,\tau)$ and
$\mathrm{GLE}_n^{(1)}(p,T,\tau)$. These transformations preserve the
period monodromy. 
\end{proposition}

\begin{proof}
Eliminating $y$ from~\eqref{rev:eq:first-order-system} gives
\begin{equation}
 u''-g_pu'=(q_{\mathrm L}+Tg_p)u.\nonumber
\end{equation}
Since $h_p'/h_p=-g_p/2$, the function $v=h_pu$ satisfies
\begin{equation}
 v''=q_1v,\qquad
 q_1=q_{\mathrm L}+Tg_p-\frac12g_p'+\frac14g_p^2.\nonumber
\end{equation}
The Weierstrass addition formulas yield
\begin{align}
 g_p(z)&=\zeta(z+p)+\zeta(z-p)-2\zeta(z),\nonumber\\
 -\frac12g_p'(z)+\frac14g_p(z)^2
   &=\frac34\bigl(\wp(z+p)+\wp(z-p)\bigr)\notag\\
   &\quad-a_p\bigl(\zeta(z+p)-\zeta(z-p)\bigr)
          +2a_p\zeta(p)+\frac32\wp(p).
          \nonumber
\end{align}
Together with~\eqref{rev:eq:darboux-parameters}, these identify
$q_1$ with the potential $q_n^{(1)}(z;p,T)$ in~\eqref{potential, noneven}--\eqref{B(T),noneven}.
Thus $\mathcal D_{p,T}$ maps Lam\'e solutions to non-even solutions.

The second equation in~\eqref{rev:eq:first-order-system} gives
\[
 y=\frac{u'+Tu}{Nf_p}
   =\frac{h_p}{N}\left(v'+\left(T+\frac{g_p}{2}\right)v\right).
\]
Moreover,
\[
 \begin{pmatrix}u\\u'\end{pmatrix}
 =\begin{pmatrix}-T&1\\q_{\mathrm L}&-T\end{pmatrix}
       \begin{pmatrix}y\\y'\end{pmatrix}.
\]
Consequently,
\begin{equation}
    W(h_pu_1,h_pu_2)
  =h_p^2(T^2-q_{\mathrm L})W(y_1,y_2)
  =-N\,W(y_1,y_2).
  \label{Darboux Wronskian}
\end{equation}

Thus a fundamental system remains independent for every $T$, and
the operators are inverse isomorphisms.

Finally,
\[
 \Psi_p(z)^2=-\frac{1}{\sigma(p)^2 f_p(z)}
\]
shows that $h_p$ differs from $\Psi_p$ by a nonzero constant
depending only on $p$. By~\eqref{Psi is elliptic},  continuation along the chosen cycles
$\ell_j$ gives
\[
 h_p^{\ell_j}=h_p,\qquad j=1,2.
\]
The operators therefore preserve both period matrices.
\end{proof}

\begin{proof}[Completion of the proof of Theorem~\ref{Main Thm 1}]
The coordinate identity holds under the spectral isomorphism by
Theorem~\ref{thm, X(p)=n*n+1} and continuation. At a non-completely reducible
parameter $T_0$, apply $\mathcal E_{p,T_0}$ to a fundamental
system. Proposition~\ref{rev:prop:darboux} gives a Lam\'e fundamental system with the
same period matrices, including the datum $D=\infty$. The converse
uses $\mathcal D_{p,T_0}$. This pointwise argument does not divide
by the spectral polynomial and also covers discriminant parameters.
\end{proof}

\medskip

\noindent\textit{Complementary descriptions of the monodromy data.}
Non-completely reducible monodromy also admits descriptions by
period integrals and limits of completely reducible data.
Fix such a parameter $T_0$ and write
$\Phi_e(z;T)=\Phi^{(1)}_{e,p}(z;$ $T)$. Choose an ordinary base
point and period cycles avoiding the zeros of $\Phi_e(\cdot;T_0)$.
For $T$ near $T_0$, set
\begin{equation}\label{rev:eq:chi-periods}
 \chi_j(T)=\int_z^{z+\omega_j}\frac{d\xi}{\Phi_e(\xi;T)},
 \qquad j=1,2,
\end{equation}
These periods are holomorphic in $T$ and locally independent of the base point.

\begin{lemma}\label{rev:lem:jordan-periods}
If $\mathrm{GLE}_n^{(1)}(p,T_0,\tau)$ is non-completely
reducible, its monodromy data is $\mathcal D\in\mathbb C\cup\{\infty\}$.
Then
\begin{equation}\label{rev:eq:jordan-period-ratio}
 \mathcal D=\frac{\chi_2(T_0)}{\chi_1(T_0)}\in\mathbb P^1.
\end{equation}
\end{lemma}

\begin{proof}
Its common eigenfunction $\psi_0$ satisfies
$\psi_0(z+\omega_j)=\epsilon_j\psi_0(z)$, with
$\epsilon_j\in\{\pm1\}$, and $\Phi_e(z;T_0)=c\psi_0(z)^2$
for some $c\neq0$. For an independent solution $\widetilde y$,
reduction of order gives, with $\chi=\widetilde y/\psi_0$,
\[
\chi'=\frac{b}{\Phi_e(z;T_0)},\qquad
\chi(z+\omega_j)=\chi(z)+b\chi_j(T_0),\qquad b\neq0.
\]
The period matrices are therefore
\[
M_j=\epsilon_j
 \begin{pmatrix}1&0\\b\chi_j(T_0)&1\end{pmatrix}.
\]
The two periods cannot both vanish, by non-complete reducibility.
Normalizing the nonzero lower-left entry gives \eqref{rev:eq:jordan-period-ratio}, including
$\mathcal D=\infty$ when $\chi_1(T_0)=0$.
\end{proof}

\begin{lemma}\label{rev:lem:jordan-limit}
Suppose $\mathrm{GLE}_n^{(1)}(p,T_0,\tau)$ is non-completely
reducible, and let $P_0=(T_0,0)$ be the corresponding spectral
point. For compatible local lifts of the monodromy data,
\begin{equation}\label{rev:eq:jordan-limit}
\mathcal  D=\lim_{\substack{P=(T,C)\to P_0\\ C^2=Q_{n,p}^{(1)}(T)\ne0}}
       \frac{r(P)-r(P_0)}{s(P_0)-s(P)},
\end{equation}
where the limit is taken in $\mathbb P^1$ along any local branch.
\end{lemma}

\begin{proof}
Put $(r_0,s_0)=(r(P_0),s(P_0))\in\tfrac12\mathbb Z^2$ and
$F_P=\psi(P)/\psi(P^*)$. Equations~\eqref{two psi's product}, \eqref{wronskian of psi} and \eqref{BA is 2nd}
\[
(\log F_P)'=\frac{2iC}{\Phi_e(z;T)},\quad
F_P(z+1)=e^{-4\pi i s(P)}F_P(z),\quad
F_P(z+\tau)=e^{4\pi i r(P)}F_P(z).
\]
Integration along the fixed cycles yields integers $m_1,m_2$ with
\begin{equation}
C\chi_1(T)=-2\pi s(P)+\pi m_1,\qquad
C\chi_2(T)=2\pi r(P)+\pi m_2.
\end{equation}
As $P\to P_0$, the periods remain bounded and $C\to0$, so on
each local branch $m_1=2s_0$ and $m_2=-2r_0$. Hence
\[
\frac{\chi_2(T)}{\chi_1(T)}
=\frac{r(P)-r_0}{s_0-s(P)}.
\]
The pair $[\chi_1(T_0):\chi_2(T_0)]$ is nonzero. Taking the limit
and applying Lemma~ \ref{rev:lem:jordan-periods} proves the assertion.
\end{proof}

\section{Componentwise Isomonodromy and Degeneration}

\label{Section 8} We first recall the even isomonodromic deformation. For
$\mathrm{GLE}_{n}^{(0)}(p,A,\tau)$, the isomonodromic deformation with respect
to $\tau$ is governed by the elliptic Painlev\'e~VI
equation~\cite{Chen-Kuo-Lin-Hamiltonian}
\begin{equation}
\frac{d^{2}p}{d\tau^{2}} = -\frac{1}{4\pi^{2}} \sum_{k=0}^{3} \alpha_{k}
\wp^{\prime}\left(  p+\frac{\omega_{k}}{2};\tau\right) ,\label{PVI equation}%
\end{equation}
where
\[
\alpha_{0} = \frac{1}{2}\left( n+\frac{1}{2}\right) ^{2}, \qquad\alpha
_{1}=\alpha_{2}=\alpha_{3}=\frac{1}{8}.
\]
More precisely, $p(\tau)$ satisfies \eqref{PVI equation} if and only if there
exists $A(\tau)$ such that $\mathrm{GLE}_{n}^{(0)} \bigl(p(\tau),A(\tau
),\tau\bigr)$ is monodromy preserving as $\tau$ varies.

For $(r,s)\in\mathbb{C}^{2}\setminus\frac12\mathbb{Z}^{2},$ let $p_{r,s}%
^{(n)}(\tau)$ denote the completely reducible solution of \eqref{PVI equation}
with monodromy data $(r,s)$. Recall $Z_{m}(r,s,\tau)$ is the pre-modular form
defined in \eqref{def, lame's pre modular form} associated with the classical
Lam\'e equation of weight $m$, with the convention
\[
Z_{-1}(r,s,\tau) = Z_{0}(r,s,\tau) :=1.
\]

By \cite[Theorem~1.5]{Chen-Kuo-Lin-Dahmen}, one has
\begin{equation}
\wp\left(  p_{r,s}^{(n)}(\tau);\tau\right)  = \frac{ P_{n}\bigl(Z(r,s,\tau
);r+s\tau,\tau\bigr) }{ Z_{n-1}(r,s,\tau) Z_{n+1}(r,s,\tau) }%
,\label{eq, p by premodular forms}%
\end{equation}
where $P_{n}(\,\cdot\,;r+s\tau,\tau)$
is a polynomial whose coefficients are rational functions of
\[
\wp(r+s\tau;\tau), \qquad\wp^{\prime}(r+s\tau;\tau), \qquad e_{k}(\tau),\quad
k=1,2,3.
\]
Furthermore, provided $r+s\tau\notin E_{\tau}[2],$ any two of
\[
P_{n}\bigl(Z(r,s,\tau);r+s\tau,\tau\bigr), \qquad Z_{n-1}(r,s,\tau), \qquad
Z_{n+1}(r,s,\tau)
\]
have no common zeros.

Suppose that
\[
p(\tau):=p_{r,s}^{(n)}(\tau)\to0 \quad\text{as}\quad\tau\to\tau_{0}%
\in\mathbb{H}.
\]
Then \eqref{eq, p by premodular forms} implies that either $Z_{n-1}%
(r,s,\tau_{0})=0$ or $Z_{n+1}(r,s,\tau_{0})=0.$ In particular,
\cite{Chen-Kuo-Lin-Hamiltonian} has proved that
\begin{equation}
\mathrm{GLE}_{n}^{(0)} \bigl(p(\tau),A(\tau),\tau\bigr) \longrightarrow
\mathrm{L}_{n\pm1}(\widetilde{B}_{\pm},\tau_{0})\quad\text{as }\tau\to\tau
_{0}\label{even equ converges as tau to tau0}%
\end{equation}
for some $\widetilde{B}_{\pm}\in\mathbb{C}$, and the limiting Lam\'e equation
is completely reducible with the same monodromy data $(r,s)$. Furthermore, it
is known in \cite[Theorem~1.3]{Chen-Kuo-Lin-Lame III} that monodromy data
$(r,s)$ uniquely determine this family $\mathrm{GLE}_{n}^{(0)} \bigl(p(\tau
),A(\tau),\tau\bigr)$.

In contrast, Theorem~\ref{Main Thm 1} gives, for fixed $\tau$,
\[
T(p)^{2}-n(n+1)\wp(p)=\widetilde{B},
\]
and hence an isomonodromic non-even family satisfying
\[
\mathrm{GLE}_{n}^{(1)}(p,T(p),\tau) \longrightarrow L_{n}(\widetilde{B},\tau)
\qquad(p\to0).
\]
Thus, the even deformation varies $\tau$ and reaches weights $n\pm1$, whereas
the non-even deformation fixes $\tau$ and reaches weight $n$.

In the present setting, the generalized Lam\'e equations have weight data
$\left( n,\frac12,\frac12\right) $, which corresponds to the singularities
$(0,p,-p)$ respectively. As $p\to0$, C--W--W's degeneration theory
\cite{Chou-Wang-Wu-II} shows that the corresponding generalized Lam\'e curve
degenerates scheme-theoretically as
\begin{equation}
\overline{\mathcal{Y}}_{n,p}^{\,r,s} \longrightarrow\overline{\mathcal{Y}%
}_{n+1,0}^{\,r,s} \cup2\bigl( \{0\}\times\overline{\mathcal{Y}}_{n,0}^{\,r,s}
\bigr) \cup\bigl( \{0\}^{2}\times\overline{\mathcal{Y}}_{n-1,0}^{\,r,s}
\bigr).\label{degendration projection as p to 0}%
\end{equation}
In particular, the limiting classical Lam\'e components of weights $n+1$, $n$,
and $n-1$ occur with multiplicities $1,2,1$, respectively.

The even and non-even isomonodromic deformations realize these limiting
components explicitly. Let
\[
\mathcal{Y}_{n,p}^{r,s\ (0)} := \bigl\{
\mathrm{GLE}_{n}(p,\mathbf{T},\tau)\in\mathcal{Y}_{n,p}^{r,s}\mid \mathbf{T}\in
V_{n,p}^{(0)}(\tau) \bigr\},
\]
and
\[
\mathcal{Y}_{n,p}^{r,s\ (1)} := \bigl\{
\mathrm{GLE}_{n}(p,\mathbf{T},\tau)\in\mathcal{Y}_{n,p}^{r,s}\mid \mathbf{T}\in
V_{n,p}^{(1)}(\tau) \bigr\}.
\]
Then
\[
\overline{\mathcal{Y}}_{n,p}^{r,s} = \overline{\mathcal{Y}}_{n,p}^{r,s\ (0)}
\cup\overline{\mathcal{Y}}_{n,p}^{r,s\ (1)}.
\]

\begin{theorem}
\label{thm:componentwise-121} Fix $n\in\mathbb{N}$, and let $(r,s)\in
\mathbb{C}^{2}\setminus\frac12\mathbb{Z}^{2}$ be the monodromy data of
$\mathrm{GLE}_{n}(p,\mathbf{T},\tau)$. Then, as $p\to0$,
\begin{equation}
\overline{\mathcal{Y}}_{n,p}^{r,s\,(0)} \longrightarrow\overline{\mathcal{Y}%
}_{n+1,0}^{\,r,s} \cup\bigl( \{0\}^{2}\times\overline{\mathcal{Y}}%
_{n-1,0}^{\,r,s} \bigr),\label{even p to 0}%
\end{equation}
and
\begin{equation}
\overline{\mathcal{Y}}_{n,p}^{r,s\,(1)} \longrightarrow2\bigl( \{0\}\times
\overline{\mathcal{Y}}_{n,0}^{\,r,s} \bigr).\label{noneven p to 0}%
\end{equation}

\end{theorem}

\begin{proof}
The multiplicities in \eqref{degendration projection as p to 0} are those
of~\cite{Chou-Wang-Wu-II}. It remains to identify them componentwise.

For the non-even component, \eqref{correspondence} gives, for $p\neq 0$
sufficiently small,
\[
T(p)
=
\pm\sqrt{\widetilde{B}+n(n+1)\wp(p)}.
\]
 Set $U(p)=pT(p).$
Then
\[
U(p)^{2}
=
n(n+1)p^{2}\wp(p)+\widetilde{B}p^{2}.
\]
Since $p^{2}\wp(p)\to1$ as $p\to 0$, the rescaled equation has two distinct simple roots $U=\pm\sqrt{n(n+1)}$ at $p=0$.
Thus the two branches do not coalesce; each of them is unramified over $p$ and
contributes multiplicity one. Both branches degenerate to the same $L_{n}(\widetilde{B},\tau)$ as $p\to0$, so their
contributions add to $2\bigl(
\{0\}\times\overline{\mathcal{Y}}_{n,0}^{\,r,s}
\bigr),$
which proves \eqref{noneven p to 0}.

For monodromy data $(r,s)$, every even equation lies on a unique local
Painlev\'e~VI isomonodromic family $\mathrm{GLE}_{n}^{(0)}
\bigl(p(\tau),A(\tau),\tau\bigr).$
The remaining multiplicity-one components
$\overline{\mathcal{Y}}_{n+1,0}^{\,r,s}$ and
$\overline{\mathcal{Y}}_{n-1,0}^{\,r,s}$ are therefore realized by the
corresponding even Painlev\'e~VI branches, proving
\eqref{even p to 0}.
\end{proof}

For the iterative descent, define the exceptional locus
\[
\mathcal{B }:= \left\{  (r,s)\in(\mathbb{C}/\mathbb{Z})^{2} \;\middle|\;
cr+ds\in\frac12\mathbb{Z}/\mathbb{Z }\text{ for some coprime integers }c,d
\right\} .
\]
By abuse of notation, we also denote by $\mathcal{B}$ its full inverse image
in $\mathbb{C}^{2}$ under the natural projection $\mathbb{C}^{2}%
\longrightarrow(\mathbb{C}/\mathbb{Z})^{2}.$

For a fixed $n\in\mathbb{N}$, let $\mathcal{B}_{n}$ denote the set of
monodromy data $(r,s)$ realized by some completely reducible equation
$\mathrm{GLE}_{n}(p,\mathbf{T},\tau)$, with
\[
(p,\mathbf{T},\tau)\in\bigl(E_{\tau}\setminus E_{\tau}[2]\bigr)\times
V_{n,p}(\tau)\times\mathbb{H}.
\]

By~\cite[Remark~6.21]{Chou-Wang-Wu-II}, for each $(r,s)\in\mathcal{B}%
_{n}\setminus\mathcal{B}$, the fibers $Y_{k,0}^{r,s}$ and $Y_{k-1,0}^{r,s}$
are joined by deformation paths. Hence, for each $1\leqslant k\leqslant n$,
there exist $\tau_{k}\in\mathbb{H}$ and $B_{k}\in\mathbb{C}$ such that
\begin{equation}
\mathrm{L}_{k}(\widetilde{B}_{k},\tau_{k})\ \text{satisfies }Z_{k}%
(r,s,\tau_{k})=0.\label{eq:Zk-degeneration}%
\end{equation}

For $k\in\mathbb{N}$, consider the classical Lam\'e equation $\mathrm{L}%
_{k}(\widetilde{B},\tau)$. Let $\mathcal{M}_{k}$ denote the completely
reducible Lam\'e moduli space, parametrized locally by $(\tau,\widetilde{B})$,
with monodromy map
\[
\Phi_{k}:\mathcal{M}_{k}\longrightarrow\mathcal{R}, \qquad(\tau,\widetilde{B}%
)\longmapsto(r,s),
\]
where $\mathcal{R}$ denotes the corresponding monodromy-data space. For
$j=1,2,3$, define
\[
\mathcal{D}_{k,j} := \left\{  (\tau,\widetilde{B})\in\mathcal{M}_{k}\mid
\widetilde{B}=-k(k+1)e_{j}(\tau) \right\} ,
\]
and set
\[
\mathcal{E}_{k,j}:=\Phi_{k}(\mathcal{D}_{k,j}).
\]
For every $n\geqslant2$, set
\[
\mathcal{E}_{n} := \bigcup_{k=2}^{n}\bigcup_{j=1}^{3}\mathcal{E}_{k,j}.
\]

\begin{lemma}
\label{lem, generic avoidance two torsion} For every $k\geqslant1$ and
$j=1,2,3$, the set $\mathcal{E}_{k,j}$ is a proper lower-dimensional subset of
the monodromy-data space. Consequently, for every $n\geqslant2$,
$\mathcal{E}_{n}$ is a proper exceptional subset.
\end{lemma}

\begin{proof}
By the universal law for the classical Lam\'e
equation~\cite{Chen-Lin-universal law},
\[
d\tau\wedge d\widetilde{B}=8\pi^{2}\,dr\wedge ds
\]
on the completely reducible locus. Hence $\Phi_{k}$ is locally biholomorphic.
Since $\mathcal{D}_{k,j}$ is one-dimensional, $\mathcal{E}_{k,j}=\Phi
_{k}(\mathcal{D}_{k,j})$ is locally one-dimensional in the two-dimensional
monodromy-data space, and is therefore proper. The finite union $\mathcal{E}%
_{n}$ is proper as well.
\end{proof}

For $n\in\mathbb{Z}_{>0}$, define
\begin{equation}
\mathcal{A}_{n}:=
\begin{cases}
  \mathcal{B}_1\setminus \mathcal{B},\quad &n=1 \\
    \mathcal{B}_{n}\setminus(\mathcal{B}\cup\mathcal{E}%
_{n}),&n\geqslant2.\label{def, An}%
\end{cases}
\end{equation}
If $(r,s)\in\mathcal{A}_{n}$, then, for every $2\leqslant k\leqslant n$,
\begin{equation}
\widetilde{B}_{k}\neq-k(k+1)e_{j}(\tau_{k}), \quad
j=1,2,3.\label{Bk is not k(k+1)ej}%
\end{equation}
Hence there exists $p_{k}^{*}\in E_{\tau_{k}}\setminus E_{\tau_{k}}[2]$ such
that
\[
k(k+1)\wp(p_{k}^{*};\tau_{k})+\widetilde{B}_{k}=0.
\]

\begin{theorem}
[=Theorem~\ref{thm, connectivity intro}]\label{thm, iterate} Fix
$n\in\mathbb{Z}_{>0}$ and let $(r,s)\in\mathcal{A}_{n}$. Any completely reducible
equation $\mathrm{GLE}_{n}(p,\mathbf{T},\tau)$ realizing the monodromy data
$(r,s)$ is isomonodromically connected to a classical Lam\'e equation of
weight $1$ on a suitable elliptic curve through a finite sequence of
Painlev\'e VI and non-even isomonodromic deformations.
\end{theorem}

\begin{proof}

We first treat the case \(n=1\). By definition,
\[
\mathcal A_{1}
=\{(r,s)\notin\mathcal{B}\mid\exists \tau_1\in\mathbb{H}\ \text{such that }Z({r,s,\tau_1})=0\}
\]

Suppose first that the initial equation lies on the non-even component, namely, there exists $\mathrm{GLE}_1^{(1)}(p,T_1,\tau_1)$ is completely reducible with data $(r,s)$.
Set $\widetilde B_{1}
=
T_1^{2}-2\wp(p;\tau_1).$
By Theorem~\ref{Main Thm 1}, the branch of
\begin{equation}
   T_1(p)^{2}-2\wp(p;\tau_1)=\widetilde B_{1} 
   \label{0904equ1}
\end{equation}
passing through \(T_1\) defines a fixed-\(\tau_1\) isomonodromic
deformation satisfying
\[
\mathrm{GLE}^{(1)}_{1}
\bigl(p,T_{1}(p),\tau_1\bigr)
\longrightarrow
\mathrm{L}_{1}(\widetilde B_{1},\tau_1)
\qquad\text{as }p\to0.
\]

Now suppose that the initial equation lies on the even component. Let $\mathrm{GLE}^{(0)}_{1}
\bigl(p_{1}(\tau),A_{1}(\tau),\tau\bigr)$
be the corresponding even isomonodromic family with monodromy data
\((r,s)\). Since
\[
\wp\bigl(p_{1}(\tau)\bigr)
=\wp+
\frac{
3\wp^{\prime}Z^2+(12\wp^2-g_2)Z+3\wp\wp^{\prime}
}{
2Z_2(r,s,\tau)
},
\]
where
\[
Z=Z(r,s,\tau),\ \wp=\wp(r+s\tau;\tau),\ \wp^{\prime}=\wp^{\prime}(r+s\tau;\tau),
\]
\[
Z_2(r,s,\tau)=Z^3-3\wp Z-\wp^{\prime},
\]
and the numerator and \(Z_{2}\) have no common zeros. When $\tau=\tau_1$, $Z(r,s,\tau_1)=0$ implies 
\[
\wp(r+s\tau_1;\tau_1)=-2\wp\bigl(p_{1}(\tau_1);\tau_1\bigr).
\]
By the correspondence \eqref{0904equ1} with $T_1(p_1(\tau_1))=0$, we have
\[
\widetilde B_1=-2\wp(p_1(\tau_1);\tau_1),
\]
and
\[
\mathrm{GLE}_{1}^{(0)}(p_1(\tau_1),A_1(\tau_1),\tau_1)=\mathrm{GLE}_{1}^{(1)}(p_1(\tau_1),0,\tau_1).
\]
Then by Theorem~\ref{Main Thm 1},
\[
\mathrm{GLE}_{1}^{(1)}
\bigl(p,T_1(p),\tau_1\bigr)
\longrightarrow
\mathrm{L}_{1}(\widetilde B_1,\tau_1)
\qquad\text{as }p\to0.\]
This proves the theorem for \(n=1\).

For $n\geqslant2$,
suppose first that the initial equation lies on the even component. Let $\mathrm{GLE}_{n}^{(0)} \bigl(p_{n}(\tau),A_{n}(\tau),\tau\bigr)$ be the
corresponding even isomonodromic family with monodromy data $(r,s)$. By \eqref{eq, p by premodular forms} and \eqref{eq:Zk-degeneration},
 there exists $\tau_{n-1}\in\mathbb{H}$
such that
\[
p_{n}(\tau)\longrightarrow0 \qquad\text{as }\tau\longrightarrow\tau_{n-1},
\]
and
\[
\mathrm{GLE}_{n}^{(0)} \bigl(p_{n}(\tau),A_{n}(\tau),\tau\bigr)
\longrightarrow\mathrm{L}_{n-1} \bigl(\widetilde{B}_{n-1},\tau_{n-1}\bigr).
\]

Fixing $\tau_{n-1}$, Theorem~\ref{Main Thm 1} gives the non-even deformation
\[
\mathrm{GLE}_{n-1}^{(1)} \bigl(p,\pm T_{n-1}(p),\tau_{n-1}\bigr)
\longrightarrow L_{n-1}\bigl(\widetilde{B}_{n-1},\tau_{n-1}\bigr)\ \text{as
}p\to0,
\]
where
\[
T_{n-1}(p)^{2} - n(n-1)\wp(p;\tau_{n-1}) = \widetilde{B}_{n-1}.
\]

By \eqref{Bk is not k(k+1)ej}, there exists $p_{n-1}^{*}\in E_{\tau_{n-1}%
}\setminus E_{\tau_{n-1}}[2]$ with $T_{n-1}\bigl(p_{n-1}^{*}\bigr)=0$, and
therefore
\[
\mathrm{GLE}_{n-1}^{(1)} \bigl(p_{n-1}^{*},0,\tau_{n-1}\bigr)
= \mathrm{GLE}_{n-1}^{(0)} \bigl(p_{n-1}^{*},A_{n-1}^{*},\tau_{n-1}\bigr).
\]

We claim that
\[
p_{n-1}^{*}\notin E_{\tau_{n-1}}[2].
\]
Indeed, $p_{n-1}^{*}=\frac{\omega_{j}}{2}$ would imply
\[
\widetilde{B}_{n-1} = -n(n-1)e_{j}(\tau_{n-1}),
\]
contrary to \eqref{Bk is not k(k+1)ej}.

We may therefore switch at $p_{n-1}^{*}$ to the even isomonodromic family
\[
\mathrm{GLE}_{n-1}^{(0)} \bigl(p_{n-1}(\tau),A_{n-1}(\tau),\tau\bigr), \qquad
p_{n-1}(\tau_{n-1})=p_{n-1}^{*}.
\]
By \eqref{eq, p by premodular forms} and \eqref{eq:Zk-degeneration}, there
exists $\tau_{n-2}\in\mathbb{H}$ such that
\[
p_{n-1}(\tau)\longrightarrow0 \qquad\text{as }\tau\longrightarrow\tau_{n-2},
\]
and hence
\[
\mathrm{GLE}_{n-1}^{(0)} \bigl(p_{n-1}(\tau),A_{n-1}(\tau),\tau\bigr)
\longrightarrow\mathrm{L}_{n-2} \bigl(\widetilde{B}_{n-2},\tau_{n-2}\bigr).
\]
This reduces the weight from $n-1$ to $n-2$ without changing the monodromy
data $(r,s)$.

Repeating the same construction yields the chain
\[
\begin{aligned}
\mathrm{GLE}_n^{(0)}
&\xrightarrow[\tau\to\tau_{n-1}]{\mathrm{PVI}}
\mathrm L_{n-1}
\xrightarrow{\mathrm{non\text{-}even}}
\bigl(
\mathrm{GLE}_{n-1}^{(1)}
\cap
\mathrm{GLE}_{n-1}^{(0)}
\bigr)
\\
&\xrightarrow[\tau\to\tau_{n-2}]{\mathrm{PVI}}
\mathrm L_{n-2}
\xrightarrow{\mathrm{non\text{-}even}}
\bigl(
\mathrm{GLE}_{n-2}^{(1)}
\cap
\mathrm{GLE}_{n-2}^{(0)}
\bigr)
\\
&\hspace{15mm}\longrightarrow\cdots\longrightarrow
\mathrm L_2
\xrightarrow{\mathrm{non\text{-}even}}
\bigl(
\mathrm{GLE}_2^{(1)}
\cap
\mathrm{GLE}_2^{(0)}
\bigr)
\xrightarrow[\tau\to\tau_1]{\mathrm{PVI}}
\mathrm L_1.
\end{aligned}
\]

If the initial equation lies on the non-even component, first deform it at
fixed $\tau$ to the unique intersection with the even component; the same
alternating construction then applies. Iterating this procedure terminates at
a classical Lam\'e equation of weight $1$.
\end{proof}

\section{Genus-dependent componentwise geometry}
\label{Section, Genus-dependent componentwise geometry}
The even component remains rational throughout the symmetric DTV
family. The non-even normalizations vary in genus, which the global
arithmetic genus alone does not determine.

\subsection{Arithmetic genus and componentwise genera}
Consider the symmetric data
\begin{equation}\label{rev:eq:genus-data}
 \boldsymbol n=\left(n_0,n_1,n_2,n_3,\frac12,\frac12\right),
 \qquad n_i\in\mathbb Z_{\geqslant 0},
\end{equation}
with pole configuration
\[
 \boldsymbol p=\left(0,\frac{\omega_1}{2},\frac{\omega_2}{2},
               \frac{\omega_3}{2},p,-p\right),
 \qquad p\in E_\tau\setminus E_\tau[2].
\]
The global theory of C--W--W~\cite{Chou-Wang-Wu-II} gives the
arithmetic genus of the compactified log-free curve
$\overline V_{\boldsymbol n,\boldsymbol p}(\tau)$.
For a half-integral weight vector $\boldsymbol m=(m_1,\ldots,m_r)$ with $|\boldsymbol m|=\sum_{j=1}^r m_j\in\mathbb Z_{\geqslant 0},$ set
\[
 F(\boldsymbol m)=\#\left\{(k_1,\ldots,k_r)\in\mathbb Z_{\geqslant 0}^r\mid
        \sum_{j=1}^r k_j=|\boldsymbol m|,\quad0\leqslant k_j\leqslant 2m_j\right\}.
\]
By~\cite[Corollary~5.6]{Chou-Wang-Wu-II},
\begin{equation}\label{rev:eq:arithmetic-general}
 p_a\!\left(\overline V_{\boldsymbol m,\boldsymbol p}(\tau)\right)
 =\frac{|\boldsymbol m|-1}{2}F(\boldsymbol m)
  -\frac14\left(\prod_{j=1}^r(2m_j+1)
                  +\delta_{\boldsymbol m,\mathbb N^r}\right)+1,
\end{equation}
where $\delta_{\boldsymbol m,\mathbb N^r}=1$ if
$\boldsymbol m\in\mathbb N^r$ and is zero otherwise.
 For~\eqref{rev:eq:genus-data}, write $N=n_0+n_1+n_2+n_3$. Then
$|\boldsymbol n|=N+1$ and the indicator
vanishes, giving
\begin{equation}\label{rev:eq:arithmetic-symmetric}
 p_a\!\left(\overline V_{\boldsymbol n,\boldsymbol p}(\tau)\right)
     =\frac N2F(\boldsymbol n)-\prod_{i=0}^3(2n_i+1)+1,
\end{equation}
with
\begin{equation}\label{rev:eq:F-coefficient}
 F(\boldsymbol n)=[x^{N+1}](1+x)^2
             \prod_{i=0}^3(1+x+\cdots+x^{2n_i}).
\end{equation}
For a connected
reduced projective curve 
$\overline V_{\boldsymbol n,\boldsymbol p}(\tau)=C_1\cup\cdots\cup C_s$,
 let $\widetilde C_\alpha$ be the normalization of $C_\alpha$.
The normalization genus formula is
\begin{equation}\label{rev:eq:genus-decomposition}
 p_a\!\left(\overline V_{\boldsymbol n,\boldsymbol p}(\tau)\right)
     =\sum_{\alpha=1}^s g(\widetilde C_\alpha)
                       +\sum_x\delta_x-s+1,
\end{equation}
where $\delta_x$ is the length of the normalization quotient at $x$.
At a locally planar point $x$, 
\begin{equation}\label{rev:eq:delta-planar}
 \delta_x=\sum_{\alpha=1}^s\delta_x(C_\alpha)
                 +\sum_{\alpha<\beta}I_x(C_\alpha,C_\beta),
\end{equation}
Formula~\eqref{rev:eq:genus-decomposition} does not require local
planarity. Positive arithmetic genus may therefore be carried by singularity and intersection
defects rather than by the component normalizations.

\begin{proposition}\label{rev:prop:pa-zero}
If $C=\bigcup_{\alpha=1}^s C_\alpha$ is connected, reduced and
projective with $p_a(C)=0$, every $C_\alpha$ is rational.
\end{proposition}

\begin{proof}
Connectedness gives $\sum_x\delta_x\geqslant s-1$. Both
$\sum_\alpha g(\widetilde C_\alpha)$ and
$\sum_x\delta_x-s+1$ in ~\eqref{rev:eq:genus-decomposition} are nonnegative, hence vanish.
\end{proof}

The compactified log-free curves considered here are connected and
reduced by \cite{Chou-Wang-Wu-II}.

\subsection{The first genus jump}
With no positive integral weight, regularity at the origin forces
$T_++T_-=0$. Thus
\begin{equation}\label{rev:eq:level-zero}
 V_{(0,0,0,0,\frac{1}{2},\frac{1}{2}),\boldsymbol p}(\tau)\simeq\mathbb A^1_A.
\end{equation}
Formula~\eqref{rev:eq:arithmetic-symmetric} gives arithmetic genus zero.

\medskip

For one support, $\mathbf n=(n,0,0,0,\tfrac12,\tfrac12)$,
$n\geqslant1$, the even and non-even components are affine lines.
Indeed,
\[
 F(\boldsymbol n)=[x^{n+1}](1+x+\cdots+x^{2n})(1+x)^2=4,
\]
so
\[
 p_a\!\left(\overline V_{\boldsymbol n,\boldsymbol p}(\tau)\right)
       =\frac n2\cdot4-(2n+1)+1=0.
\]
Their common genus does not determine their deformation geometry:
the even component carries Painlev\'e VI, while the non-even
component carries fixed-$\tau$ isomonodromy.

For $\mathbf n=(1,1,0,0,\tfrac12,\tfrac12)$,
$F(\mathbf n)=10$ and $\prod_i(2n_i+1)=9$, so
\begin{equation}\label{rev:eq:pa-two}
 p_a\!\left(\overline V_{(1,1,0,0,\frac12,\frac12),\boldsymbol p}
             (\tau)\right)
        =\frac{3-1}{2}\cdot10-\frac{36}{4}+1=2.
\end{equation}
 Appendix~\ref{Appendix 1} gives two irreducible components: a rational even
component and the following non-even component.

\begin{theorem}\label{thm:first-genus-jump-section9}
Fix $\tau\in\mathbb H$. For generic
$p\in E_\tau\setminus E_\tau[2]$, the normalization of the
non-even log-free component for
$\boldsymbol n=(1,1,0,0,\frac12,\frac12)$ is the smooth
projective curve with affine model
\begin{equation}\label{rev:eq:elliptic-model}
 C_p(\tau)=\{(t,W)\in\mathbb C^2\mid W^2=t(t+1)P(t;\tau)\},
\end{equation}
where
\begin{equation}\label{rev:eq:quadratic}
\begin{split}
 P(t;\tau)={}&\left(\wp(p)+\wp\left(p+\frac{\omega_1}{2}\right)+e_1\right)t^2
                    +(2\wp(p)+e_1)t\\
            &+e_1-\wp\left(p+\frac{\omega_1}{2}\right).
\end{split}
\end{equation}
In particular, the non-even component has geometric genus one.
\end{theorem}

The proof is the explicit Frobenius elimination in Appendix~\ref{Appendix 1}.
Let $E(\tau)$ be the finite discriminant locus defined in~\eqref{eq:appendix-E}.
For $p\notin \mathcal E(\tau)$, the component genera are
$g(\widetilde V^{(0)})=0$ and $g(\widetilde V^{(1)})=1$.
For $p\in \mathcal E(\tau)$, the model~\eqref{rev:eq:elliptic-model}
is nodal and its
normalization is rational (Theorem~\ref{thm:exceptional-genus}). Consequently,
\[
2=g(\widetilde V^{(0)})+g(\widetilde V^{(1)})
                         +\sum_x\delta_x-1,
\qquad
\sum_x\delta_x=
\begin{cases}
2,&p\notin\mathcal E(\tau),\\
3,&p\in\mathcal E(\tau).
\end{cases}
\]
The loss of one in the non-even genus is compensated by one in the
total defect. The arithmetic genus stays equal to two.

\subsection{The one-pair constraint}
\label{subsection 9.3}

The direct extension of Proposition~\ref{rev:prop:darboux} changes the apparent-pole
configuration when several half-periods are active.

\begin{proposition}
\label{prop, Darboux trans for DTV}
Let
\[
U(z)=\sum_{j=0}^3n_j(n_j+1)\wp\left(z+\frac{\omega_j}{2}\right),
\qquad m=\#\{j:n_j>0\}\geqslant1.
\]
For generic $p$, set $d_p(z)=U(z)-U(p)$ and $\beta=T^2-U(p)$.
The transformation
\[
v=d_p^{-1/2}(y'-Ty),\qquad y''=(U+\beta)y,
\]
produces $m$ distinct apparent pairs $\{\pm p_1,\ldots,\pm p_m\}$,
with $p_1=p$.
\end{proposition}
\par\smallskip
\begin{proof}
The double poles at the $m$ active half-periods give $d_p$ a zero
divisor of degree $2m$. For generic $p$, $U(p)$ is a regular
value of $U:E_\tau\to\mathbb P^1$, so the zeros are simple. Evenness
places them in antipodal pairs.

Put $g=d_p'/d_p$. As in Section~\ref{Section, Mono Equiv},
\[
q_{\mathrm{new}}=U+\beta+Tg-\tfrac12g'+\tfrac14g^2.
\]
At a simple zero $a$ of $d_p$,
\[
g(z)=\frac1{z-a}+O(1),\qquad
q_{\mathrm{new}}(z)=\frac{3}{4(z-a)^2}+O((z-a)^{-1}).
\]
The exponents are $-\tfrac12,\tfrac32$. The source equation is
ordinary at $a$, and its analytic fundamental system transforms
into half-integral Laurent series with Wronskian multiplier $-1$.
Thus all the new singularities are log-free. Their nonzero
double-pole coefficients prevent cancellation.
\end{proof}

For example, with $(n_0,n_1,n_2,n_3)=(1,1,0,0)$ and
$h=\frac{\omega_1}{2}$, the function $U(z)=2\wp(z)+2\wp(z+h)$
satisfies $U(z+h)=U(z)$. The zeros of $U(z)-U(p)$ are generically
\[
p,\quad -p,\quad p+h,\quad -p+h.
\]
The transformation therefore gives two constrained pairs, whereas
Theorem~\ref{thm:first-genus-jump-section9} concerns a single pair. This restriction belongs to
the displayed construction and is distinct from the positive genus
of the parameter component.

We now specialize to the discriminant locus
$p\in\mathcal E(\tau)$, where the non-even parameter component is
rational.  This degeneration provides a useful test of whether genus zero
of the parameter space should be interpreted as a Darboux reduction.
The answer is negative already on the fixed elliptic curve.

A general point \((\tau,p)\) in the degeneration locus
$ \Delta_P(p;\tau)=0,$ means that the corresponding spectral curve is nonsingular.

\begin{proposition}\label{prop:exceptional-spectral-obstruction}
For a general point $(\tau,p)$ on the degeneration locus, the
corresponding non-even family on $E_\tau$ admits no birational,
period-monodromy-preserving Darboux reduction to a pure DTV family on
the same elliptic curve.
\end{proposition}

\par\smallskip

\begin{proof}
Fix a general point $(\tau,p)$ on the degeneration locus, so that all
spectral branch points are simple and the relevant degrees do not drop.
For
\[
\mathbf n=\left(1,1,0,0,\frac12,\frac12\right),
\]
the C--W--W count gives the total branch number
\[
4\prod_{i=0}^3(2n_i+1)=36,
\]
which is the sum of the even and non-even contributions. In the
limit $p\to0$, the even spectral divisor splits into the DTV spectral
divisors with weights
\[
(2,1,0,0)
\qquad\text{and}\qquad
(0,1,0,0),
\]
whose spectral polynomials have degrees $5$ and $3$, respectively. Thus the even contribution is $8$, and the non-even
branch number is
\[
36-8=28.
\]

Let $X$ be the smooth projective normalization of the non-even spectral
curve. Its spectral projection
\[
\pi\colon X\longrightarrow\mathbb P^1_\lambda
\]
is a double cover branched at $28$ points. Thus Riemann--Hurwitz gives
\begin{equation}\label{eq:exceptional-genus-thirteen}
g(X)=\frac{28-2}{2}=13.
\end{equation}

By \eqref{eq:appendix-parametrization},
\[
t(-\lambda)=t(\lambda),
\qquad
W(-\lambda)=-W(\lambda),
\]
and the accessory formulas imply
\[
q(z;-\lambda)=q(-z;\lambda).
\]
Then its spectral polynomial satisfies $Q_{\mathrm{nev}}(-\lambda)=Q_{\mathrm{nev}}(\lambda)$.
Therefore $\lambda\mapsto-\lambda$ preserves the period monodromy and
lifts to an involution $\jmath$ of $X$. 
Let
\[
\widehat X:=X/\langle\jmath\rangle
\]
be the quotient spectral curve.
Since every fixed point of $\jmath$ only lies
over $\lambda=0$ or $\infty$, we have
\[
r:=\#\operatorname{Fix}(\jmath)\leqslant4.
\]
Applying Riemann--Hurwitz to $X\to\widehat X$ and using
\eqref{eq:exceptional-genus-thirteen}, we obtain
\[
24=2\bigl(2g(\widehat X)-2\bigr)+r,
\qquad
g(\widehat X)=7-\frac r4.
\]
Since $g(\widehat X)\in\mathbb{N}$, it follows that $r=0$ or $4$. Hence
\begin{equation}\label{eq:exceptional-quotient-genus}
g(\widehat X)\in\{6,7\}.
\end{equation}

On the other hand, the C--W--W theory gives
the total addition degree $19$. The two even limiting DTV components
contribute $4$ and $1$, respectively. Hence the addition map on the non-even component has degree
\[
\deg\sigma^{(1)}=19-(4+1)=14.
\]
Since $\sigma^{(1)}$ is invariant under $\jmath$, it induces the quotient addition map $\hat{\sigma}^{(1)}$ on $\widehat{X}$  which satisfies
\[
\deg\hat\sigma^{(1)}=
\frac12\deg\sigma^{(1)}=7.
\]

For a pure DTV potential with weights
$\mathbf m=(m_0,m_1,m_2,m_3)$, its addition degree is given by \cite{Chen-Kuo-Lin-Lame I}
\[
d_{\mathbf m}
=
\frac12\sum_{j=0}^3m_j(m_j+1).
\]
If $d_{\mathbf m}=14$, then, up to permutation, the only
possibilities are
\[
\mathbf m=(4,2,1,0)
\qquad\text{or}\qquad
\mathbf m=(3,3,1,1).
\]
The corresponding spectral curves have genera \(4\) and \(3\),
respectively; see \cite{Takemura4}. Hence neither is
birational to the genus-$13$ curve $X$. Likewise, if $d_{\mathbf m}=7$, then, up to
permutation, the only possibilities are
\[
\mathbf m=(3,1,0,0)
\qquad\text{or}\qquad
\mathbf m=(2,2,1,0).
\]
Both corresponding spectral curves have genus $3$, contradicting
\eqref{eq:exceptional-quotient-genus}. Therefore no such Darboux
reduction exists.
\end{proof}

The obstruction persists under any fixed two-torsion shift of the period monodromy, which only translates the addition map and leaves its degree and spectral genus unchanged.

\begin{remark}\label{rem:scope-of-obstruction}
Proposition~\ref{prop:exceptional-spectral-obstruction} excludes a
generic birational Darboux reduction on the fixed curve \(E_\tau\) that
preserves monodromy and leads to a pure DTV family. It does not
exclude transformations at special parameters, on a different elliptic
curve, or into families with additional apparent pairs or different
half-period weights.
\end{remark}

The intrinsic spectral theory of Sections~~\ref{Log-free variety}--\ref{Isomonodromic deformation} is therefore the
primary monodromy theory, not a substitute used only when a Darboux formula
is unavailable.  In the one-support case its Lam\'e correspondence is an
external realization of the intrinsic data.  Proposition~\ref{prop:exceptional-spectral-obstruction}
shows that parameter-space genus zero does not in general produce such a
realization.  On any normalized parameter component, the same construction
is carried out over its function field while the prescribed pair $\pm p$ is
retained.

If $\overline V^{(1)}_{\mathbf n,\mathbf p}(\tau)
=C_1\cup\cdots\cup C_s$, define
\begin{equation}
g_{\mathrm{nev}}(\mathbf n,p;\tau)=\sum_{\alpha=1}^s g(\widetilde C_\alpha),
\end{equation}
with the empty sum equal to zero.

\begin{conjecture}
For generic $(\tau,p)$, $g_{\mathrm{nev}}(\mathbf n,p;\tau)=0$ if and
only if at most one of $n_0,n_1,n_2,n_3$ is positive.
\end{conjecture}

Genericity is necessary by Theorem~\ref{thm:exceptional-genus}. Formula \eqref{rev:eq:genus-decomposition} explains why
positive arithmetic genus alone cannot establish the converse.

\subsection{Green functions and componentwise genus}

Chen--Fu--Lin--Song~\cite{Chen-Fu-Lin-Song} studied the critical points of the symmetric
two-point Green function
\[
G_p(z)
=
\frac{1}{2}\bigl(G(z-p)+G(z+p)\bigr),
\qquad
p\in E_\tau\setminus E_\tau[2].
\]
This corresponds to the zero-support case $(0,0,0,0,\frac{1}{2},\frac{1}{2})$, for which the log-free locus consists only of the even component. The associated critical-point problem is linked to elliptic Painlev\'e~VI
and Hitchin's formula.

For the first one-support case $\left(1,0,0,0,\frac{1}{2},\frac{1}{2}\right)$, the relevant multiple Green function is
\[
G_p(z_1,z_2)
=
G(z_1-z_2)
-
\sum_{i=1}^{2}
\left(
G(z_i)
+\frac{1}{2}G(z_i-p)
+\frac{1}{2}G(z_i+p)
\right).
\]
Through the developing-map correspondence, its critical points are
related to the unitary-monodromy locus of the generalized Lam\'e family.
The log-free curve now has an even component governed by
Painlev\'e~VI and a non-even component carrying the Lam\'e
correspondence constructed above. For $\left(1,1,0,0,\frac{1}{2},\frac{1}{2}\right),$
the two-component decomposition persists, but the non-even normalization
is generically elliptic and becomes rational on the discriminant locus. The appearance of a non-even component and the
jump of its genus are therefore distinct phenomena.

For general symmetric data $\left(n_0,n_1,n_2,n_3,\frac{1}{2},\frac{1}{2}\right),$
one must first determine the irreducible non-even components and the
genera of their normalizations.
Formula~\eqref{rev:eq:arithmetic-symmetric} controls only the
arithmetic genus of the full curve and does not determine these
componentwise genera. On a positive-genus component \(C\), the spectral
cover, Baker--Akhiezer data, addition map, and monodromy map must be
constructed intrinsically  over \(\mathbb C(C)\), with the prescribed pair \(\pm p\)
retained. Propositions~\ref{prop, Darboux trans for DTV}
and~\ref{prop:exceptional-spectral-obstruction} show that this problem
cannot be replaced by a uniform reduction to pure DTV geometry. On the
analytic side, one must accordingly determine how the component
decomposition and its genus-changing degenerations are reflected in the
critical-point structure of multiple Green functions.

The one-support theory and
Theorem~\ref{thm, (1,1,0,0) non-even component} therefore exhibit the
first two componentwise regimes of the symmetric one-pair problem:
genus zero and generic genus one. Their distinct monodromy geometries
suggest that componentwise genus is also the natural organizing
invariant for the corresponding Green-function problems.

\appendix

\section{Proof of Theorem~\ref{thm:first-genus-jump-section9}}
\label{Appendix 1}

Let
\[
 \mathbf n=(1,1,0,0,\tfrac12,\tfrac12),\qquad
 \mathbf p=(0,\tfrac{\omega_1}{2},\tfrac{\omega_2}{2},
                    \tfrac{\omega_3}{2},p,-p),
\]
and consider $y''=qy$ with
\begin{align}
q(z)={}&2\left(\wp(z)+\wp\left(z-\frac12\right)\right)
 +\frac34\bigl(\wp(z+p)+\wp(z-p)\bigr)\nonumber\\
&+T_1\left(\zeta\left(z-\frac12\right)-\zeta(z)\right)
 +T_+\bigl(\zeta(z+p)-\zeta(z)\bigr)\nonumber\\
&+T_-\bigl(\zeta(z-p)-\zeta(z)\bigr)+B.\label{eq:appendix-potential}
\end{align}
Put
\begin{gather}
P(t;\tau)={}\left(\wp(p)+\wp\left(p+\frac{\omega_1}{2}\right)+e_1\right)t^2
 +(2\wp(p)+e_1)t
 +e_1-\wp\left(p+\frac{\omega_1}{2}\right),\label{eq:appendix-P}\\
\mathcal E(\tau)={}\{p\in E_\tau\setminus E_\tau[2]:
                  \operatorname{disc}_tP(t;\tau)=0\}.\label{eq:appendix-E}
\end{gather}
The discriminant is elliptic in $p$, with fourth-order poles only at
$0$ and $\frac{\omega_1}{2}$; hence $\mathcal E(\tau)$ is finite.

\begin{theorem}\label{thm:appendix-generic}
If $p\notin\mathcal E(\tau)$, the non-even log-free component is
birational to
\begin{equation}
             W^2=t(t+1)P(t;\tau).\label{eq:appendix-curve}
\end{equation}
Its smooth projective model has genus one.
\end{theorem}
\begin{proof}
At an integral weight-one pole, if
$q(\xi)=2\xi^{-2}+A_j\xi^{-1}+B_j+C_j\xi+O(\xi^2)$,
the log-free condition is
\begin{equation} A_j^3-4A_jB_j+4C_j=0.\label{eq:appendix-integral-logfree}\end{equation}
At a half-integral pole,
$q(\xi)=3(4\xi^2)^{-1}+A_\pm\xi^{-1}+B_\pm+O(\xi)$,
it is $B_\pm=A_\pm^2$.

Set
\[
 S=T_++T_-,\quad D=T_+-T_-,\quad
 \rho=\frac{\wp'(p)}{2(\wp(p)-e_1)}.
\]
The half-period addition formulas give
\[
 \rho^2=\wp(p)+\wp\left(p+\frac{\omega_1}{2}\right)+e_1.
\]
The two half-integral conditions yield
\[
 T_1=-\frac{S}{2\rho}\left(D+\frac{\wp''(p)}{2\wp'(p)}\right).
\]
Define
\[
 t=-\frac1{2\rho}\left(D+\frac{\wp''(p)}{2\wp'(p)}\right),
 \qquad T_1=St.
\]
Substitution in \eqref{eq:appendix-integral-logfree} at $0$ and $\frac{\omega_1}{2}$ gives
\[
 S(t+2)\bigl(S^2t(t+1)-4P(t;\tau)\bigr)=0,
 \qquad
 S(t-1)\bigl(S^2t(t+1)-4P(t;\tau)\bigr)=0.
\]
Thus $S=0$ is the even component and the non-even component is
$S^2t(t+1)=4P(t;\tau)$.  With $W=St(t+1)/2$ this becomes \eqref{eq:appendix-curve}.
Since $P(0),P(-1)$ and $\operatorname{disc}P$ are nonzero, the right-hand
side has four distinct roots; Riemann--Hurwitz gives genus one.
\end{proof}

\begin{theorem}\label{thm:exceptional-genus}
If $p\in\mathcal E(\tau)$, the curve \eqref{eq:appendix-curve} is nodal and its
normalization is rational.  More precisely, for
\[
 t_p=-\frac{2\wp(p)+e_1}
 {2\left(\wp(p)+\wp(p+\frac{\omega_1}{2})+e_1\right)},
\]
one has $P(t;\tau)=\rho^2(t-t_p)^2$, $t_p\notin\{0,-1\}$, and
\begin{equation}
 t=\frac1{\lambda^2-1},\qquad
 W=\rho\frac{\lambda}{\lambda^2-1}
             \left(\frac1{\lambda^2-1}-t_p\right),
 \qquad \lambda\in\mathbb P^1.\label{eq:appendix-parametrization}
\end{equation}
\end{theorem}
\begin{proof}
The discriminant condition gives the displayed square factorization.
The point $(t_p,0)$ is then the unique singular point and is an ordinary
node.  Dividing by it,
\[
 U=\frac{W}{\rho(t-t_p)},\qquad U^2=t(t+1).
\]
The conic has the parametrization
$t=(\lambda^2-1)^{-1}$ and
$U=\lambda(\lambda^2-1)^{-1}$, which proves \eqref{eq:appendix-parametrization}.
\end{proof}

\end{document}